\documentclass[reqno,11pt]{amsart}
\usepackage[utf8]{inputenc}
\DeclareUnicodeCharacter{0301}{}

\calclayout
\usepackage{amsthm,amsfonts,amstext,amssymb,mathrsfs,amsmath,latexsym,mathtools}
\usepackage{enumerate}
\usepackage{bm,accents}
\usepackage{mdwlist}
\usepackage[abs]{overpic}
\usepackage[mathscr]{euscript}
\usepackage{mathrsfs}
\DeclareMathAlphabet{\mathpzc}{OT1}{pzc}{m}{it}
\usepackage{parskip}
\usepackage{comment}

\usepackage{tikz}
\usetikzlibrary{shapes.geometric, arrows}

\usepackage[font=small]{caption}
\usepackage[labelformat=simple]{subcaption}

\usepackage{color}
\usepackage{esint} 
\usepackage[colorinlistoftodos,prependcaption,textsize=small]{todonotes}

\definecolor{red}{RGB}{255,0,0}
\definecolor{green}{RGB}{0,100,0}
\definecolor{blue}{RGB}{0,0,255}

\usepackage{cleveref}

\crefname{equation}{equation}{equations}
\crefname{figure}{Figure}{Figures}

\theoremstyle{plain}
\newtheorem{thm}{Theorem}[section]
 
\newtheorem{prop}[thm]{Proposition}
\newtheorem{lem}[thm]{Lemma}
\newtheorem{cor}[thm]{Corollary} 

\theoremstyle{definition}
\newtheorem{definition}[thm]{Definition}
\newtheorem{example}[thm]{Example}

\theoremstyle{remark}
\newtheorem{remark}[thm]{Remark}
\numberwithin{equation}{section}

\numberwithin{equation}{section}

\renewcommand{\Re}{\mathop{\rm Re}}

\renewcommand{\P}{\mathbb P}
\DeclareMathOperator{\supp}{supp}

\DeclareMathOperator{\const}{const}

\DeclareMathOperator{\diag}{diag}

\newcommand{\EE}{{\mathcal E}}

\newcommand{\N}{\mathbb{N}}
\newcommand{\C}{\mathbb{C}}
\newcommand{\R}{\mathbb{R}}
\newcommand{\Z}{\mathbb{Z}}

\newcommand{\os}[1]{\accentset{\circ}{#1}}

\newcommand\restr[2]{{
		\left.\kern-\nulldelimiterspace 
		#1 
		\vphantom{\big|} 
		\right|_{#2} 
}}

\title[Electrostatic partners and zeros]{Electrostatic partners  and zeros of orthogonal and multiple orthogonal polynomials}

\author[A. Mart\'{\i}nez-Finkelshtein]{Andrei Mart\'{\i}nez-Finkelshtein}

\address[AMF]{Department of Mathematics, Baylor University, Waco TX, USA, and Department of Mathematics, University of Almer\'{\i}a, Almer\'{\i}a, Spain}

\email{A\_Martinez-Finkelshtein@baylor.edu}

\author[R.~Orive]{Ram\'on Orive}

\address[RO]{Departamento de An\'alisis Matem\'atico, Universidad de La Laguna, Canary Islands, Spain}

\email{rorive@ull.edu.es}

\author[J.~S\'anchez-Lara]{Joaqu\'{\i}n S\'anchez-Lara}

\address[JSL]{Departamento de Matem\'atica Aplicada, Universidad de Granada, Spain}

\email{jslara@ugr.es}

\date{\today}

\keywords{Orthogonal polynomials; Multiple orthogonal polynomials; Zeros; Electrostatic model; Equilibrium; Linear differential equations}

\subjclass[2020]{Primary:  42C05; Secondary: 30C15; 31A15; 33C45; 33C47}

\begin{document}

\begin{abstract}
For a given polynomial $P$ with simple zeros, and a given semiclassical weight $w$, we present a  construction that yields a linear second-order differential equation (ODE), and in consequence, an electrostatic model for zeros of $P$. The coefficients of this ODE are written in terms of a dual polynomial that we call the electrostatic partner of $P$. This construction is absolutely general and can be carried out for any polynomial with simple zeros and any semiclassical weight on the complex plane. An additional assumption of quasi-orthogonality of $P$ with respect to $w$ allows us to give more precise bounds on the degree of the electrostatic partner. In the case of orthogonal and quasi-orthogonal polynomials, we recover some of the known results and generalize others. Additionally, for the Hermite--Pad\'e or multiple orthogonal polynomials of type II, this approach yields a system of linear second-order differential equations, from which we derive an electrostatic interpretation of their zeros in terms of a vector equilibrium. More detailed results are obtained in the special cases of Angelesco, Nikishin, and generalized Nikishin systems. We also discuss the discrete-to-continuous transition of these models in the asymptotic regime, as the number of zeros tends to infinity, into the known vector equilibrium problems. Finally, we discuss how the system of obtained second-order ODEs yields a third-order differential equation for these polynomials, well described in the literature. We finish the paper by presenting several illustrative examples. 
\end{abstract}

\maketitle

\section{Introduction}

Hermite polynomials\footnote{ \,   $[\cdot]$ stands for the integer part or the floor function.  }
\begin{equation} \label{hermiteH}
H_N(x)=N!\sum_{\ell=0}^{\left[  N/2\right] }\frac{(-1)^%
	{\ell}(2x)^{N-2\ell}}{\ell!\;(N-2\ell)!}=2^Nx^N+\dots
\end{equation}
 are probably the simplest representatives of the family of \textit{classical} orthogonal polynomials. 
They satisfy the linear differential equation
\begin{equation} \label{odeHermite}
y''(x) -2 x y'(x) +2N y(x)=0 
\end{equation}
and the orthogonality conditions
\begin{equation*}  
	\begin{split}
		&\int_{-\infty}^{+\infty} x^j H_N(x) e^{-x^2}\, dx=0, \qquad j=0,1,\dots, N-1,\\
		&\int_{-\infty}^{+\infty} x^N H_N(x) e^{-x^2}\, dx\neq 0.
	\end{split}
\end{equation*}
As a consequence, their zeros are all real and simple. A well-known calculation that goes back to Stieltjes \cite{Stieltjes1885} (see also  \cite[Theorem 6.8]{Szego75} or \cite{MR1379147}) shows that there are two equivalent physical interpretations of these zeros:
\begin{itemize}
\item as equilibrium positions of equally charged points on the plane in the presence of an external field (background potential); or
\item as an appropriately rescaled  configuration of vortices on the plane under assumption that they all  have same circulations and rotate as a rigid body.
\end{itemize}
We explain these notions in more detail in Section \ref{sec:vortices}. Both models are rooted in the linear second order differential equation \eqref{odeHermite} satisfied by these polynomials. There are several ways of obtaining this equation, all of them relying on a specific feature of  the orthogonality weight, namely the fact that its logarithmic derivative is a rational function. This idea allows to extend the classical theory   to the so-called \textit{semiclassical} orthogonal polynomials, a construction that probably goes back to J.~Shohat \cite{Shohat}, see also \cite{MR1340939, MR1637827}. This generalization preserves several convenient features of classical orthogonal polynomials, such as a Rodrigues-type formula or the existence of raising and lowering operators, see e.g.~\cite{Ismail2000} or \cite{Ismail05}; each one of these properties leads to \eqref{odeHermite}. 

The elegance of the above mentioned model attracted attention of generations of researchers and lead to several generalizations. For instance, we can choose as a starting point a second-order linear differential equation with polynomial coefficients (\textit{generalized Lam\'e equations} in algebraic form), whose particular cases are the hypergeometric and the Heun differential equation \cite{Ronveaux95}, and develop an electrostatic/vortex dynamics interpretation for the zeros of its polynomial solutions (\textit{Heine--Stieltjes polynomials}). This was carried out in the classical works of B\^ocher \cite{Bocher97}, Heine \cite{Heine1878} and Van Vleck \cite{Vleck1898}; for more modern treatment, see e.g.~\cite{MR4091604, Dimitrov:00, Grunbaum98, MR2345246}, as well as   \cite{MR2647571, MR2770010, MR2003j:33031} for the asymptotic results. 

Another approach starts from the orthogonality conditions with respect to a semiclassical (or even a more general) weight, as it was done in the pioneering work of Ismail   \cite{Ismail2000, Ismail2000c, Ismail2001}, which has been extended in many directions, see e.g.~ \cite{zbMATH07222111, MR826706, MR2873079, zbMATH06596279,  zbMATH06963441,  zbMATH07122712, zbMATH07326239},  to cite a few. One of such generalizations is the case of \textit{quasi-orthogonal polynomials}, which satisfy ``incomplete'' orthogonality conditions. As it was shown in \cite{MR3926161} in the simplest case of one condition short of full orthogonality, such polynomials also satisfy a linear second order differential equation that can be interpreted in electrostatic terms. 

\textit{Hermite-Pad\'e} or \textit{multiple} orthogonal polynomials (MOP) of type II are defined by distributing the orthogonality conditions among different weights or measures. In the simplest case of two  weights $w_1$, $w_2$, supported on $\Delta_1\subset \R$ and $\Delta_2\subset \R$, respectively, and for a given  multi-index $\bm n= (n_1,n_2)\in \Z_{\ge 0}^2$,  it is a polynomial $P_{\bm n}$ of total degree at most $N=|\bm n|:=n_1+n_2$, such that
\begin{equation} \label{typeIIintro}
	\int_{\Delta_i}x^jP_{\bm n}(x)w_i(x)dx  \begin{cases}
	=0,&\quad j\leq n_i-1,\\
	\neq 0,&\quad j=n_i,
\end{cases}
\qquad i=1, 2.
\end{equation}
Polynomial $P_{\bm n}$ appears as a common denominator of a pair of rational approximants to Markov functions related to the weights $w_i$'s, see Section~\ref{sec:generalHermitePade}. For a more detailed account of the corresponding theory we recommend the monograph \cite{niksor}, as well as the works of Aptekarev, Gonchar, Kuijlaars, Rakhmanov, Stahl, Suetin, Van Assche, Yattselev, and others, see e.g.~\cite{MR1702555, MR2475084, AptBraVA, Aptekarev:97, MR1240781, MR3602528, MR2187942, GRS97, VanAssche06, MR1808581, MR3304586, MR3907776, MR2796829, MR3058747, MR3137137}  (the list is far from complete). MOP find applications in number theory, numerical analysis, integrable systems, interacting particle systems and random matrix models \cite{MR2963452, ALT2011, MR2470930, MR2327035}. Although the general analytic theory of multiple orthogonal polynomials is in its infancy, their zeros (especially, their asymptotic behavior) have been studied in several particular situations, known as the Angelesco and Nikishin cases, described in Section~\ref{sec:special}. However, there is no known electrostatic interpretation of these zeros. 
Linear differential equations satisfied by multiple orthogonal polynomials have been found for many families of polynomials, see e.g.~\cite{AptBraVA, Aptekarev:97,  MR2214212, MR3055365}, but in all cases these are equations of order 3 and higher. The problem is that an electrostatic interpretation of the solutions to these ODE is not straightforward. 

Our main goal is to present a unified construction that yields an electrostatic model for polynomials related with a (system of) semiclassical weights. As it was mentioned, the known constructions of the differential equations for polynomials use either  a Rodrigues formula or  the so-called raising and lowering operators that can be combined into a single ODE \cite{MR2214212, MR3055365, Ismail05}. Instead of that, we  start from a construction that can associate to an arbitrary polynomial $P$ and to a semiclassical weight $w$ supported on a set $\Delta \subset \C$ another polynomial $S$,  that we have called its \textit{electrostatic partner}, see Definition~\ref{defCompanionP} and the schematic representation below.

\setlength{\unitlength}{1mm}
	\begin{center}
\begin{tikzpicture} 
	\draw[thick] (2,1) rectangle (5,2.8);
		\draw[->,thick] (1,1.4) -- (2,1.4);
			\put(6,15){$w$};
			\put(6,11){$\Delta$};
		\draw[->,thick] (1,2.4) -- (2,2.4);
			\put(6,23){$P$};
			\draw[->,thick] (5,1.9) -- (6,1.9);
				\put(61,18){$	S$};
\end{tikzpicture}
				\end{center}
			
Using $S$ and $w$ we can write a linear second order differential equation with polynomial coefficients whose solutions are $P$ and the corresponding function of the second kind $q$, defined in Section~\ref{sec:quasiorth}. This shows that the zeros of $P$ (assumed simple) are in an electrostatic equilibrium in an external field created by $w$ and by the attracting zeros of $S$ (understanding by this a stationary point of their energy, and not necessarily its local or global minimum).  This construction  uses only the semiclassical character of $w$; no orthogonality conditions on $P$ are required. An additional assumption that $P$ is quasi-orthogonal with respect to $w$ (in the complex case, we mean by that a non-hermitian orthogonality, see \eqref{orthog1}) allows us  to make more precise statements about the electrostatic partner of $P$. Moreover, two alternative representations for $S$ in this case yield a generalization of an identity involving Wronskian and Casorati determinants of $P$ and $q$, known in the case of classical orthogonal polynomials, see Section~\ref{sec:quasi}. 
			
Since the definition of type II Hermite-Pad\'e orthogonal polynomials \eqref{typeIIintro} boils down to two simultaneous quasi-orthogonality conditions,  we can associate with the corresponding MOP $P_{\bm n}$ \textit{two} electrostatic partners, $S_{\bm n, 1}$ and $S_{\bm n, 2}$, and a system of two linear differential equations of order 2, whose solution is  $P_{\bm n}$. This apparent redundancy can be used to find an electrostatic model for the zeros of $P_{\bm n}$. Namely, by a procedure similar to the definition of an electrostatic partner, we associate with $P_{\bm n}$, $w_1$ and $w_2$ a polynomial $R_{\bm n}$:
			
	\begin{center}
	\begin{tikzpicture} 
		\draw[thick] (2,1) rectangle (5,2.8);
		\draw[->,thick] (1,1.2) -- (2,1.2);
			\draw[->,thick] (1,1.9) -- (2,1.9);
		\put(4,13){$w_2$};
		\put(4,9){$\Delta_2$};
		\put(4,21){$w_1$};
		\put(4,17){$\Delta_1$};
		\draw[->,thick] (1,2.6) -- (2,2.6);
		\put(4,25){$P_{\bm n}$};
		\draw[->,thick] (5,1.9) -- (6,1.9);
		\put(61,18){$	R_{\bm n}$};
	\end{tikzpicture}
\end{center}

With this construction, the zeros of $P_{\bm n}$ and the zeros of $S_{\bm n, 1}$  (or $S_{\bm n, 2}$) are in a vector equilibrium given by their mutual interaction and by the vector external field created by $w_1$, $w_2$ and the zeros of $R_{\bm n}$, see Section~\ref{sec:multiple_orth} for details. This model is especially convenient because it is known that the asymptotic distribution of the zeros of $P_{\bm n}$ is usually described by vector equilibria. We discuss this connection and provide some heuristic arguments for this discrete-to-continuous transition in Section~\ref{sec:asymptotics}, where several particular configurations are analyzed in detail. 
So far, both two- and three-component  critical vector measures have been used to describe asymptotics in several cases. Our construction suggests that there is one universal two-component vector equilibrium valid for all known configurations corresponding to perfect systems, and that all descriptions constitute just its particular manifestations. 

In order to establish another connection with previous literature, we describe in Section~\ref{sec:ODE3} how to combine the system of ODEs from Section~\ref{sec:multiple_orth} into a third order linear differential equation whose solutions are  $P_{\bm n}$ and the corresponding functions of the second kind. An additional advantage of this construction is that it is possibly generalized to the case of more than 2 weights and explains the appearance of higher order ODEs (see Remark~\ref{remark:higherorder} in Section~\ref{sec:ODE3}).

In the last section we discuss  several examples of multiple orthogonal polynomials well known in the literature.

We hope that this approach can be applied in some other contexts; in particular, it would be interesting to explore a possible electrostatic interpretation of the zeros of the Type I multiple orthogonal polynomials, see e.g.~\cite{niksor}.

Since this paper unfortunately contains a large amount of technical details and auxiliary results (some of them, relegated to Appendices~\ref{appendixA} and \ref{sec:realcase}), we finish this introduction with a short navigation guide for the reader interested in the main highlights:
\begin{itemize}
	\item An electrostatic partner $S$ of a given polynomial $P$ (in a sense, the starting fundamental construction) is introduced in Definition~\ref{defCompanionP}, whose consistency is justified by Theorem~\ref{lem1aux}.
	\item The second order linear differential equations whose solution is $P$ is introduced in Theorem~\ref{CorolarioEDO}, which leads to an electrostatic model (Proposition~\ref{cor:criticalGeneral}) for zeros of $P$, which are shown to be in equilibrium in a field created in part by the attracting zeros of the electrostatic partner $S$. Additional properties of $S$ under assumptions that $P$ satisfies some orthogonality conditions are discussed in Section \ref{sec:quasi}.
	\item This construction is extended to a type II Hermite-Pad\'e orthogonal polynomial with respect to two weights, giving us now two second order linear differential equations (Theorem~\ref{TeoClaveHP}). Additionally, we get another set of differential equations, now for the electrostatic partners (Theorem~\ref{thm:WronskianS/w}, which is based on a construction from Proposition~\ref{prop:WronskianS/w}). 
	\item As a consequence, we derive a vector equilibrium model for the two sets of point charges: the zeros of the Hermite-Pad\'e orthogonal polynomial and the zeros of its electrostatic partner(s), Theorem~\ref{thm:vectorelectrostatics}. 
	\item More precise results about the location of the zeros of the electrostatic partners in some widely studied cases of Hermite-Pad\'e orthogonal polynomials are matter of Section~\ref{sec:special}. They allow us to discuss in Section~\ref{sec:asymptotics} the discrete-to-continuous transition in the equilibrium model as the degrees tend to $\infty$, and to compare the resulting models with the description of the asymptotic distribution of zeros in terms of the vector equilibrium, both with 2 and 3 components. In particular, Corollary~\ref{cor:asymptoticsGeneral} suggests that the universal description can be achieved using a 2-component vector equilibrium with the interaction matrix
	$$ 
	\begin{pmatrix}1&-1/2\\-1/2&1\end{pmatrix}.
	$$ 
	\item Since third order  linear differential equations associated to the multiple orthogonal polynomials are well known in the literature, we have included in Section~\ref{sec:ODE3} their derivation from the system of second order ODEs described in Theorem~\ref{TeoClaveHP}.
	\item Last, but definitely not least, we have a set of curious examples in Section~\ref{sec:examples}, whose examination poses several interesting questions and suggests possible lines of further research.	
\end{itemize}

\section{Electrostatics of point charges and vortex dynamics} \label{sec:vortices}

\subsection{Identical point charges and vortices}
\ 

We can associate with $N$ pairwise distinct points $\zeta_i$ on the plane ($\zeta_i \neq \zeta_j $  for $i\neq j$) their discrete  ``counting'' measure 
\begin{equation}\label{defMucritDiscrete}
	\mu=\sum_{k=1}^N  \delta_{\zeta_k} ,  
\end{equation}
where $\delta_x$ is a unit mass (Dirac delta) at $x$, and define its (discrete) logarithmic energy\footnote{Actually, the magnitude in \eqref{EnergyDiscrt} is twice the logarithmic energy, which is not relevant, but explains the factor $2$ in \eqref{defWeightedEnergy} introduced for consistency.}
\begin{equation} \label{EnergyDiscrt}
	\mathcal E(\mu)   :=
	\sum_{i\neq j}      \log \frac{1}{|\zeta _i-\zeta _j|}
\end{equation}
(we can extend the notion of the energy to the case when two or more $\zeta_j$'s coincide by assuming that then $\mathcal E(\mu)=+\infty$).
Additionally, given a real-valued function (\emph{external field} or \textit{background potential})  $\varphi $, finite at $\supp(\mu)$, we consider the weighted energy
\begin{equation}\label{defWeightedEnergy}
	\mathcal E_\varphi(\mu) := \mathcal E (\mu)+ 2 \sum_{k=1}^N \varphi(\zeta _k)  \,.
\end{equation}
For our purposes, it will be sufficient to assume that $\varphi=\Re \Phi$, where $\Phi$ is an analytic (in general, multivalued) function in  $ \C$, excluding its (finite number) of isolated singularities and branch points, with a single-valued derivative $\Phi'$.

\begin{definition}[\cite{MR2770010}] \label{defCriticalScalar}
	We say that $\mu$ in \eqref{defMucritDiscrete} is \emph{$\varphi$-critical} or just \emph{critical measure}  if $\supp(\mu)$ is disjoint with the set of singularities of $\varphi$, and is a stationary point of the weighted discrete energy $\EE_\varphi(\mu)=\EE_\varphi(\zeta _1, \dots, \zeta _N)$ defined in \eqref{defWeightedEnergy} :
	\begin{equation}\label{gradient}
		\nabla \, \EE_\varphi(\zeta_1, \dots, \zeta_N)=0\,,
	\end{equation}
	or equivalently,
	$$
	\frac{\partial}{\partial z} \, \EE_\varphi(\zeta_1, \dots, z, \dots \zeta_N)\big|_{z=\zeta_k} =0, \quad k = 1, \dots, N, \qquad \frac{\partial}{\partial z} = \frac{1}{2}\, \left(\frac{\partial}{\partial x} - i \frac{\partial}{\partial y}\right).
	$$
\end{definition}
We also say that the configuration of points (or charges) is in \textit{electrostatic equilibrium} in the external field $\varphi$. 
Notice that with $\varphi=\Re \Phi$, we can write explicitly the equilibrium conditions for $\EE_\varphi(\zeta_1, \dots, \zeta_N)$ as the system of equations
\begin{equation} \label{eq:condEq}
	\sum_{\substack{i, j=1 \\
			i\neq j  }}^N 
		\frac{1}{\zeta_{j}-\zeta_{i}}- \Phi^{\prime}\left(\zeta_{j}\right) =0,  \quad j=1, \dots, N.
\end{equation}
Let
\begin{equation} \label{eq:defPOly}
y(z):=\prod_{j=1}^N (z-\zeta_j);
\end{equation}
a common terminology is that $\mu$ in \eqref{defMucritDiscrete} is the \emph{zero counting measure} of the polynomial $y$, for which we will use the notation 
\begin{equation} \label{defCountingMeaure}
		\nu(y): =\sum_{j=1}^n   \delta_{\zeta_j} .
	\end{equation}
It is easy to check that  
\begin{equation} \label{eq:polyidentities}
	y'(z)=y(z) \sum_{j=1}^N \frac{1}{z-\zeta_j}, \qquad y''(z)= y(z) \sum_{\substack{i, j=1 \\
			 i\neq j  }}^N \frac{1}{(z-\zeta_i)(z-\zeta_j)},
\end{equation}
from where 
$$
\sum_{\substack{i, j=1 \\
		i\neq j  }}^N 
	\frac{1}{\zeta_{j}-\zeta_{i}}=\frac{y''}{2y'}(\zeta_j).
$$
In particular,  \eqref{eq:condEq} is equivalent to
\begin{equation} \label{eq:condEqBis}
	\left( y''- 2\Phi^{\prime} y'\right) \left(\zeta_{j}\right) =0,  \quad j=1, \dots, N.
\end{equation}

In the case when all zeros $\zeta_j$'s are on the real line and the external field is given by $\varphi(x)=x^2/2$, we get from \eqref{eq:condEqBis} that $y''-2 x y'$ matches $y$ up to a multiplicative constant. Comparing the leading coefficients we conclude that $y$ solves the differential equation \eqref{odeHermite}; in other words, the zeros of Hermite polynomials are in electrostatic equilibrium on $\mathbb R$ in the external field $\varphi(x)=x^2/2$, as observed by Stieltjes in \cite{Stieltjes1885}
\footnote{As it is pointed out in \cite{MR1379147}, Stieltjes mentions without proving that the equilibrium configuration is actually the minimum of the energy. The proof can be found for instance in \cite[Section 6.7]{Szego75}.}.  He also realized that this electrostatic model is easily generalized to all classical families of polynomials (Jacobi, Laguerre and Bessel), see  \cite{Stieltjes1885b, MR1554668}, or    \cite{Szego75} and \cite{Ismail05} for a more modern account. 

Another approach to zeros of these polynomials is via vortex dynamics. The notion of a point vortex is a classical approximation in ideal hydrodynamics of planar flow,  introduced almost 150 years ago in Helmholtz's classical paper on vortex dynamics \cite{Helmholtz1858}.  Considering the flow plane to be the complex plane, the equations of motion for $N$ point vortices with circulations $\gamma_{i}\in \R$ at positions $\zeta_i$, $i=1, \dots, N$,  in a background flow $\Psi $, is 
\begin{equation} \label{vortices1}
\overline{	\left(\frac{d \zeta_{j}}{d t} \right)}=\frac{1}{2 \pi i} \sum_{\substack{i=1 \\
			i\neq j  }}^{N} \frac{\gamma_{i}}{\zeta_{j}-\zeta_{i}}  + \frac{1}{2\pi i}\, \overline{	\Psi(\zeta_j) }, \quad j=1,2, \ldots, N.
\end{equation}
In this paper,  the overline   indicates complex conjugation.

By \eqref{eq:condEq}, stationary vortices $(d\zeta_j/dt=0$) correspond to electrostatic equilibrium in the external field $\varphi=\Re \Phi$ if the background flow $\Psi$ satisfies $\overline{\Psi (\zeta_j)}=-\Phi'(\zeta_j)$, $j=1, \dots, N$. 

Alternatively, if a vortex configuration rotates as a rigid body with angular velocity $\Omega$, then $\overline{d\zeta_j/dt}$ is equal to $\overline{\zeta_j}$ times a purely imaginary  constant  proportional to the angular velocity.
If we assume additionally that all $\zeta_j$'s are real and identical (all $\gamma_i$'s are equal) and there is no external flow field, then after rescaling \eqref{vortices1} boils down to 
\begin{equation} \label{vortices1bis}
	  \zeta_{j} = \sum_{\substack{i=1 \\
			i\neq j  }}^{N} \frac{1}{\zeta_{j}-\zeta_{i}}  , \quad j=1,2, \ldots, N.
\end{equation}
Let us use again the polynomial $y$ defined in \eqref{eq:defPOly}, known in this field as the \textit{generating polynomial} for the vortex configuration (see \cite{MR1831715}). We can rewrite the second identity in \eqref{eq:polyidentities} equivalently as
\begin{equation} \label{eq:polyidentitiesBis}
  y''(z)= -2 y(z) \sum_{j=1}^N \sum_{\substack{i=1 \\
			i\neq j  }}^N \frac{1}{(\zeta_i-\zeta_j)(z-\zeta_j)},
\end{equation}
which together with \eqref{vortices1bis} yields again the differential equation \eqref{odeHermite}. Thus,   the zeros of the Hermite polynomials give us the positions  of  vortices on $\R$ such that the configuration rotates like a rigid body. Clearly, these considerations can be extended to more general families of polynomials. 

A  reader interested in vortex dynamics should check the nice surveys \cite{MR2305271} and \cite{MR2538285}.

\subsection{Groups of point charges and vortices }
\

We can extend Definition~\ref{defCriticalScalar} to a vector setting (for our purpose, it will be sufficient  to consider two-component vectors) that allows us to handle groups of differently charged particles. Given a vector of discrete measures  $\vec \mu=(\mu_1, \mu_2)$, with  $\supp (\mu_1) \cap \supp(\mu_2)=\emptyset$, 
\begin{equation} \label{defMuJ}
	\mu_1=\sum_{k=1}^{n_1}\delta_{\zeta_k} ,\qquad \mu_2=\sum_{j=1}^{n_2}\delta_{\xi_j},
\end{equation}
a real \textit{interaction parameter} $-1<a<1$,
and a vector external field $\vec \varphi=(\varphi_1, \varphi_2)$, both $\varphi_i$ real-valued and finite at $\supp(\mu_1)\cup \supp(\mu_2)$, 
the corresponding \textit{weighted vector energy} is
\begin{equation}\label{def:EnergiaVect}
	\mathcal{E}_{\vec \varphi, a}(\vec \mu  ):=
	\mathcal{E}(\mu_1)+2a\sum_{k=1}^{n_1}\sum_{j=1}^{n_2}\log\frac{1}{|\zeta_k-\xi_j|}
	+\mathcal{E}(\mu_2)+2\sum_{k=1}^{n_1}\varphi_1(\zeta_k)+2\sum_{j=1}^{n_2}\varphi_2(\xi_j)
\end{equation}
(see the notation in \eqref{EnergyDiscrt}). 
We can restate this definition using vector notation and the  symmetric positive-definite matrix
$$M:=\begin{pmatrix}1&a\\a&1\end{pmatrix} $$
and say that the weighted vector energy in \eqref{def:EnergiaVect} corresponds to the \textit{interaction matrix} $M$.
Moreover, for measures  \eqref{defMuJ} we can write alternatively
$$
\mathcal{E}_{\vec \varphi, a}(\zeta_1,\dots,\zeta_{n_1},\xi_1,\dots,\xi_{n_2} ):=\mathcal{E}_{\vec \varphi, a}(\vec \mu  ).
$$

\begin{definition} \label{defCriticalVector}
	We say that $\vec \mu$  is a \textit{critical vector measure} for $\mathcal{E}_{\vec \varphi, a }$ if $\supp(\mu_1)\cup \supp(\mu_2)$ is a stationary configuration for  $\mathcal{E}_{\vec \varphi, a }$:
	$$
	\nabla \, \mathcal{E}_{\vec \varphi, a }(\zeta_1,\dots,\zeta_{n_1},\xi_1,\dots,\xi_{n_2})=0.
	$$
\end{definition}

For any Borel measure $\mu $ on $\C$ we can define its \textit{logarithmic potential},
\begin{equation} \label{defLogPot}
	U^\mu(z)  :=  \int \log\frac{1}{|z-t|}\, d\mu(t).
\end{equation}
From the expression for $\mathcal{E}_{\vec \varphi, a}$ it follows that Definition \ref{defCriticalVector} is equivalent to simultaneous equilibrium conditions
\begin{equation} \label{VectorequilConditions}
	\begin{split}
		\mu_1 \text{ is $F_1$-critical, with } F_1 & := a\,  U^{\mu_2} + \varphi_1, \\ 
		\mu_2 \text{ is $F_2$-critical, with } F_2 & := a\,  U^{\mu_1} + \varphi_2.
	\end{split}
\end{equation}

Alternatively, consider the situation when in the absence of the background flow, the circulations $\gamma_i$'s in \eqref{vortices1} take only two possible values,  
$$
\gamma_{k}= \begin{cases}\gamma>0, & \text { for } \quad k=1,2, \ldots, n_1, 
	\\ -\gamma, & \text { for } \quad k=n_1+1, n_1+2, \ldots, n_1+n_2=N
\end{cases}
$$
We can rename $\xi_k:=\zeta_{n_1+k}$,  $k=1,\dots, n_2$; in this way, for stationary vortices  we have the equations 
\begin{equation} \label{groupvortices1}
	\begin{split}
	  \sum_{\substack{i=1 \\
			i\neq j  }}^{n_1} \frac{1}{\zeta_{j}-\zeta_{i}} & =  \sum_{k=1 }^{n_2} \frac{1}{\zeta_{j}-\xi_{k}} , \quad j=1,2, \ldots, n_1,  \\
		 \sum_{\substack{i=1 \\
				i\neq j  }}^{n_2} \frac{1}{\xi_{j}-\xi_{i}} & =  \sum_{k=1 }^{n_1} \frac{1}{\xi_{j}-\zeta_{k}}   , \quad j=1,2, \ldots, n_2.
	\end{split}
\end{equation}
To study these vortex patterns, we define again the generating polynomials
$$
y(z)=\prod_{j=1}^{n_1}\left(z-\zeta_{j}\right), \quad v(z)=\prod_{k=1}^{n_2}\left(z-\xi_{k}\right).
$$
Formulas  \eqref{eq:polyidentities} and \eqref{eq:polyidentitiesBis} show that \eqref{groupvortices1} yield the bilinear  identity
\begin{equation}\label{eq:Tkach}
	y'' v - 2 y' v' + y v'' =0. 
\end{equation}
This is currently known as Tkachenko's equation, since it was first derived by Tkachenko in his dissertation in 1964. Polynomial solutions of this equation were studied by Burchnall and Chaundy  \cite{MR1576413}. Adler and Moser \cite{MR501106} showed that \eqref{eq:Tkach} is solved by two consecutive polynomials that nowadays are know as Adler-Moser polynomials, see also \cite{MR2924501}. Moreover, comparing \eqref{groupvortices1} with \eqref{eq:condEq} and \eqref{VectorequilConditions} we conclude that the zeros of  consecutive Adler-Moser polynomials are stationary configurations (or equivalently,  $\vec \mu=(\mu_1,  \mu_2)$ defined in  \eqref{defMuJ} is a critical vector measure) for the vector energy $\mathcal{E}_{\vec \varphi, a}(\vec \mu )$ defined in  \eqref{def:EnergiaVect}, with 
$$
\vec \varphi\equiv (0,0) \quad \text{and} \quad a=-\frac{1}{2},
$$
fact that was already observed in \cite{MR824780}.

Further generalizations of these ideas, and in particular, their relation to rational solutions of Painlev\'e equations, can be found in  \cite{MR2305271, MR3335711, MR2538285, MR2402236, MR2989511,MR3508236, MR2530570, MR2336164, Loutsenko:2004, MR1831715}, to cite a few.

\section{Electrostatics for semiclassical weight}\label{sec:quasiorth}

In the rest of the paper we will try to adhere to the following notational convention, whenever possible: we will use capital letters to denote polynomials, and small letters to indicate general, usually multivalued, functions. A few exceptions of these rules will be clearly indicated.

\subsection{The semiclassical weight} \label{sec:generalcase}
\

We  start from a monic polynomial
$$
A(z)=z^{\deg(A)} + \text{lower degree terms;}
$$
we use the notation $\mathbb A$ for the set of zeros of $A$ on $\C$ and admit $\mathbb A=\emptyset$. Let
$$
\Delta := \Gamma_1 \cup \dots \cup \Gamma_k,
$$
where each $\Gamma_j$, $j=1,\dots, k$, is an oriented Jordan piece-wise analytic arc joining pairs of points from $\mathbb A\cup \{\infty\}$  and not containing any other point from $\mathbb A$ in its interior  $\os \Gamma_j$. For simplicity, we assume that the interiors of $\Gamma_j$'s are pairwise disjoint ($\os \Gamma_i\cap \os \Gamma_j=\emptyset$, $i\neq j$), but as it will be clear from what follows, this is not an actual restriction. 

Given another polynomial, $B$, 
we define, up to a normalization constant,  (multivalued) analytic functions  in $\C$,
\begin{equation} \label{defW}
w(z):=\exp \left( \int ^z \frac{B(t)}{ A(t)}\, dt \right), \quad v(z):= A(z) w(z),
\end{equation}
with only possible singularities (either isolated or branch points) at $\mathbb A\cup \{\infty\}$. Notice that a priori we do not assume that $A$ and $B$ are relatively prime. This implies in particular, that $B/A$ can be analytic at some zero of $A$, so that not all end-points of the arcs $\Gamma_j$'s are necessarily branch points, zeros or isolated singularities of $w$. 

The orientation of the arcs $\Gamma_j$ in $\Delta$ defines the left- and right-side boundary values of the function $w$ on $\Gamma_j$ that we denote by $w_+$ and $w_-$, respectively; two different values of $w_+$, as well as $w_+$ and $w_-$, differ by a multiplicative constant. We  fix on $\Delta$ the \textit{weight} by assuming that on each component  $\Gamma_j$ it coincides, up to a non-zero multiplicative constant, with $w_+$; for the sake of simplicity of notation, we will be denoting this weight by the same letter $w$. A fundamental assumption is that such a weight 
 has finite moments:
$$
\int_\Delta |z|^m |w(z)| |dz|<+\infty, \quad m=0, 1, 2, \dots
$$
As consequence, for $m=0, 1, 2, \dots$, 
\begin{equation}\label{endpoints}
z^m	v(z)=0 \text{ at every endpoint of a subarc of } \Delta.
\end{equation}
Clearly, $w$ is piece-wise differentiable on $\Delta$ and
\begin{equation} \label{ratW}
\frac{w'(z)}{w(z)}= \frac{B(z)}{ A(z)}, \quad z\in \Delta;
\end{equation}
this equality is actually valid in $\C\setminus \mathbb A$. 

The weight  $w$ is known as \textit{semiclassical}, and the value
\begin{equation} \label{def:class}
\sigma := \max \{\deg(A)-2, \deg(B)-1 \}
\end{equation}
is often referred to as its \textit{class}, see e.g.~\cite{MR1340939,MR1637827}. This is ambiguous in the case when $A$ and $B$ have a common factor, so we prefer to say that $\sigma$ is the \textit{class of the pair $(A,B)$} and  assume $\sigma\geq 0$. Relation \eqref{ratW} can be written in the form
$$
\left( A w\right)'-\left( A'+B\right)w=0,
$$
known as the \textit{Pearson differential equation}, see e.g.~\cite{chihara:1978, Szego75}.

\begin{example} \label{exampleJacobi1}
The simplest and well known example is the case of the Jacobi weight, when 
\begin{equation} \label{ABJacobi}
A(x)=x^2-1, \quad B(x)=(\alpha+\beta) x +  \alpha-\beta.
\end{equation}
If $\Re \alpha, \Re \beta >-1$ then condition \eqref{endpoints} is satisfied for $\Delta=[-1,1]$. 

This is an example of a \textit{classical} weight ($  \sigma=0$); here
$$
w(z)=(z-1)^{\alpha } (z+1)^{\beta }, \quad v(z)=(z-1)^{\alpha+1} (z+1)^{\beta+1 }.
$$
\end{example}
\begin{example} \label{exampleAngelescoJJ}
With
$$
A(x)=x(x-a)(x-b), \quad b<0<a, \quad B(x)=\alpha x(x-b)+ \beta x (x-a)+ \gamma (x-a)(x-b),
$$
we have 
$$
w(x)=(x-a)^{\alpha } (x-b)^{\beta } x^{\gamma }, \quad v(x)=(x-a)^{\alpha+1} (x-b)^{\beta+1} x^{\gamma+1}.
$$
If $\alpha, \beta, \gamma>-1$, we may take  
$$
\Gamma_1=[b,0], \quad \Gamma_2=[0,a], \quad \Delta=\Gamma_1\cup \Gamma_2,
$$
so that condition \eqref{endpoints}  holds. 
This is a  semiclassical weight of class $\sigma=1$.
\end{example}

\subsection{The electrostatic partner} \label{sec:companion}
\

In this section we carry out a purely formal construction that will gain content with additional assumptions in Sections \ref{sec:quasi} and \ref{sec:multiple_orth}. 

For any $w$-integrable function $f$ on $\Delta$ we define its \textit{weighted Cauchy transform}  
\begin{equation} \label{defCauchy}
 \mathfrak C_w[f](z):=\int_\Delta\frac{f(t)w(t)}{t-z}dt,
\end{equation} 
holomorphic in $\C\setminus \Delta$. The case of $f\equiv 1$ is particularly important, so we will use a brief notation
\begin{equation} \label{defwhat}
\widehat w(z):= \mathfrak C_w[1](z)=\int_\Delta\frac{ w(t)}{t-z}dt;
\end{equation}
$\widehat w$ is also known as a \textit{Markov function} related to the weight $w$.
 
Given a polynomial  $P\not\equiv 0$, its \textit{polynomial\footnote{  \, Hilbert's Nullstellensatz implies that $P(t)-P(z)$ can be factored as $(t-z) \tilde P(t,z)$, where $\tilde P$ is a polynomial in its both variables, which implies that $Q$ defined in \eqref{defQnew} is a polynomial in $z$. } of the second kind} $Q$ is defined as
\begin{equation} \label{defQnew}
	Q(z):= \int_\Delta\frac{P(t)-P(z)}{t-z} w(t)dt.
\end{equation}
Additionally,  we call 
\begin{equation} \label{defqN}
	q(z):= 	\frac{	\mathfrak C_w[P](z)}{ w(z)}, \quad z\in \C\setminus\Delta ,
\end{equation}
the corresponding \textit{function of the second kind} of $P$. They are related by the evident identity
\begin{equation} \label{identityQ}
 P(z) \widehat w(z) + Q(z ) =  \mathfrak C_w[P](z) = q(z) w(z).
\end{equation} 
Although $q$ is not necessarily single-valued in $\C\setminus \Delta$, 
\begin{equation} \label{Wq}
w q' =  	(\mathfrak C_w[P])' - 	\mathfrak C_w[P] \,  \frac{w' }{w}  =  	(\mathfrak C_w[P])' - 	\mathfrak C_w[P]  \, \frac{B }{ A }
\end{equation}
is meromorphic in $\C\setminus \Delta$, with only possible poles at the poles of $B/A$.  

We denote by $	\mathfrak{Wrons}[f_1, \dots, f_k] $ the Wronskian determinant of the functions $f_1, \dots, f_k$, namely
\begin{equation} \label{eq:Wronskian}
	\mathfrak{Wrons}[f_1, \dots, f_k] := \det \begin{pmatrix}
		f_1 & \dots & f_k  \\
			f_1' & \dots & f_k'  \\
		\vdots & \dots & \vdots \\
		f_1^{(k-1)} & \dots & f_k^{(k-1)}
	\end{pmatrix}.
\end{equation}
Finally, for a polynomial $P$, we define the transform
\begin{equation} \label{defTransform2}
\mathfrak D_w[P]	 : =    \det
	\begin{pmatrix} P  & 	\mathfrak C_w[P]   \\A   P' & A \left( \mathfrak C_w[P] \right)' -B		\mathfrak C_w[P] \end{pmatrix} =   v \,  \mathfrak{Wrons}[P,q] ,
\end{equation}
(with $v$ from \eqref{defW}), a priori  holomorphic in $\C\setminus \Delta$.

\begin{thm} \label{lem1aux}
If $P$ is a polynomial of degree $N\in \N$ then there exist a polynomial $U$ of degree $\le N-1$ and a polynomial $H$ of degree $\le  \sigma$ such that
	\begin{equation} \label{polynD}
\mathfrak D_w[P]  =  \det \begin{pmatrix}
		P  & \mathfrak C_w[P]  \\
		U  &  \mathfrak C_w[U] + H
	\end{pmatrix}  .
\end{equation}
Moreover, $\mathfrak D_w[P]$ is a polynomial of degree $\le N+\sigma$.
\end{thm}
\begin{proof}
	We start with an identity for $ 	\mathfrak C_w[P] $: for $z\in \C \setminus \Delta$,
	\begin{equation} \label{identity1bis}
	 \mathfrak C_w[A P ' + B P ](z)= 	A(z) \left( \mathfrak C_w[P]\right)  '(z)  -D(z),
	\end{equation}
where
	\begin{equation} \label{identity1bis1}
	D(z):= 	 \int_\Delta \frac{A(z)-A(x) +A'(x)(x-z)}{(x-z)^2}P (x) w(x)\, dx
\end{equation}
is a polynomial of degree $\le \deg(A)-2$. In particular, $D\equiv 0$ if $\deg(A) \le 1$. 

The identity can be established by direct calculation:  integrating by parts and with account of \eqref{endpoints}--\eqref{ratW},
	\begin{align*}
		\mathfrak C_w[A P ' + B P ](z) & = \int_\Delta \frac{A (x)}{x-z}\frac{d\left(P w \right)}{dx} \, dx = - \int_\Delta \frac{d }{dx}\left(\frac{A (x)}{x-z}\right)   P  (x)w (x) dx,
	\end{align*}
	so that the left-hand side in  \eqref{identity1bis} is equal to
	$$
\int_\Delta \frac{A(x) -A'(x)(x-z)}{(x-z)^2}P (x) w(x)\, dx, 
	$$
	and \eqref{identity1bis} follows.
	
Denote by $E $ the polynomial part of the expansion of $ A P ' / P    $ at $\infty$,   i.e.
	$$
	A(z)\frac{P '}{P }(z) = E (z)+ \mathcal O\left(\frac{1}{z} \right), \quad z\to \infty.
	$$
It is easy to see that $E $  is polynomial of degree  $\leq \deg(A)-1$. In this way,  
	\begin{equation}\label{defU}
		U :=AP ' - E  P   
	\end{equation}
  is a polynomial of degree $\leq N-1$.
	
By \eqref{identity1bis},
	\begin{align*}
		\mathfrak C_w[U ]   & = \mathfrak C_w[AP' - E  P ]= \mathfrak C_w[A P' + B P]  + \mathfrak C_w[(-E-B) P]  \\
		& = A  \left(\mathfrak C_w[P]\right)'   - (B+E)(z) \mathfrak C_w[P ] - H ,
	\end{align*}
	where, with the definition \eqref{identity1bis1},
	\begin{equation}\label{DefH}
	H(z):= D(z)+ \int_\Delta   \frac{(E+B) (x)-(E+B) (z) }{x-z}     \,  p(x) w(x)\, dx
	\end{equation}
	is a polynomial of degree    $\le  \sigma$.  In particular (see \eqref{Wq}),
	\begin{equation} \label{identityUhat}
v q' = Aw q' =	 A \left(\mathfrak C_w[P]\right)'  -  B  \mathfrak C_w[P ] =   \mathfrak C_w[U ]     +E  \mathfrak C_w[P ]  + H.
\end{equation}
	
Using \eqref{defU} and \eqref{identityUhat} in \eqref{defTransform2}    we conclude that
	$$
	\mathfrak D_w[P]	
	 =   \det
	\begin{pmatrix} P  & 	\mathfrak C_w[P ]  \\AP '&  \mathfrak C_w[U ]      +E  \mathfrak C_w[P ]  + H    \end{pmatrix}
	=    \det
\begin{pmatrix} P  & 	\mathfrak C_w[P ]   \\ U &  \mathfrak C_w[U ]        + H   \end{pmatrix} ,
	$$
	establishing \eqref{polynD}.
	
 In order to show that $\mathfrak D_w[p]	$ is a polynomial  (of degree at most $N +\sigma $) we use the standard arguments from the Riemann--Hilbert characterization of orthogonal polynomials (see e.g. \cite{MR2000g:47048}). Namely, denote by $\bm Y$ the matrix in the right-hand side of \eqref{polynD}, 
	$$
	\bm Y(z):=  \begin{pmatrix} P  & 	\mathfrak C_w[P ]   \\
		U  &  \mathfrak C_w[U ] + H\end{pmatrix}  ,
	$$
which by definition is holomorphic in $\C\setminus \Delta$. The Sokhotsky-Plemelj formulas imply that $\bm Y$ satisfies that
	\begin{equation} \label{JumpCondition}
	\bm Y_+(x)=\bm Y_-(x) \, \begin{pmatrix}
		1 & w(x) \\
		0 & 1
	\end{pmatrix}, \quad x \in \Delta,
\end{equation}
	where $ \bm Y_\pm$ denote the boundary values of $ \bm Y$ on $\Delta$ from the left/right sides, respectively. Taking determinants in both sides of \eqref{JumpCondition} the Morera theorem yields that   $\det(\bm Y)$ is analytic across $\Delta$. Since the first column of $\bm Y$ is bounded at each finite point of $\C$, and the local behavior of the second column at the end points of $\Delta$ essentially matches that of $\widehat w$ (see \cite[Ch.\ 1]{Gakhov}), it  follows that  $\det(\bm Y)$ is  an entire function. Moreover, since $H$ is a polynomial of degree $\le  \sigma$, there exist a constant $c\in \C\setminus\{0\} $ and $s\in \N \cup \{0\}$, $s\le \sigma$, such that
	$$
	\bm Y(z)= \begin{pmatrix} 1 & 0 \\ 0 & c  \end{pmatrix}\left(\bm I + \mathcal O\left(\frac{1}{z}\right) \right)\, \diag\left(z^{N}, z^{ s} \right), \quad z \to \infty.
	$$
	Liouville's theorem shows that $\det(\bm Y)=	\mathfrak D_w[P]	$ is a polynomial of degree  $N+ s \le N +\sigma $, which concludes the proof. 
\end{proof}

\begin{remark} \label{rem1}
A relation of the form \eqref{identityUhat} was used in \cite{Magnus} as a definition of the semiclassical character of the weight, and in fact, it characterizes \eqref{ratW}.  
\end{remark}

\begin{definition} \label{defCompanionP}
Let $P\neq 0$ be a polynomial. Then we call the polynomial
\begin{equation} \label{defCompanion}
S := \mathfrak D_w[P],
\end{equation}
defined by \eqref{defTransform2}, the \textit{electrostatic partner} of $P$ induced by the weight $w$.
\end{definition}

\begin{remark} \label{remark:scaling}
The definition \eqref{defCompanion} shows that normalization of $S$ is equivalent to normalization of $P$: the scaling $P \mapsto \lambda P$, $\lambda\in \C$, is equivalent to $S\mapsto \lambda^2 S$. As it will be clear in the next section (Proposition \ref{cor:criticalGeneral}), the actual normalization of $S$ is irrelevant to the electrostatic model.
\end{remark}

Some properties of the electrostatic partner $S$ and of the function of the second kind $q$ of $P$ are established in the Appendix \ref{appendixA}.

\subsection{The differential equation and electrostatic model} \label{sec:theODE}
\

\begin{thm}\label{CorolarioEDO}
	Let $P$ be a polynomial, $q$ its function of the second kind defined in \eqref{defqN}, and $S$ the electrostatic partner defined in \eqref{defCompanion}. Then $P$ and $q$ are two solutions of the same second order linear differential equation with polynomial coefficients
	\begin{equation} \label{odegeneral}
	AS y''  +  (A'S -AS' + BS)  y'  + C y   = 0 .
	\end{equation}
If $\deg(A)\leq 1$ then 	  $C=  \mathfrak D_{v}[P']$, with $v=Aw$. 
\end{thm}
\begin{proof}
	Away from the zeros of $A$, the formal identity
	\begin{equation}    	\label{2.1}
	A (z) v(z) 	\mathfrak{Wrons}[y, P, q] (z)=	A^2(z) w(z) \det	\begin{pmatrix}
			y &P &q \\
			y'&P  '&q  '\\
			y'' &P ''&q  ''
		\end{pmatrix}(z)=0
	\end{equation}
	is clearly satisfied by $y=P$ and $y=q $, and thus, by any linear combination of these two functions. Expanding the determinant along the first column yields  the following second order differential equation with respect to $y$:
	\begin{equation}
		f(z) =  f_2(z) y''(z) + f_1(z) y'(z) + f_0(z) y (z) = 0 ,
		\label{ode2.2}
	\end{equation}
	where  (see \eqref{defTransform2}),
	\begin{align}
		f_2  &=
		A v \det		\begin{pmatrix}
			P & q  \\
			P' & q '
		\end{pmatrix}  	,
		\label{2.2}
		\\
		f_1  & =  - A v	\det	\begin{pmatrix}
			P  & q  \\
			P '' & q ''
		\end{pmatrix},
		\label{2.3}
		\\
		f_0 & =
		A v \det		\begin{pmatrix}
			P' & q' \\
			P'' & q''
		\end{pmatrix}(z)= \det	
		\begin{pmatrix}
			P' &  v  q' \\
			A P'' & A v  q''
		\end{pmatrix}.
		\label{2.4}
	\end{align}
By \eqref{defTransform2}, $f_2=AS$, and thus it is a polynomial. Furthermore,  differentiating $f_2$ and using \eqref{defW}, \eqref{mainIdentitySN}, it is straightforward to deduce that
	$$
	f_1 	=A'S -AS' + BS
	$$
	is also a polynomial.

	By \eqref{Wq} and \eqref{identity1bis},
	$$
	v q' =    A	\left(\mathfrak C_w[P] \right)' - 	B \mathfrak C_w[P]   = \mathfrak C_w[AP']     + D =   \mathfrak C_{v}[ P '   ] + D .
	$$
	Differentiating this identity and using \eqref{ratW} we obtain that
	$$
	A v q'' = A  \left( \mathfrak C_{v}[ P '   ] \right)' - (A'+B) \mathfrak C_{v}[ P '   ]  + \left[ AD'-A'D -BD \right].
	$$
	Thus,
	\begin{align*}
		f_0 & =   \det	
		\begin{pmatrix}
			P' & v  q' \\
			A P'' & A v  q''
		\end{pmatrix} \\
		& =  \det	
		\begin{pmatrix}
			P' &  \mathfrak C_{v}[ P '   ]    \\
			A P'' & A  \left( \mathfrak C_{v}[ P '   ] \right)' - (A'+B) \mathfrak C_{v}[ P '   ]
		\end{pmatrix} +  \det	
		\begin{pmatrix}
			P' &   D  \\
			A P'' & AD'-A'D -BD
		\end{pmatrix} .
	\end{align*}
	Observe that  the weight $v = Aw$ is also semiclassical, with
	$$
	\frac{v'(z)}{v(z)}= \frac{B_1(z)}{ A(z)}, \quad  B_1(z)= A'(z)+B(z), \quad z \in \Delta,
	$$
	so that by \eqref{defTransform2},
	$$
	\begin{pmatrix}
		P' &  \mathfrak C_{v}[ P '   ]    \\
		A P'' & A  \left( \mathfrak C_{v}[ P '   ] \right)' - (A'+B) \mathfrak C_{v}[ P '   ]
	\end{pmatrix} = \mathfrak D_{v}[P'],
	$$
	and we get that
	\begin{equation} \label{expressionforC}
	f_0  =  \mathfrak D_{v}[P'] +  \det	
	\begin{pmatrix}
		P' &  D  \\
		A P'' & AD'-A'D -BD
	\end{pmatrix} :=C.
	\end{equation}
	The first term in the right hand side is the electrostatic partner to $P'$ induced by the semiclassical weight $v$, while the determinant  is clearly a polynomial. Moreover, if $\deg(A)\le 1$, it follows from \eqref{identity1bis1} that $D\equiv 0$, which concludes the proof. 
\end{proof}

\begin{remark}\label{remark:operator}
Incidentally, we have established the following differential identity:
\begin{equation} \label{diffOp1}
\mathcal L[y]:= 	(Av)\,  	\mathfrak{Wrons}[y, P, q] = 	AS y''  +  (A'S -AS' + BS)  y'  + C y  ,
\end{equation}
for some polynomial $C$. We will use this later, in the proof of Theorem~\ref{thmODER}. 

Moreover, formula \eqref{expressionforC} and Remark \ref{remark:scaling} show that equation \eqref{odegeneral} is independent of normalization of $S$: a scaling $S \mapsto \lambda^2 S$, $\lambda\in \C$, can be interpreted as $P\mapsto \lambda P$, and thus, $C\mapsto \lambda^2 C$.
\end{remark}

Using the notions of Section \ref{sec:vortices}, and in particular, characterization \eqref{eq:condEqBis}, we see that  \eqref{odegeneral} yields 
that zeros of $P$ are in electrostatic equilibrium in the external field $\varphi(z)=\Re \Phi(z)$, if
$$
\Phi'(z) := -\frac{1}{2} \, \frac{A'S -AS' + BS}{AS}(z) = - \frac{1}{2}  \left(\frac{A'(z)}{A(z)}+\frac{B(z)}{A(z)}-\frac{S'(z)}{S(z)}\right).
$$
Taking into account the definition \eqref{defW}, we conclude:
\begin{prop}\label{cor:criticalGeneral}
Assume that the polynomial $P$ of degree $N$ does not vanish at the zeros of $AS$, where $S= \mathfrak D_w[P]$ is its electrostatic partner. Then the discrete zero-counting measure $\nu(P)$ of $P$ is $\varphi$-critical for the external field
	\begin{equation} \label{fieldGeneral}
		\varphi(z)=\frac{1}{2}\log \left|  \frac{S}{v } \right|(z) .
	\end{equation}
\end{prop}
We can write alternatively that
$$
	\varphi(z)= \frac{1}{2} \left( U^{\nu(A)}(z) - U^{\nu(S)}(z)+ \log \left|  \frac{1}{w } \right|(z) \right).
$$
In other words, the zeros of $P$ (which under assumptions of Proposition~\ref{cor:criticalGeneral} are necessarily simple) are in equilibrium in the external field $\varphi$ induced  by the orthogonality weight $w$, with an additional contribution from point charges of size $1/2$: a \textit{repulsion} from positive charges at $\mathbb A$ and an \textit{attraction} from negative charges at the zeros of the electrostatic partner $S$ (whose location is a priori unknown). The presence of attracting ``ghost charges'' was observed already by M.~H.~Ismail \cite{Ismail2000, Ismail2000c} in his electrostatic interpretation for zeros of orthogonal polynomials with respect to generalized Jacobi weights.

\begin{remark} \label{remark:Heine-Stieltjes}
The study of polynomial solutions of second order ODE of the form \eqref{odegeneral} is a very classical problem that goes back at least two hundred years, see e.g.~\cite{MR2770010, MR2003j:33031, MR2763943} for some historical references and background. In this context, polynomial coefficients $C$ are known as Van Vleck polynomials, and the solution $P$ is the corresponding Heine-Stieltjes polynomial.

The one-to-one correspondence between Heine-Stieltjes polynomials and the discrete critical measures has been established in \cite{MR2770010}. Recently, an additional characterization in terms of non-Hermitian orthogonality satisfied by Heine-Stieltjes polynomials was found in  \cite{BertolaGrava2022}. For the differential equation \eqref{odegeneral}, it follows from their work that there exist a set $\Delta$ and the constants defining the weight $w$ (see the explanation in Section~\ref{sec:generalcase}) such that $P$ satisfies $n+\sigma$ orthogonality conditions with respect to $w/S^2$.  The close connection between orthogonality and electrostatics is well known. In this sense, the electrostatic model in Proposition~\ref{cor:criticalGeneral}, and especially, the form of the external field \eqref{fieldGeneral}, is compatible with the findings of \cite{BertolaGrava2022}.
\end{remark}

\section{Quasi-orthogonality} \label{sec:quasi}

We revisit the facts established in Section~\ref{sec:quasiorth} under an additional assumption on our polynomial $P$.
A (monic) polynomial $P_N$ of degree $N$ is called \textit{quasi-orthogonal}\footnote{A more restrictive notion of quasi-orthogonality was introduced by Chihara \cite{MR86898}, where he assumed a condition equivalent to $n=N-1$; see also \cite{MR3926161}.} with respect to the weight  $w$ on $\Delta$ if it satisfies the following (in general, non-Hermitian) orthogonality relations: for $n\in \N$, $n\le N$,
\begin{equation} \label{orthog1}
\begin{split}
&\int_\Delta x^j P_N(x) w(x)dx=0\,,\qquad j=0,1,\dots, n-1,\\
m_n:=&\int_\Delta x^n P_N(x) w(x)dx\neq 0.
\end{split}
\end{equation}
In particular, if $n=N$,  polynomial $P_N$ is the $N$-th monic  \textit{orthogonal polynomial}. Clearly, last condition in \eqref{orthog1}, that is, $m_n\neq 0$,  is a constraint on the weight $w$ and orthogonality contour $\Delta$.

Using the notation introduced in \eqref{defCauchy}--\eqref{defqN}, we denote
\begin{equation} \label{defQuasi}
	Q_N(z):= \int_\Delta\frac{P_N(t)-P_N(z)}{t-z} w(t)dt
\end{equation}
and
\begin{equation} \label{defQ}
	 	\mathfrak C_w[P_N](z)=\int_\Delta\frac{P_N(t)w(t)}{t-z}dt,  \quad z \in \C \setminus \Delta.
\end{equation}
It is useful to observe that \eqref{orthog1} yields that for any polynomial $T$ of degree $\leq n$,
\begin{equation} \label{eq:CauchTrEq}
	T \mathfrak C_{w}[P_N] = \mathfrak C_{w}[T P_N], \quad \text{that is,} \quad T(z)\, \int_{\Delta} \frac{P_N(t)w(t)}{t-z}dt\,= \int_{\Delta}\frac{T(t) P_N(t) w(t)}{t-z}dt,
\end{equation}
which by \eqref{identityQ} in particular shows that, as $z\to\infty$,
\begin{equation} \label{asymptPN}
	\mathfrak C_w[P_N](z) = P_N(z) \widehat w(z) + Q_N(z )  =-\frac{m_n}{z^{n+1}}+\dots;
\end{equation}
if $\Delta$ is unbounded, we understand the equality above in the asymptotic sense. Notice also that if $P_N$ satisfies a full set of orthogonality conditions ($N=n$), then \eqref{asymptPN} shows that the rational function
$$
\pi_N:=-\frac{Q_N}{P_N}
$$
is the $N$-th diagonal Pad\'e approximant to $\widehat w$ at infinity, fact that is well known. 

We  denote  by $q_N$, $U_N$ and $H_N$ the functions $q$, $U$ and $H$ introduced in \eqref{defqN}, \eqref{defU} and \eqref{DefH}   for $P=P_N$, and let
\begin{equation} \label{main2}
	S_N := \mathfrak D_w[P_N]	 = \det
	\begin{pmatrix} P_N & 	\mathfrak C_w[P_N]    \\A   P_N' & A \left( \mathfrak C_w[P_N]\right)' -B		\mathfrak C_w[P_N] \end{pmatrix} 
\end{equation}
be the electrostatic partner to $P_N$, induced by the weight $w$. 
By Theorem~\ref{lem1aux},
\begin{equation} \label{mainIdentitySN}
S_N=v\,   \mathfrak{Wrons}[P_N, q_N]
 =  \det \begin{pmatrix}
	P_N  & \mathfrak C_w[P_N]   \\
	U_N  &  \mathfrak C_w[U_N]  + H_N
\end{pmatrix}  ,
\end{equation}
and $S_N$ is a polynomial of degree $\le N+\sigma$.  In fact, due to quasi-orthogonality relations we can say more: the upper bound on the degree of $S_N$ is lessened by the number of orthogonality conditions: 
\begin{cor}\label{lemMain}
If $P_N$ satisfies \eqref{orthog1} then for the electrostatic partner $S_N$, 
\begin{equation} \label{leadingC}
	S_N(z)= m_n \, z^{N-n} \left[  (N+n+1) \, z^{ \deg(A)-2}  \left(1 + O\left( 1/z \right)  \right)+    \kappa \, z^{ \deg(B)-1} \left(1+ O\left( 1/z \right)  \right)\right], \quad z\to \infty,
\end{equation}
where $\kappa$ is the leading coefficient of the polynomial $B$. In particular, with assumptions \eqref{orthog1}, $\deg(S_N)\le N-n+\sigma$ and can be strictly less only if $\deg(A)=\deg(B)+1$ and $\kappa =-(N+n+1)$.

Moreover,
if $n\ge \sigma +1$ then $H_N\equiv 0$, and if $n\ge \sigma +2$, then additionally,  $U_N $ (of degree $\le N-1$) is  quasi-orthogonal:
\begin{equation} \label{quasiorthSnBis}
	\int_\Delta x^j U_{N}(x)w(x)dx=0, \quad j=0,\dots, n-\sigma-2.
\end{equation}
\end{cor}
\begin{proof}
Using \eqref{asymptPN} in the definition \eqref{main2} we obtain \eqref{leadingC}.  Also from \eqref{asymptPN}, \eqref{identityUhat} and since $\deg(B+E)\leq \sigma +1$, we see that
\begin{equation} \label{asymptU}
	\mathfrak C_w[U_N] 	 (z) +H(z) = A(z) (\mathfrak C_w[P_N] )'(z)  - (B+E)(z) \mathfrak C_w[P_N] (z)  =  \mathcal O\left( \frac{1}{z^{n-\sigma} }\right) , \quad z \to \infty.
\end{equation}
If $n\ge \sigma +1$, \eqref{asymptU} implies  that $H\equiv 0$; with the assumption $n\ge \sigma +2$  we also get  the quasi-orthogonality conditions \eqref{quasiorthSnBis}.
\end{proof}

\begin{remark} \label{rem:OP}
It is interesting to examine the conclusions of Corollary \ref{lemMain} in the particular case of  polynomials $P_N$ orthogonal on $\Delta$ with respect to the weight $w$ (so that $n=N$). Consequently, the degree of $S_N$ is uniformly bounded: $\deg S_N\leq \sigma$. Furthermore, $\deg S_N=\sigma$ if either
$$
\deg(A)\neq \deg(B)+1 \quad \text{or} \quad \kappa \neq -(N+n+1),
$$
see \eqref{leadingC}. 
If additionally $\sigma=0$ (i. e. when $P_N$ is basically a classical orthogonal polynomial), \eqref{quasiorthSnBis} asserts that $U_{N}$ is the $(N-1)$-th orthogonal polynomial, and up to normalization, $U_{N}=P_{N-1}$, in which case \eqref{mainIdentitySN} boils down to
$$
S_N=	V  \det \begin{pmatrix}
	P_N & q_N \\
	P'_{N } & q_N'
\end{pmatrix} =  \const \times \det \begin{pmatrix}
	P_N & \mathfrak C_w[P_N]   \\
	P_{N-1} &  \mathfrak C_w[P_{N-1}]  
\end{pmatrix},
$$
that is, polynomial $S_N$ is the product of a factor related with the weight and the Wronskian of $P_N$ and $q_N$, and also is the Casorati determinant associated with $P_N$; such an identity appears for instance in \cite[formula (3.6.13)]{Ismail05}. Thus, \eqref{mainIdentitySN}  is a generalization of such kind of relations to quasi-orthogonal polynomials with respect to semiclassical weights, a fact that has an independent interest.
\end{remark}

With the additional assumption of quasi-orthogonality of $P_N$,  Theorem \ref{CorolarioEDO} and Proposition \ref{cor:criticalGeneral} are still valid, with $S$ replaced by $S_N$. For instance, $P_N$ satisfying \eqref{orthog1} and its function of the second kind $q_N$ are  solutions of the  linear differential equation with polynomial coefficients  
	\begin{equation} \label{ode1}
	AS_N y''+(A'S_N -AS'_N+BS_N )y'+C_N y=0,
	\end{equation}
where $S_N$ is the electrostatic partner defined in \eqref{main2}, and $C_N$ is a polynomial.  Its degree  can be easily estimated by using in  \eqref{ode1} that $\deg P_N=N$ and $\deg(S_N) \le N-n+\sigma$ , which yields that
\begin{equation} \label{degSN}
	\deg(C_N) \leq N-n+2\sigma.
\end{equation}
In some cases, we can be more precise. For instance, if $ \deg (A) \leq   \deg(B)  $ and $\deg(P_N)=N$ then   using \eqref{leadingC} and \eqref{ode1} we conclude that 
$$
\deg (S_N)=N-n+\deg(B)-1, \quad \deg(C_N) = N-n+2\deg(B)-2.
$$
 
Equation \eqref{ode1} is known in the literature on semiclassical orthogonal polynomials, see for instance \cite[Eq.~(20)]{Shohat} or  \cite[Eq.~(18)]{Magnus}. The equation in \cite{Magnus} was obtained for orthogonal polynomials only, when the polynomial $S_N$ (denoted by $\Theta_n$ there) has degree $\le \sigma$. In this sense,  \eqref{ode1} is an extension of these results.  Recall that Magnus uses in \cite{Magnus} an alternative definition for the semiclassical weights, in terms of an identity for the Cauchy transform  \eqref{defCauchy}, which is equivalent to \eqref{ratW}, see Remark~\ref{rem1}.

\begin{example}\label{ex:Jacobi1}
Let us return to Jacobi polynomials 
\begin{equation} \label{JacPolExplicit}
P_N(x)=P_N^{(\alpha, \beta)}(x)=\frac{1}{2^{N}\, N ! }(x-1)^{-\alpha}(x+1)^{-\beta}\left[(x-1)^{\alpha+N}(x+1)^{\beta+N}\right]^{(N)},
\end{equation}
corresponding to the weight considered in Example~\ref{exampleJacobi1}, for which  
$$
A(x)=x^2-1, \quad B(x)=(\alpha+\beta) x +  \alpha-\beta,
$$
and $\sigma=0$.  
It is known that $P_{N}^{\left( \alpha ,\beta \right) }$ may have a multiple zero at
$x=1$ if $\alpha \in \{-1,\ldots,-N\}$, at $x=-1$ if $\beta \in
\{ -1,\ldots,-N\} $ or, even, at $x=\infty $ (which means a degree
reduction) if $N+\alpha +\beta \in \{-1, \ldots,-N\}$; otherwise, all zeros are simple, see e.g.~\cite{ETNA05, Szego75}.

If $ \alpha,  \beta >-1$, we can take $\Delta=[-1,1]$. 
By Corollary~\ref{lemMain},  the electrostatic partner $S_N$ is a constant ($\not \equiv 0$ if $\alpha + \beta \neq -2N-1$). Since
$$
v(z)=(z-1)^{\alpha+1} (z+1)^{\beta+1 }, 
$$
the zeros of $P_N$ are in equilibrium in the external field
\begin{equation} \label{extFieldJacobi}
\varphi(z)=\frac{1}{2}\log \left|  \frac{1}{(z-1)^{\alpha+1} (z+1)^{\beta+1 } } \right|=\frac{\alpha+1}{2} U^{\delta_1}(z)+\frac{\beta+1}{2} U^{\delta_{-1}}(z).
\end{equation}

In other cases, considered non-standard, when either $\alpha \leq -1$ or $\beta\le -1$, Jacobi polynomials satisfy  non-hermitian quasi-orthogonality conditions, see \cite[Theorem 4.1]{ETNA05}. 

For instance, if $\alpha, \beta, \alpha + \beta \notin \mathbb Z$, and $ -N<\alpha < -1$, then all zeros of $P_N=P_{N}^{\left( \alpha ,\beta \right) }$ are simple, and $P_N(-1)\neq 0$, see e.g.~\cite[Ch.~IV]{Szego75}. Polynomial $P_N$ satisfies also a quasi-orthogonality relation \eqref{orthog1} with $n=N-[-\alpha]$, but with a modified weight, $w(z)=(z-1)^{\alpha+[-\alpha]}(z+1)^\beta$, so that now
$$
v(z)=(z-1)^{\alpha+[-\alpha]+1} (z+1)^{\beta+1 }.
$$
Moreover, as  $\Delta$ we can take an arbitrary curve oriented clockwise, connecting $1 - i 0$ with $1 + i 0$ and lying entirely in $\C \setminus [-1, +\infty)$, except for its
endpoints; if $\beta>-1$, then $\Delta$ can be deformed into $[-1,1]$.
 
By Corollary~\ref{lemMain}, the electrostatic partner $S_N$ is of degree exactly $ [-\alpha]$. We will show next that, up to a normalizing constant,
\begin{equation} \label{expressionS_NJac}
S_N(x)=(x-1)^{[-\alpha]}.
\end{equation}
However, notice that by Proposition~\ref{cor:criticalGeneral}, the discrete zero-counting measure $\nu(P_N)$ of $P_N=P_{N}^{\left( \alpha ,\beta \right) }$ is $\varphi_N$-critical in the external field
$$
	\varphi(z)=\frac{1}{2}\log \left|  \frac{S_N}{v } \right|(z) ,
$$
which coincides with the one given in \eqref{extFieldJacobi}. In other words, even with the non-standard values of the parameters $\alpha, \beta$, we still get equilibrium in the field \eqref{extFieldJacobi}.

Returning to the expression for $S_N$ in the case under consideration, it is known that for all values of the parameters, $P_N$ is a solution of the differential equation 
$$
	Ay''+(B+A')y'-\lambda_N y=0, \quad \lambda_N= N(N+\alpha+\beta+1),
$$
with $A$ and $B$ given in \eqref{ABJacobi}, see e.g.~\cite[Ch.~IV]{Szego75} or \cite[Section 18.8]{MR2723248}. At the same time,   from our discussion it follows that $P_N$ is a solution of the differential equation \eqref{ode1}, namely
$$
AS_N y''+(A'S_N -AS'_N+B_1S_N )y'+C_N y=0, \quad B_1(x)=(\alpha+[-\alpha]+\beta) x +  \alpha+[-\alpha]-\beta.
$$
Combining these two equations  we obtain the identity
$$
\left((B-B_1)S_N + AS_N'\right)\,P_N' = (\lambda_N S_N + C_N)\,P_N,
$$ 
which using the explicit expressions for $A$, $B$ and $B_1$, can be rewritten as
\begin{equation}\label{id1S1}
	(x+1) \left((x-1) S'_N(x) - [-\alpha] S_N(x)\right) P'_N(x) = (\lambda_N S_N(x) + C_N(x)) P_N(x).
\end{equation}
Recall that $\deg S_N=[-\alpha] \le N$; a simple argument shows that 
$$
\deg \left((x-1) S'_N(x) - [-\alpha] S_N(x)\right)<N. 
$$
Indeed, the assertion is obvious for $\deg S_N< N$; if $\deg S_N=[-\alpha] = N$ then the leading coefficients of $(x-1) S'_N(x)$ and $[-\alpha] S_N(x)$ match, so the assertion follows also in this case. 

Since in the situation we are analyzing $P_N(x)$ and $(x+1)P'_N(x)$ are relatively prime (up to a multiplicative constant), by \eqref{id1S1} we conclude that $(x-1) S'_N(x) - [-\alpha] S_N(x)$ must vanish at all $N$ distinct zeros of $P_N$, which is possible only if $(x-1) S'_N(x) - [-\alpha] S_N(x)\equiv 0$, which implies \eqref{expressionS_NJac}. Incidentally, we also obtain that in this case, $C_N=-\lambda_N S_N  $.

We will return to the example of Jacobi polynomials with non-standard values of the parameters  in the Section~\ref{sec:complex}, when we will address multiple orthogonality.

\end{example}

\begin{example}\label{ex:JacobiQuasi}
	It is instructive to compare our construction with the results of \cite{MR3926161} in the case of quasi-orthogonal Jacobi polynomials. These are polynomials
$$
P_N(x)=\widehat P_N(x) + c \, \widehat P_{N-1}(x), \quad c \in \R,
$$ 
where $\widehat P_N$ is the $N$-th orthonormal Jacobi polynomial 
$$
\widehat P_N(x) =\sqrt{\frac{(\alpha+\beta+2 N+1) \Gamma(\alpha+\beta+N+1) N ! }{2^{\alpha+\beta+1} \Gamma(\alpha+N+1) \Gamma(\beta+N+1)}}\, P_N^{(\alpha, \beta)}(x)
$$
and $P_N^{(\alpha, \beta)}$ is defined in \eqref{JacPolExplicit}. Obviously, for $\alpha, \beta>-1$, $P_N$ satisfies \eqref{orthog1} with $n=N- 1$ and the weight $w(x)=(x-1)^{\alpha } (x+1)^{\beta }$. According to Corollary~\ref{lemMain},  $S_N(x)=x- t_N$ (up to a multiplicative constant), and by Proposition~\ref{cor:criticalGeneral},  the discrete zero-counting measure $\nu(P_N)$ of $P_N$ is $\varphi$-critical in the external field
\begin{align*}
	\varphi(x) & =\frac{\alpha+1}{2}\,  U^{\delta_1}(x)+\frac{\beta+1}{2}\,  U^{\delta_{-1}}(x)- \frac{ 1}{2} U^{\delta_{t_N}}(x) \\
	& =\frac{\alpha+1}{2}\,  \log \frac{1}{|x-1|} +\frac{\beta+1}{2}\,   \log \frac{1}{|x+1|} - \frac{ 1}{2}  \log \frac{1}{|x-t_N|} .
\end{align*}
Explicit expressions allow to calculate $t_N$ by definition \eqref{mainIdentitySN}:
$$
t_{N}=-\frac{(\alpha+\beta+1+2 N)+c ^{2}(\alpha+\beta-1+2 N)}{(2 N+\alpha+\beta) c } a_{N} + \frac{\beta^{2}-\alpha^{2}}{(2 N+\alpha+\beta)^{2}},
$$
where
$$
a_{N}=\frac{2}{\alpha+\beta+2 N} \sqrt{\frac{N (\alpha+N)(\beta+N)(\alpha+\beta+N)}{(\alpha+\beta-1+2 N)(\alpha+\beta+1+2 N)}}\, .
$$
This expression coincides with the one obtained in \cite[Section 5.2]{MR3926161}. 
\end{example}

\section{Multiple orthogonality of Type II}\label{sec:multiple_orth}

\subsection{The general case} \label{sec:generalHermitePade}
\

We are interested in type II \textit{multiple} or \textit{Hermite-Pad\'e} orthogonal polynomials with respect to semiclassical weights. For the sake of simplicity, we restrict ourselves  to two positive weights $w_1$, $w_2$, supported on $\Delta_1\subset \R$ and $\Delta_2\subset \R$, respectively. Given an ordered pair $\bm n= (n_1,n_2)\in \Z_{\ge 0}^2$,  where $\Z_{\ge 0}=\N\cup \{0\}$, we look for a (monic) polynomial $P_{\bm n}$ of total degree at most $N=|\bm n|:=n_1+n_2$, such that
\begin{equation} \label{defHP}
	\int_{\Delta_i}x^jP_{\bm n}(x)w_i(x)dx  \begin{cases}
		=0,&\quad j\leq n_i-1,\\
		\neq 0,&\quad j=n_i,
	\end{cases}
	\qquad i=1, 2.
\end{equation}
In this way, the definition of type II Hermite-Pad\'e orthogonal polynomials \eqref{defHP} boils down to two simultaneous quasi-orthogonality conditions.

Additionally, we assume that both weights are semiclassical and belong to the framework discussed in Section~\ref{sec:quasiorth}. More precisely, we assume that 
for each $i=1, 2$, $\Delta_i\subset \R$ is a finite union of non-overlapping intervals joining real zeros of a real polynomial $A_i$ and eventually $\infty$, and that each weight $w_i$ is defined on $\Delta_i$ in such a way that on each component  $\Gamma_j$ it coincides, up to a non-zero multiplicative constant, with a boundary value $(w_i)_+$ of the function defined below:
\begin{equation} \label{defWHP}
	w_i(z):=\exp \left( \int ^z \frac{B_i(t)}{  A_i(t)}\, dt \right) , \quad  \quad i=1, 2,
\end{equation}
for some real polynomials $B_1, B_2$. As before, we  use the notation $\mathbb A_i$ for the set of zeros of $A_i$ on $\C$, $i=1,2$.

We assume also that $w_i$'s have finite moments:
$$
\int_{\Delta_i} |x|^m |w_i(x)| dx <+\infty, \quad m=0, 1, 2, \dots, \quad i=1, 2.
$$
As a consequence, for $m=0, 1, 2, \dots$, 
\begin{equation}\label{endpoints2}
z^m	v_i(z)=0 \text{ at endpoints of every  subinterval of $\Delta_i$, $i=1, 2$,}
\end{equation}
with
\begin{equation} \label{defWHPv}
	v_i(z):= A_i(z) w_i (z), \quad i=1, 2.
\end{equation}
We also have that
\begin{equation} \label{ratWHP}
	\frac{w_i'(z)}{w_i (z)}=\frac{B_i(z)}{ A_i(z)}, \quad i=1, 2,
\end{equation}
with the identity taking place away from the singularities of $w_i$.  
As in \eqref{def:class},
\begin{equation} \label{def:classHP}
	\sigma_i := \max \{\deg(A_i)-2, \deg(B_i)-1 \},\quad i=1, 2.
\end{equation}

\begin{remark} \label{remark:A1=A2}
In the special case when $A_1=A_2=A$ and $B_1=B_2=B$, we will always assume, without loss of generality, that both $w_i$ are normalized in such a way that for every selection of the branch,
	$$
	w_1(z)=w_2(z)=w(z)=\exp \left( \int ^z \frac{B(t)}{ A(t)}\, dt \right) .
	$$
This  situation was considered for instance in \cite{Aptekarev:97} (and previously, in  \cite{kaliaguine:1981} and \cite{kaliaguine:1996}). Several examples of the case when $\Delta_1=\Delta_2$,  $A_1=A_2$, but $B_1\neq B_2$ for classical weights  ($\sigma_i=0$) appear in  \cite{AptBraVA}.
\end{remark}

For $P_{\bm n}$ we  define the corresponding polynomials
\begin{equation} \label{defQHP}
Q_{{\bm n},i}(z):= \int_{\Delta_i}\frac{P_{\bm n}(t)-P_{\bm n}(z)}{t-z} w_i(t)dt, \quad i=1, 2,
\end{equation}
and functions of the second kind,
\begin{equation}\label{2ndkindfunt}
	q_{\bm n,i}(z):=\frac{\mathfrak C_{w_i}[P_{\bm n}](z) }{w_i (z)}= \frac{1}{ w_i (z)}\int_\Delta\frac{P_{\bm n}(t)w_i(t)}{t-z}dt , \quad i=1, 2.
\end{equation}
By \eqref{eq:CauchTrEq}--\eqref{asymptPN},  for $i=1, 2$, 
\begin{equation} \label{eq:asymptCauchy}
	 \mathfrak C_{w_i}[P_{\bm n}](z) = P_{\bm n} (z) \widehat w_i(z) + Q_{{\bm n},i}(z ) = \mathcal O(z^{-n_i-1}), \quad z\to \infty. 
\end{equation}
This shows that the pair of rational functions
	\begin{equation} \label{def_approx}
\left( \pi_{{\bm n},1}, \pi_{{\bm n},2}\right):=\left( -\frac{Q_{{\bm n},1}}{P_{\bm n}}, - \frac{Q_{{\bm n},2}}{P_{\bm n}}\right)
\end{equation}
are the \textit{simultaneous} or \textit{Hermite--Pad\'e approximants of type II} to the pair of functions $\left(  \widehat w_1 ,  \widehat  w_2 \right) $, and that the functions $\mathfrak C_{w_i}[P_{\bm n}]$ are the corresponding \textit{residues}, see  e.g.~\cite{Aptekarev:97, MR4238535}.

We will impose an additional condition enforcing an independence of both weights, namely we will assume that
\begin{equation} \label{mainconditionwronks}
\mathfrak{Wrons}[P_{\bm n}, q_{\bm n,1}, q_{\bm n,2}]  \not\equiv 0.
\end{equation}
It is equivalent to assuming that $\mathfrak{Wrons}[P_{\bm n}, q_{\bm n,1}, q_{\bm n,2}]\neq 0$ at some point different from the zeros of $A_1A_2$.  Possibly, condition \eqref{mainconditionwronks} is equivalent to $\bm n$ being a normal index, although this is just a natural conjecture that deserves further study.
 
By the analysis from Section \ref{sec:quasiorth}, $P_{\bm n}$ has now two electrostatic partners, 
	\begin{equation} \label{defS12}
		S_{\bm n, i}(z) := \mathfrak D_{w_i}[P_{\bm n}]= v_i(z)  	f_{\bm n,i} (z), \quad i=1, 2,
\end{equation}
where 
\begin{equation} \label{wronskF}
	f_{\bm n,i}:= \mathfrak{Wrons}[P_{\bm n}, q_{\bm n, i}], \quad i=1, 2,
\end{equation}
and $\mathfrak{Wrons}[\cdot, \cdot]$ stands for the Wronskian as defined in \eqref{eq:Wronskian}. An equivalent formula that might shed some light on the behavior of these polynomials is
	\begin{equation} \label{S_inew}
		S_{\bm n, i}(z) = v(z) P_{\bm n}^2(z) \,  \left(\frac{\widehat w_i(z)-\pi_{\bm n,i}(z)}{w_i(z)}  \right)' = v(z) P_{\bm n}^2(z) \,  \left(\frac{\mathfrak C_{w_i}[P_{\bm n}](z)}{w_i(z)P_{\bm n}(z)} \right)'   , \quad i=1, 2,
\end{equation}
where $\pi_{\bm n,i}$ are defined in \eqref{def_approx}; it can be checked using \eqref{eq:asymptCauchy} by direct computation. 

By Theorem~\ref{CorolarioEDO},
\begin{thm}\label{TeoClaveHP}
	Let $\bm n= (n_1,n_2)\in \Z_{\ge 0}^2$, and let $P_{\bm n}$ be  a monic polynomial  of degree $N=n_1+n_2$ (assuming it exists) that satisfies the multiple orthogonality conditions \eqref{defHP}.  
	Then $S_{\bm n, i}$  are polynomials   of degree at most $N-n_i+\sigma_i $, $i=1, 2$, and there exist polynomials $C_{\bm n, i} $, 
	$$
	\deg(C_{\bm n, i}) \leq N-n_i+2\sigma_i , \quad i=1, 2 , 
	$$  
	 such that   $P_{\bm n}$ is a solution of the system of linear differential equations
	\begin{equation} \label{odeHP}
		\begin{split}
			& A_i S_{\bm n, i}\, y''+(A'_i S_{\bm n, i}-A_i S_{\bm n, i}'+B_iS_{\bm n, i})\, y'+C_{\bm n, i} \,  y=0,  \quad i=1, 2.
		\end{split}
	\end{equation}
\end{thm}

\begin{remark}\label{rem:linearcomb}
In the case when $A_1=A_2$ and $B_1=B_2$ we can replace in the expressions \eqref{odeHP} the polynomial $S_{\bm n, i}$ by any linear combination
$$
a \, S_{\bm n, 1} + b \, S_{\bm n, 2}, \quad a, b \in \R.
$$
The resulting differential equations still have $P_{\bm n}$ as one of their solutions; see the result of numerical experiments for Appell's polynomials at the end of Section~\ref{sec:AngelescoJacobi}, especially the plots in Figure~\ref{Fig2Appell}.
\end{remark}

An application of  Proposition \ref{cor:criticalGeneral} yields:
\begin{cor}\label{cor:HPinterp}
		Assume that the polynomial $P_{\bm n}$ has no common zeros neither  with $A_1 S_{\bm n,1}$ nor with $A_2 S_{\bm n,2}$.
	Then  the discrete zero-counting measure $\nu(P_{\bm n})$ of $P_{\bm n}$ is $\varphi_i$-critical, in the sense of Definition~\ref{defCriticalScalar}, for the external field
	\begin{equation} \label{fieldHP}
		\varphi_i(z)=\frac{1}{2}\log \left|  \frac{S_{\bm n,i}}{v_i    } \right|(z) , 
	\end{equation}
for both $i=1, 2$. 
\end{cor}

Thus, the zeros of the Hermite--Pad\'e polynomial $P_{\bm n}$ are in equilibrium (i.e., their counting measure is critical) in two different external fields, each one induced by the corresponding orthogonality weight $w_i$, with an addition of attracting charges placed at the zeros of the electrostatic partner $S_{\bm n,i}$. This ``redundancy'' allows us to provide also an electrostatic model  for these additional charges forming the external fields $\varphi_i$. Formally, it will be a consequence of the differential equations for the electrostatic partners $S_{\bm n, i}$ that we establish next, but we need to introduce first another auxiliary polynomial, this time generated by both weights simultaneously.

We define the following differential operators:
\begin{equation} \label{defDiffOperators}
\mathcal L_i[y]:= 	(A_i v_i ) \, 	\mathfrak{Wrons}[y, P_{\bm n}, q_{\bm n,i} ], \quad i=1, 2.
\end{equation}
Notice that we omit indicating the dependence of $\mathcal L_i$ from $\bm n$ for the sake of brevity of notation.
\begin{prop}\label{prop:WronskianS/w}
Function 
\begin{equation} \label{defRpolyn}
	R_{\bm n}:=  A_2  v_2 \, \mathcal L_1[q_{\bm n,2}] = - A_1 v_1 \,  \mathcal L_2[q_{\bm n,1}] 
	=(A_1 A_2 v_1 v_2)  \, \mathfrak{Wrons}[P_{\bm n}, q_{\bm n,1}, q_{\bm n,2}]
\end{equation}
is a polynomial of degree $\leq 2\sigma_1+2\sigma_2+3$.

If $A_1=A_2=:A$, then $A$ is a factor of $R_{\bm n}$. If in addition $B_1=B_2$, then $A^2$ is a factor of $R_{\bm n}$, i.e.
\begin{equation} \label{defR*}
	R_{\bm n}^*:=\frac{R_{\bm n}}{A^2}= v^2   \, \mathfrak{Wrons}[P_{\bm n}, q_{\bm n,1}, q_{\bm n,2}], \quad v:= A   w,
\end{equation} is a polynomial.
\end{prop}
Notice that by our assumption \eqref{mainconditionwronks}, $R_{\bm n}\not \equiv 0$.
\begin{proof}
Let us define	
	\begin{equation}\label{IdR0}
		R_{\bm n} :=\frac{A_1S_{\bm n, 1} C_{\bm n, 2}-A_2S_{\bm n, 2}C_{\bm n, 1}}{P_{\bm n}'}.
	\end{equation}
Multiplying the equations in \eqref{odeHP} evaluated at $y=P_{\bm n}$ by $A_2S_{\bm n, 2}/P_{\bm n}$ and $A_1S_{\bm n, 1}/P_{\bm n}$ , respectively, and subtracting we obtain
\begin{equation}\label{IdR1}
R_{\bm n}\,  P_{\bm n} = -\mathfrak{Wrons} [A_1, A_2]	 S_{\bm n, 1}S_{\bm n, 2}  +A_1A_2 \, \mathfrak{Wrons}[S_{\bm n, 1},  S_{\bm n, 2}]  +(A_2B_1-A_1B_2)S_{\bm n, 1}S_{\bm n, 2} .
\end{equation}
Notice that $R_{\bm n}\,  P_{\bm n}$ is a polynomial.  An immediate consequence of the statement a) of Proposition \ref{prop:fromTHM2.1} is that $ S_{\bm n, 1} / P_{\bm n}'$ and $S_{\bm n, 2}/P_{\bm n}'$ are analytic at the zeros of $P_{\bm n}$, which  together with the bounds on the degree of  $C_{\bm n, i}$ from Theorem \ref{TeoClaveHP}  yields that  $R_{\bm n}$ is a polynomial of degree at most $2\sigma_1+2\sigma_2+3$.

On the other hand, straightforward calculation using \eqref{ratWHP} and \eqref{wronskF} shows also that
\begin{align*}
-\mathfrak{Wrons} [A_1, A_2]	 S_{\bm n, 1}S_{\bm n, 2}  +A_1A_2 \, \mathfrak{Wrons}[S_{\bm n, 1},  S_{\bm n, 2}]  & +(A_2B_1-A_1B_2)S_{\bm n, 1}S_{\bm n, 2} \\
& = A_1 A_2 v_1 v_2 \,  \mathfrak{Wrons}[f_{\bm n,1}, f_{\bm n, 2}],
\end{align*}
so that \eqref{IdR0} reduces to 
\begin{equation} \label{wronksHP}
A_1 A_2 v_1 v_2 \,  \mathfrak{Wrons}[f_{\bm n,1}, f_{\bm n, 2}] =P_n R_{\bm n}.
\end{equation}

Let
$$
d_{\bm n}:=\mathfrak{Wrons}[f_{\bm n,1}, f_{\bm n, 2}]. $$ 
Since
$$
f_{\bm n,i} '= \mathfrak{Wrons}[P_{\bm n}, q_{\bm n, i}]' = \det \begin{pmatrix}
	P_{\bm n} &  q_{\bm n, i} \\
	P_{\bm n}'' & q_{\bm n, i}''
\end{pmatrix},
$$
this gives us the following identity for $d_{\bm n}$:
\begin{equation} \label{identitiesW}
	\begin{split}
	d_{\bm n} =&  \det \begin{pmatrix}
		P_{\bm n} &  q_{\bm n, 1} \\
		P_{\bm n}' & q_{\bm n, 1}'
	\end{pmatrix}  \det \begin{pmatrix}
		P_{\bm n} &  q_{\bm n, 2} \\
		P_{\bm n}'' & q_{\bm n, 2}''
	\end{pmatrix} - \det \begin{pmatrix}
		P_{\bm n} &  q_{\bm n, 2} \\
		P_{\bm n}' & q_{\bm n, 2}'
	\end{pmatrix}  \det \begin{pmatrix}
		P_{\bm n} &  q_{\bm n, 1} \\
		P_{\bm n}'' & q_{\bm n, 1}''
	\end{pmatrix}  \\
= & P_{\bm n} \, \mathfrak{Wrons}[P_{\bm n}, q_{\bm n,1}, q_{\bm n,2}] ,
	\end{split}
\end{equation}
which together with \eqref{wronksHP} proves the first part of assertion. 

If $A_1=A_2=:A$ then \eqref{IdR1} becomes
\begin{equation*}
	A^2(S_{\bm n, 1}S_{\bm n, 2}'-S_{\bm n, 1}'S_{\bm n, 2})+A(B_1-B_2)S_{\bm n, 1}S_{\bm n, 2}=A  \frac{R_{\bm n}}{A}  P_{\bm n},
\end{equation*}
with
\begin{equation} \label{expreforR}
\frac{R_{\bm n}}{A} =  \frac{S_{\bm n, 1}C_{\bm n, 2}-S_{\bm n, 2}C_{\bm n, 1}}{P_{\bm n}'}.
\end{equation}
Assume that    $A(z_0)=0$. If also $P_{\bm n}(z_0)= 0$ then by the same argument as before, $ S_{\bm n, 1} / P_{\bm n}'$ and $S_{\bm n, 2}/P_{\bm n}'$ are analytic at $z_0$, as well as the right hand side of \eqref{expreforR}. If $P_{\bm n}(z_0)\neq 0$ but $P_{\bm n}'(z_0)= 0$, then by \eqref{odeHP},  
$$
 C_{\bm n, i} (z_0)\,  P_{\bm n}(z_0)=- A(z_0)  S_{\bm n, i}(z_0)\, P_{\bm n}''(z_0) , \quad i=1,2.
$$
In this case, $A  P_{\bm n}''/P_{\bm n}'$ is analytic at $z_0$, which implies again that the expression in the right hand side of \eqref{expreforR} is analytic at $z_0$. This proves that $R_{\bm n}/A$ is a polynomial. 

If in addition $B_1=B_2$, then \eqref{IdR1} reduces to
\begin{equation} \label{identityRpartCase}
	A^2(S_{\bm n, 1} S_{\bm n, 2}'-S_{\bm n, 1}'S_{\bm n, 2})=
	R_{\bm n} P_{\bm n}
	\qquad \Rightarrow\qquad \frac{R_{\bm n}}{A^2}=\frac{S_{\bm n, 1} S_{\bm n, 2}'-S_{\bm n, 1}'S_{\bm n, 2}}{P_{\bm n}}.
\end{equation}
Hence, $R_{\bm n}/A^2$ could have poles only at the common roots of $A$ and $P_{\bm n}$. But in this case, by Proposition~\ref{prop:fromTHM2.1}, $S_{\bm n, 1}/P_{\bm n}$ and $S_{\bm n, 2}/P_{\bm n}$ are analytic at the zeros of $ P_{\bm n}$. The proof is complete. 
\end{proof}

\begin{remark} \label{remark:equal}
		In the case when
		$$
		A_1=A_2=A,  \qquad \deg(A)< \deg(B_i)+1, \quad i =1, 2,  
		$$
		formula \eqref{IdR1} shows that
		\begin{equation} \label{degRoverAparticularcase}
			\deg\left( \frac{R_{\bm n}}{A}\right)= \deg(B_1)+ \deg(B_2)+ \deg(B_1-B_2)-2.
		\end{equation}
		If on the other hand, $A_1=A_2=A$,   $B_1=B_2=B$,  
		$$
		\sigma = \max \{\deg(A)-2, \deg(B)-1 \}=1 \quad \text{and} \quad n_1=n_2,
		$$
		then $R_{\bm n}^*$ is a constant; in other words,
		$$
		R_{\bm n}(x) = \const \times A^2.
		$$
		Indeed, by \eqref{identityRpartCase}, 
		$$
		R_{\bm n}^*=\frac{S_{\bm n, 1} S_{\bm n, 2}'-S_{\bm n, 1}'S_{\bm n, 2}}{P_{\bm n}}
		$$
		is a polynomial. Since 
		$\deg(S_{\bm n,1})=n_2+1$ and $\deg(S_{\bm n,2})=n_1+1$, $\deg(R_{\bm n}^*)\le 1$.  The  assumption that $n_1=n_2$ implies that the leading coefficient of $S_{\bm n,1}S_{\bm n,2}'$ and $S_{\bm 
			n,1}'S_{\bm n,2}$ match, which proves that $R_{\bm n}^*$ is a constant.
		\end{remark}

Now we are ready to produce the promised differential equations satisfied by the electrostatic partners:
\begin{thm}\label{thm:WronskianS/w}
There exist polynomials $D_1$ and $D_2$ (in general, dependent on ${\bm n}$) such that $S_{\bm n, 1}$  is   solution of the linear differential equation 
		\begin{equation} \label{ODEsForSis}
				A_1A_2P_{\bm n}R_{\bm n}\, y''+((2A_1A_2'+A_1B_2-A_2B_1)P_{\bm n}R_{\bm n}-A_1A_2(P_{\bm n}R_{\bm n}'+P_{\bm n}'R_{\bm n}))y'+D_1y   =0\,, 
		\end{equation}
	and $S_{\bm n, 2}$ satisfies
	\begin{equation} \label{ODEsForSis2}
			A_1A_2P_{\bm n}R_{\bm n}\, y''+((2A_1'A_2-A_1B_2+A_2B_1)P_{\bm n}R_{\bm n}-A_1A_2(P_{\bm n}R_{\bm n}'+P_{\bm n}'R_{\bm n}))y'+D_2y  =0.
	\end{equation}

		If $A_1=A_2$ and $B_1=B_2$, then the two differential equations coincide, i.e., $S_{\bm n, 1}$ and $S_{\bm n, 2}$ are solutions of the same differential equation
		$$P_{\bm n}R_{\bm n}^*y''-(P_{\bm n}'R_{\bm n}^{*}+P_{\bm n}(R_{\bm n}^{*})')y'+D^*y=0\,,$$
		where $R_{\bm n}^*$ was defined in \eqref{defR*}, and $D^*$ is a certain polynomial, dependent on ${\bm n}$.
\end{thm}
\begin{proof}
Notice that, away from the zeros of $A_1$ and $A_2$, the formal identity
	\begin{equation}    	\label{eq:det1}
		\mathfrak{Wrons}[y, f_{\bm n, 1}, f_{\bm n,2} ] =	\det	\begin{pmatrix}
			y &f_{\bm n,1} &f_{\bm n,2} \\
			y'&f_{\bm n,1} '&f_{\bm n,2}  '\\
			y'' &f_{\bm n,1}''&f_{\bm n,2}  ''
		\end{pmatrix}(z)=0
	\end{equation}
	is  satisfied by $y=f_{\bm n,i}$, $i=1,2 $. Expanding the determinant along the first column yields
	$$
	u_2(z) y''(z) - u_1(z) y'(z) + u_0(z) y (z) = 0 ,
	$$
	with
	$$
	u_2  =  \mathfrak{Wrons}[f_{\bm n,1}, f_{\bm n, 2}]= 	d_{\bm n},
	\quad
	u_1   =   \det	\begin{pmatrix}
		f_{\bm n,1} &f_{\bm n,2} \\
		f_{\bm n,1}''&f_{\bm n,2}  ''
	\end{pmatrix}(z)= \mathfrak{Wrons}[f_{\bm n,1}, f_{\bm n, 2}]'= 	d_{\bm n}',
	$$
	and
	\begin{equation}
		u_0 =   \det	\begin{pmatrix}
			f_{\bm n,1} '&f_{\bm n,2}  '\\
			f_{\bm n,1}''&f_{\bm n,2}  ''
		\end{pmatrix}(z)=  \mathfrak{Wrons}[f'_{\bm n,1}, f'_{\bm n, 2}](z).
		\label{f21}
	\end{equation}
	Differentiating \eqref{defS12} we get that for $ i=1, 2$,
	\begin{align*}
		A_i V_i \, f_{\bm n,i}' & = A_i	S_{\bm n, i}' - S_{\bm n, i}\left( A_i'+ B_i \right),  \\
		A_i^2 V_i\,  f_{\bm n,i}'' & = A_i \left(-S_{\bm n, i} \left(A_i''+B_i'\right)+A_i S_{\bm n, i}'' -B_i  S_{\bm n, i}'\right)+\left(2
		A_i'+B_i\right) \left(S_{\bm n, i} \left(A_i' +B_i\right)-A_i S_{\bm n, i}'\right).
	\end{align*}
	This shows that
	$$
	D:= A_1^2A_2^2 v_1 v_2  u_0 =  \det	\begin{pmatrix}
		A_1^2 v_1 \,	f_{\bm n,1} '& A_2^2 v_2\,  f_{\bm n,2}  '\\
		A_1^2 v_1 \,	f_{\bm n,1}''& A_2^2 v_2\,  f_{\bm n,2}  ''
	\end{pmatrix}
	$$
	is a polynomial. 
	
	Thus, we conclude that with this polynomial $D$, functions
	$f_{\bm n,i}$  are two independent solutions of the linear differential equation
	$$
	d_{\bm n}  y''
	-d_{\bm n}  'y'
	+\frac{D}{A_1^2A_2^2 v_1 v_2 }y=0. 
	$$
		
	With the change of variable $y\mapsto y/ v_1 $ and using the definition  \eqref{defS12} we see that $S_{\bm n,1}$  is a solution of the equation
	$$
	g_2(z) y''(z) + g_1(z) y'(z) + g_0(z) y (z) = 0 ,
	$$
	with
	\begin{align*}
		g_2&=\frac{d_{\bm n}}{v_1 }=\frac{R_{\bm n}P_{\bm n}}{A_1 A_2 v_1^2 v_2 },\\
		g_1&=\frac{-d_{\bm n}'}{v_1 }-\frac{2d_{\bm n} v_1'}{v_1^2}
		=-\frac{R_{\bm n}P_{\bm n}}{A_1 A_2 v_1^2 v_2 }\left(\frac{P_{\bm n}'}{P_{\bm n}}+\frac{R_{\bm n}'}{R_{\bm n}}-\frac{2A_1'+B_1}{A_1}-\frac{2A_2'+B_2}{A_2}+2\frac{A_1'+B_1}{A_1}\right),\\
		g_0&=\frac{d_{\bm n}' v_1'}{v_1^2 }+\frac{2d_{\bm n} (v_1' )^2}{v_1^3 }
		-\frac{d_{\bm n} v_1''}{v_1^2 }+\frac{D}{A_1^2 A_2^2v_1^2 v_2} \\
		&=\frac{R_{\bm n}P_{\bm n}}{A_1^2 A_2 v_1^2 v_2}\left(\left(\frac{P_{\bm n}'}{P_{\bm n}}+\frac{R_{\bm n}'}{R_{\bm n}}
		-\frac{2A_2'+B_2}{A_2}\right)(A_1'+B_1)
		-A_1''-B_1'+\frac{D}{A_2}\right),
	\end{align*}
	which shows that each   $A_1^2 A_2^2 v_1^2v_2 g_j$ is a polynomial.   The differential  equation for $S_{\bm n,2}$ is obtained in an analogous way.
	
	Finally, in the case $A_1=A_2$ and $B_1=B_2$ (and as explained in Remark~\ref{remark:A1=A2}, $w_1 =w_2 =w$), both changes of variable coincide, so we have the same linear differential equation of order $2$ with polynomial coefficients for $S_{\bm n,1}$ and $S_{\bm n,2}$. Indeed
	\begin{equation}    	\label{eq:det2}
		\mathfrak{Wrons}[y, S_{\bm n, 1}, S_{\bm n,2} ]  =	\det	\begin{pmatrix}
			y & S_{\bm n,1} &S_{\bm n,2} \\
			y'&S_{\bm n,1} '&S_{\bm n,2}  '\\
			y'' &S_{\bm n,1}''&S_{\bm n,2}  ''
		\end{pmatrix} =0
	\end{equation}
	is  satisfied by $y=S_{\bm n,i}$, $i=1,2 $ and by \eqref{identityRpartCase}, the coefficients of $y''$, $y'$ and $y$ are
	$$\mathfrak{Wrons}[S_{\bm n,1}, S_{\bm n, 2}]=P_{\bm n}R_{\bm n}^*\,,\qquad
	-\mathfrak{Wrons}[S_{\bm n,1}, S_{\bm n, 2}]'=-(P_{\bm n}R_{\bm n}^*)'\,,\qquad
	\mathfrak{Wrons}[S_{\bm n,1}', S_{\bm n, 2}']=D^*\,,$$
	respectively. The statement is proved.
\end{proof}

The differential equations \eqref{ODEsForSis}--\eqref{ODEsForSis2} imply that, respectively, 
\begin{equation*}  
	\begin{split}
		 y''+\left( \frac{2A_2'  }{A_2} -  \frac{ B_1 }{A_1} +  \frac{ B_2 }{A_2}  -  \frac{   R_{\bm n}' }{  R_{\bm n}}  -  \frac{   P_{\bm n}'}{ P_{\bm n}} \right) y' & =0 \quad \text{at the zeros of } 
		 S_{\bm n, 1}, \\
		y''+\left( \frac{ 2A_1'   }{A_1  } + \frac{  B_1   }{A_1 }- \frac{  B_2    }{A_2 }  -   \frac{    R_{\bm n}'  }{  R_{\bm n}} -   \frac{    P_{\bm n}'  }{ P_{\bm n} }  \right) y' &=0 \quad \text{at the zeros of } 
		S_{\bm n, 2},
	\end{split}
\end{equation*}
and the characterization \eqref{eq:condEqBis}, along with definitions \eqref{ratWHP},  yields an electrostatic interpretation of the zeros of its solutions. Recall that $R_{\bm n} \not \equiv 0$ is the polynomial defined in \eqref{defRpolyn} (see alternative expressions in \eqref{IdR0}, \eqref{IdR1}). Then
\begin{cor}\label{cor:Scritical}
	Let $i=1, 2$, and let
	$$
	\phi_1: = \frac{1}{2}\log\left|\frac{P_{\bm n}R_{\bm n}}{A_1 A_2 }\right| + \frac{1}{2}\log\left|\frac{ v_1}{  v_2}\right|, \quad \phi_2(z):= \frac{1}{2}\log\left|\frac{P_{\bm n}R_{\bm n}}{A_1 A_2 }\right| + \frac{1}{2}\log\left|\frac{ v_2}{  v_1}\right|.
	$$
	If for $i\in \{1, 2\}$, the roots of $S_{\bm n,i}$ are simple, then the discrete zero-counting measure $\nu\left( S_{\bm n,i} \right)$ of $S_{\bm n,i}$ is $\phi_i$-critical in the sense of Definition~\ref{defCriticalScalar}. 
\end{cor}

\begin{remark}
As it follows from \eqref{ODEsForSis}--\eqref{ODEsForSis2}, zeros of  $S_{\bm n,i}$  can be multiple only at zeros of $A_1A_2P_{\bm n}R_{\bm n}$. By the definition of $S_{\bm n,i}$  via $\mathfrak D_{w_i}[P_{\bm n}]$ as in \eqref{defTransform2}, if $A_i$ and $B_i$ share a common root (case not excluded by our assumptions) then this root is also a zero of $S_{\bm n,i}$. 
\end{remark}

Notice that the roots of $S_{\bm n,i}$ are in equilibrium in  the external field $\phi_i$ created not only by charges fixed at the zeros of $A_2$ or $A_1$, but also by additional masses, all charged as $-1/2$, placed at the roots of $P_{\bm n}R_{\bm n}$. As $N=n_1 + n_2$ grows large, the dominant interaction is provided by the relation between the zeros of $S_{\bm n,i}$ and $P_{\bm n}$. This motivates to combine the statements of Corollaries \ref{cor:HPinterp} and \ref{cor:Scritical} into a single  electrostatic model that we formulate using the notion of the critical vector measures (see Definition~\ref{defCriticalVector}):
\begin{thm} \label{thm:vectorelectrostatics}
Let $R_{\bm n}\not \equiv 0$ be the auxiliary polynomial of degree $\leq 2\sigma_1+2\sigma_2+3$ defined in \eqref{defRpolyn}. 
	If  the roots of $P_{\bm n}$ and of  $S_{\bm n,1}$ are simple, then 
the discrete vector  measure $\vec \nu_1:=(\nu(P_{\bm n}),\nu(S_{\bm n,1}))$ is a  critical vector  measure for the energy functional $\mathcal{E}_{\vec \varphi, a}$, with $a=-1/2$ and
\begin{equation} \label{defVectorExternalField1}
	\vec \varphi = \left(  \frac{1}{2}\log\left|\frac{1}{v_1 }\right| ,  \frac{1}{2}\log\left|\frac{R_{\bm n} }{A_1 A_2  }\right| +  \frac{1}{2}\log\left|\frac{ v_1}{  v_2}\right|  \right).
\end{equation}

Analogously, if the roots of $P_{\bm n}$ and of  $S_{\bm n,2}$ are simple, then 
the discrete vector  measure $\vec \nu_2:=(\nu(P_{\bm n}),\nu(S_{\bm n,2}))$ is a critical  vector  measure for the energy functional $\mathcal{E}_{\vec \varphi, a}$, with $a=-1/2$ and
\begin{equation} \label{defVectorExternalField2}
\vec \varphi  = \left(  \frac{1}{2}\log\left|\frac{1}{v_2 }\right| ,  \frac{1}{2}\log\left|\frac{R_{\bm n} }{A_1 A_2  }\right| +  \frac{1}{2}\log\left|\frac{ v_2}{  v_1}\right|  \right).
\end{equation}
\end{thm}

Some observations are in order. First, the negative value of the interaction parameter $a=-1/2$ in the electrostatic model above shows that zeros of $P_{\bm n}$ and zeros of $S_{\bm n,i}$ have opposite charges, and thus are mutually attracting. This indicates that in general we should not expect the equilibrium configurations to provide minima for the energy functionals, at least, without additional constraints. 

In the case when $S_{\bm n,i} \equiv \const$, the assertion of the theorem is still valid taking $\nu(S_{\bm n,i})=0$. 

If $A_1=A_2=A$ and  $B_1=B_2$, we have that $v_1=v_2=v$, so same  external field, 
\begin{equation} \label{defVectorExternalField3}
	\vec \varphi = \left(  \frac{1}{2}\log\left|\frac{1}{v }\right| ,  \frac{1}{2}\log\left| R^*_{\bm n} \right|    \right)
\end{equation} 
is acting both on the zeros of $S_{\bm n,1}$ and $S_{\bm n,2}$. Moreover, as observed in Remark \ref{remark:equal} ii), if additionally  
$$
\sigma = \max \{\deg(A)-2, \deg(B)-1 \}=1 \quad \text{and} \quad n_1=n_2,
$$
then $	R_{\bm n}^* $ is a constant. In other words, the second component of  $\vec \varphi $ is zero, so the external field acts only on the zeros of $P_{\bm n}$.

The electrostatic model formulated above becomes meaningful if we complement it with some additional information, such as the localization of the zeros of the participating polynomials (or at least, of the bulk of them). This is impossible in the general case considered so far. Hence, we need to impose additional assumptions on the weights $w_i$'s. This will be carried out in the next section.

\subsection{Some special cases of multiple orthogonal polynomials}\label{sec:special}
\

In this section we will try to make our construction more informative by clarifying the location of the zeros of the electrostatic partners $S_{\bm n, i}$ of $P_{\bm n}$. Many of these results can be predicted by observing that the roots of 
$S_{\bm n, i}$ are de facto critical points of the error function of the Hermite-Pad\'e approximants, see \eqref{S_inew}.

\subsubsection{Angelesco systems}\label{sec:angelesco}
\

The best understood situation when the existence and uniqueness of the Hermite-Pad\'e orthogonal polynomial $P_{\bm n}$ satisfying relations \eqref{defHP} is assured is the so--called Angelesco case, introduced by Angelesco \cite{angelesco} in 1919, and later studied in \cite{MR412503} and in the works of Aptekarev, Gonchar, Kaliaguin, Nikishin, Rakhmanov and others, see e.g.~\cite{MR870267, GRS97, MR1391763, niksor}.  We assume now that $\Delta_1$ and $  \Delta_2$ are real intervals, and 
\begin{equation} \label{caseAngelesco}
\os \Delta_1 \cap \os \Delta_2=\emptyset,
\end{equation}
and $w_1, w_2$ are $\ge 0$ on their supports. Under these conditions, for every multi-index $\bm n=(n_1, n_2)$, polynomial $P_{\bm n}$ is of maximal degree, 
$$
\deg P_{\bm n} = n_1+n_2 =N, 
$$
(in other words,   $\bm n$ is a \textit{normal} index); since this is valid for every $\bm n = (n_1,n_2)$, the system is known as  \textit{perfect}, see   \cite{Mahler}. 

Moreover, in the Angelesco case, $P_{\bm n}$ has exactly $n_1$ and $n_2$ simple zeros in the interiors of $\Delta_1$ and $\Delta_2$, respectively (see \cite[Sect. 5.6]{niksor}). Additionally, the localization of the majority of the zeros of the polynomials $S_{\bm n, i}$ is given in the following result:
\begin{prop}\label{cor:zerosS12ang}
	Polynomial $S_{\bm n, 1}$ (respect. $S_{\bm n, 2}$) has $n_2-1$ (respect. $n_1-1$) zeros interlacing with those of $P_{\bm n}$ on $\Delta_2$ (respect. $\Delta_1$).
\end{prop}
\begin{proof}
Let us  write
	\begin{equation} \label{defFactorization}
		P_{\bm n}(x) = P_{\bm n,1}(x)\,P_{\bm n,2}(x),
	\end{equation}
where $P_{\bm n,i}$ is the monic polynomial whose zeros agree with those of $P_{\bm n}$ on $\Delta_i$, $i=1,2$.  	
	Taking $T = P_{\bm n,i}$ in \eqref{eq:CauchTrEq}, we conclude that
	\begin{equation*}
  \mathfrak C_{w_i}[P_{\bm n}] = \frac{1}{P_{\bm n,i}} \, \mathfrak C_{w_i}[P_{\bm n,i} \, P_{\bm n}] .
	\end{equation*}
Since the integrand in $\mathfrak C_{w_i}[P_{\bm n,i} \, P_{\bm n}] $  preserves sign on $\Delta_i$, it implies that for $i=1, 2$, the Cauchy transform $\mathfrak C_{w_i}[P_{\bm n}]$ has no zeros in $\R \setminus \Delta_i$.

It remains to apply Lemma~\ref{propInterlacing1} with $S=S_{\bm n, i}$. 
\end{proof}

Recall that by Theorem~\ref{TeoClaveHP}, $S_{\bm n, i}$  are polynomials   of degree at most $N-n_i+\sigma_i $, $i=1, 2$. Proposition~\ref{cor:zerosS12ang} shows that  all zeros of $S_{\bm n, i}$, $i=1,2$, except for at most  $\sigma_i+1$ of them (amount only depending on the classes of the weights), are well localized by this interlacing property.

Thus, the electrostatic model for the zeros of $P_{\bm n}$, stated in Corollary~\ref{cor:HPinterp}, is as follows. If we consider each one of the $n_1$ zeros of $P_{\bm n}$ on $\Delta_1$ as a positive unit charge, then they are in equilibrium (or more exactly, they are in critical configuration) in the field generated by:
\begin{itemize}
\item the repulsion of the unit positive charges placed at rest of the zeros of $P_{\bm n}$, on $\Delta_2$, 
\item the attraction of the zeros of $S_{\bm n, 1}$, this time with charge $-1/2$, all except at most $\sigma_1+1$ of them interlacing with the zeros of $P_{\bm n}$ on $\Delta_2$, and
\item the background potential from the orthogonality weight $w_1$ on $\Delta_1$. 
\end{itemize} 
A symmetric picture is obviously valid on $\Delta_2$.

Angelesco--Jacobi polynomials constitute an example of an Angelesco system. They are considered in detail in Section \ref{sec:AngelescoJacobi}.

\subsubsection{AT systems and generalized Nikishin systems}\label{sec:nikishin}

\

We assume now that, unlike the Angelesco case,  both weights $w_i\ge 0$, $i=1, 2$, are supported on the same interval:
\begin{equation} \label{assumptionNikishin1}
	\Delta_1 = \Delta_2 = \Delta = [a,b] \subset \R. 
\end{equation}
These two weights form an algebraic Chebyshev system (or an \textit{AT system}) for the multi-index $\bm n =(n_1,n_2)$ if  
\begin{equation} \label{system}
\left\{w_{1}(x), x w_{1}(x), \dots, x^{n_{1}-1} w_{1}(x), w_{2}(x), x w_{2}(x), \dots, x^{n_{2}-1} w_{2}(x) \right\}
\end{equation}
is a Chebyshev system on $\Delta$, that is, if  every non-trivial linear combination of functions from \eqref{system}  with real coefficients has at most $N=n_1+n_2$ zeros on  $\Delta$.   For further details, see e.g. \cite[Chapter 4, \S 4]{niksor} or \cite[Section 23.1.2]{Ismail05}. 
 
It is known (see e.g.  \cite[Theorem 23.1.4]{Ismail05}) that if the multi-index $\bm n =(n_1,n_2)$ is such that $\left(w_{1},   w_{2}\right)$ is an AT system on $[a, b]$ for every index $\bm m =(m_1,m_2) $ for which $m_{j} \leq n_{j}$,  $j=1, 2$,  then 
$\bm n$ is normal, and the multiple orthogonal polynomial $P_{\bm n}$  has all its $N$ zeros,  all simple, on $(a, b)$.

A construction of an AT  system that now is known to be perfect (see \cite{FiLo}) was put forward by E. M. Nikishin \cite{Nik82}; it is called an \textit{$MT$-system} in \cite{niksor}, but is nowadays known as a \textit{Nikishin system}.  Namely, we assume additionally  that the ratio $w_2/w_1$ is the Cauchy transform (also known as a \textit{Markov function}) of a non-negative function on an interval $[c,d]\subset \R$, whose interior is disjoint with $\Delta$. Besides normality for all multi-indices $\bm n$ and that all zeros of $P_{\bm n}$ are simple and belong to the open interval $(a,b)$, this allows localization of zeros of the electrostatic partner, as we show below.

\begin{remark}
Nikishin \cite{Nik82} proved the normality for all indices of the form $(n,n)$ and $(n+1,n)$,  asserting without proof that it holds for any index $(n_1,n_2)$ such that $n_1 \geq n_2$. He called it a \textit{weakly perfect} system, but a result for Markov functions (see e.g. \cite[Lemma 6.3.5]{StTo}) implies that  weak perfectness is equivalent to  perfectness of the system. Later, K. Driver and H. Stahl established the normality for any index in the case of Nikishin systems of two functions \cite{DrSt} (see also \cite{Bulo}), and more recently, U. Fidalgo and G. L\'{o}pez proved perfectness of a Nikishin system of any order \cite{FiLo}. 
\end{remark}

As before, we restrict our attention to semiclassical weights, but slightly weaken  Nikishin's assumptions. Namely, in the situation \eqref{assumptionNikishin1} we suppose that $w_1$, $w_2$ are non-negative weights on $[a,b]$ such that:
\begin{itemize}
\item $w_1$ is a semiclassical weight on $[a,b]$;

\item weight $w_2$ is of the form
\begin{equation} \label{defW2}
w_2(x)= | \Pi(x) u(x)| w_1(x), \quad x\in [a, b],
\end{equation}
where 
\begin{equation}\label{Cauchytrans}
	u(x) = \int_c^d\,\frac{U(t)\, }{x-t}\,dt
\end{equation}
is semiclassical,  with $(a,b) \cap (c,d) = \emptyset$,  $U$ continuous and non-negative on $(c,d)$, and $\Pi$ is an arbitrary polynomial with real coefficients,  non-vanishing on $(a,b) \cup (c,d)$.
\end{itemize}
Under these assumptions the weight $w_2$ is also semiclassical. The fact that $(w_1, w_2)$ forms an AT-system for $n_1 \geq n_2 + m$ can be deduced from the fact that the linear form $p + q \Pi u$, for arbitrary polynomials $p, q$ of respective degrees $\leq n_1-1, n_2-1$, and $n_1 \geq n_2 + m$, has at most $n_1+n_2-1$ zeros in $[c,d]$ (see \cite{Nik82} and  \cite[p. 1022]{RHPNik}).

\begin{example}
It is easy to see that if $c<d$, then for $\gamma, \delta\notin \mathbb Z$, $\gamma+\delta<0$, $\gamma+\delta\in \mathbb Z$, function
$$
u(x)=(x-c)^\gamma (d-x)^\delta
$$
can be represented as the Cauchy integral \eqref{Cauchytrans}. As a consequence, a  pair of weights 
$$
w_1(x)= |x-a|^{\alpha  }\,|x-b|^{\beta  }, \quad w_2(x)= |x-a|^{\alpha  }\,|x-b|^{\beta  } |x-c|^{\gamma  }\,|x-d|^{\delta  }, \quad x \in (a, b), 
$$
where $(a,b) \cap (c,d) = \emptyset$, and $\alpha, \beta, \gamma, \delta>-1$,  $\gamma, \delta\notin \mathbb Z$, $\gamma+\delta\in \mathbb Z$, constitute an example 
of a system defined above.

Since we do not assume that the intervals $(a,b)$ and $(c,d)$ are bounded, another example is the pair of weights of the form 
$$
w_1(x) = \exp (- x^{r}),   \quad w_2(x)= |x-a|^{\gamma  }\,|x-b|^{\delta}\exp (- x^{r}), \quad x\in [0,+\infty),  
$$
with $r\in \N$, $-\infty<a<b<0$, $ \gamma, \delta>-1$,  $\gamma, \delta\notin \mathbb Z$, and $\gamma+\delta\in \mathbb Z$.
\end{example}

With the introduction of the polynomial factor $\Pi$ in the weight $w_2$ we can no longer guarantee that all the zeros of $P_{\bm n}$ are in $(a,b)$; however, the following result still holds:
\begin{prop}\label{thm:zerosNik}
	Under the conditions on the weights $w_1$ and $w_2$ stated above, the Hermite--Pad\'e polynomial $P_{\bm n}$, satisfying \eqref{defHP},  has at least $n_1+\ell+1$ sign changes on $(a,b)$, where
	$$\ell = \min (n_2-1,n_1-m), \quad m=\deg(\Pi),
	$$
	while its Cauchy transform $\mathfrak C_{w_1}[P_{\bm n}]$ has at least $\ell+1$  sign changes in $(c,d)$.
\end{prop}
\begin{proof}
	We basically follow the arguments in  \cite[Sec. 2.5]{VanAssche06}. Suppose that $n_1\geq m$. Taking in \eqref{eq:CauchTrEq} $T  = P \Pi $, where $P$ is and arbitrary polynomial of degree $k\leq n_1-m$, we get that
	$$
	P(x)\, \Pi(x) \,\int_a^b\,\frac{P_{\bm n}(t)}{t-x}\,w_1(t) dt =  \int_a^b\,\frac{P(t)\,P_{\bm n}(t)\Pi(t)}{t-x}\,w_1(t) dt.
	$$ 
	Integrating this identity with respect to $d\sigma$ in $[c,d]$, applying Fubini's theorem and using \eqref{Cauchytrans} yields 
	$$
	\int_c^d\,P(x)\,\mathfrak C_{w_1}[P_{\bm n}](x)\,\Pi(x) v(x) \,d x = \,\int_a^b\,P(t)\,P_{\bm n}(t)\,w_2(t) dt\,= 0\,,
	$$
	as long as the degree $k$ of $P$ is $\leq n_2-1$. Since polynomial $\Pi(x)$ does not vanish in $(c,d)$ it proves that $\mathfrak C_{w_1}[P_{\bm n}]$ changes sign at least $\ell + 1$ times in $(c,d)$.
	
	To prove the first part of the proposition, let $P$ be a polynomial non vanishing on $(a,b)$ and such that $\mathfrak C_{w_1}[P_{\bm n}]/P$  is analytic in $\C \setminus [a,b]$. By the assertion we just proved, we can take $P$ of degree at least $\ell+1$, so that by   \eqref{eq:asymptCauchy},  
	$$
	\frac{\mathfrak C_{w_1}}{P}(z)=\mathcal O \left(\frac{1}{z^{n_1+\ell+2}}\right), \quad z\rightarrow \infty.
	$$ 
	
Let  $\Gamma$ be a  positively oriented Jordan contour encircling $[a,b]$ and leaving $[c,d]$ in its exterior. Then for $k=0, 1, \dots,  n_1+\ell$,
		\begin{equation}\label{fubini}
		\begin{split}
		0=	\frac{1}{2\pi i}\,\oint_{\Gamma}\,z^k\,\frac{\mathfrak C_{w_1}[P_{\bm n}](z)}{P(z)}\,dx & = \frac{1}{2\pi i}\,\oint_{\Gamma}\,\frac{z^k}{P(z)}\,\left( \int_a^b\,\frac{P_{\bm n}(t) }{t-z}\,w_1(t) dt\right) \,dz \\
			& = \int_{a}^b \, P_{\bm n}(t)  \left(  \frac{1}{2\pi i}\,  \oint_\Gamma \, \frac{z^k}{P(z)}\, \frac{1}{t-z}\,dz \right) w_1(t)  \,dt \\
			&  = \int_a^b\,t^k\,P_{\bm n}(t)\,\frac{w_1(t) dt}{P(t)}, 
		\end{split}
	\end{equation}
where we have used Fubini's and Cauchy's theorems. Since  $w_1/P$ has a constant sign on $(a,b)$,  $P_{\bm n}$ satisfies  quasi--orthogonality conditions (of order at least $n_1+\ell$) there. Standard arguments yield that $P_{\bm n}$ has at least $n_1+\ell+1$ sign changes on $(a,b)$.
\end{proof}

Consider the case when $n_2\leq n_1-m+1$, so that $\ell =n_2-1$. According to Proposition~\ref{thm:zerosNik}, $P_{\bm n}$ has exactly $N=n_1+n_2$ zeros, all simple, in $(a,b)$, while $\mathfrak C_{w_1}[P_{\bm n}]$ has $\ge n_2$ zeros in $(c,d)$, exactly as in the classical Nikishin setting ($m=0$).
Moreover, if these zeros of $\mathfrak C_{w_1}[P_{\bm n}]$ are disjoint with the zeros of $A_1 S_{\bm n, 1}$ then, by the second assertion of Proposition~\ref{propInterlacing1}, at least $\ell=n_2-1$ zeros of polynomial $S_{\bm n, 1}$ (out of a total of $\le n_2+\sigma_1$ of its zeros) interlace with those of $\mathfrak C_{w_1}[P_{\bm n}]$.

\subsubsection{Generalized Nikishin systems: case of overlapping supports}\label{sec:rakhmanov}

\

Generalized Nikishin systems (GN systems) were introduced in \cite{GRS97} using a rooted tree graph. A particular example of such a system, whose asymptotics was studied in \cite{MR2475084}, \cite{MR2656323} and \cite{MR2796829},  shares characteristics of both cases described in Sections~\ref{sec:angelesco} and \ref{sec:nikishin}. Namely, we assume that
\begin{equation} \label{mixedR1}
	\Delta_1 \subseteq \Delta_2,
\end{equation}
in addition to the Nikishin  relation between the non-negative semiclassical weights $w_1$ and $w_2$, given by conditions \eqref{defW2}--\eqref{Cauchytrans}, with $\Pi\equiv 1$, and the assumption that the interior of
$$
\Delta_3:=[c,d] 
$$
is disjoint with $\Delta_2$. On one hand, when $\Delta_1$  extends to the whole $ \Delta_2$, we obtain the classical Nikishin configuration of Section~\ref{sec:nikishin}.  On the other,  if $\Delta_1$ and $\Delta_2$ share an endpoint then redefinition of the support $\Delta_2$ into $\Delta_2\setminus\Delta_1 $ 
yields the Angelesco setting of Section~\ref{sec:angelesco}.

Let us study the diagonal case of $n_1=n_2=n$, so that $N=2n$. As before, the key ingredient is   the location of the zeros of the Hermite--Pad\'e polynomial $P_{\bm n}$ and its electrostatic partners $S_{\bm n,1}$ and  $S_{\bm n,2}$. It was proved in \cite{MR2796829} that \textit{for any $n$, the zeros of the Hermite--Pad\'e polynomial $P_{\bm n}$, with the possible exception of five of them, are in $\Delta_2$.} Recall that by orthogonality assumptions, at least $n$ of them belong to the subinterval $\Delta_1$. Additionally, we have:

\begin{prop}\label{prop:Rakhzeros}
		If $P_{\bm n}$ has $n+k_1\ge n$ sign changes in $\Delta_1$ and $k_2\ge 0$ sign changes in $\Delta_2 \setminus \Delta_1$,  then $S_{\bm n,1}$ has at least $\max \{k_2-2,0\}$ zeros in $\Delta_2 \setminus \Delta_1$, which interlace with the zeros of $P_{\bm n}$, and at least $\max\{k_1-3,0\}$ zeros in $\Delta_3$, interlacing with the zeros of $\mathfrak C_{w_1}[P_{\bm n}]$.
\end{prop}
\begin{proof}
	Let us denote by $x_i$, $i=1,\ldots,n+k_1$ the points of sign change of $P_{\bm n}$  in $\Delta_1$, and by $y_j$, $j=1,\dots,k_2$ the corresponding points of sign change  in $\Delta_2 \setminus \Delta_1$. Using  \eqref{eq:CauchTrEq} with
	$$
	\Pi(x) =\prod_{j=1}^{k_2} (x-y_j),
	$$
	we conclude again that $\mathfrak C_{w_1}[P_{\bm n}]$ does not change sign in each component of $\Delta_2 \setminus \Delta_1$. Then, the first part of Lemma \ref{propInterlacing1} asserts that $S_{\bm n,1}$ has at least $k_2-2$ zeros in $\Delta_2 \setminus \Delta_1$, interlacing with those of $P_{\bm n}$ (observe that $\Delta_2 \setminus \Delta_1$ can have up to two disjoint components).
	
Notice that $k_1 + k_2 \le n$, so that if $k_2=n$, our proof is finished. Suppose that $k_2<n$, and let us see that $\mathfrak C_{w_1}[P_{\bm n}]$ has at least $k_2-2$ sign changes  in $\Delta_3$.
	From \eqref{defHP} and \eqref{defW2}--\eqref{Cauchytrans}, we have that for any polynomial $P \in \mathbb P_{n-1}$,
	\begin{equation}\label{Rakhorth}
		\begin{split}
			0 & = \int_{\Delta_1}\,P(x) P_{\bm n} w_1(x) u(x) dx + \int_{\Delta_2 \setminus \Delta_1}\,P(x) P_{\bm n} w_2(x) dx \\
			& = - \int_{\Delta_3}\,\pi(t) \mathfrak C_{w_1}[P_{\bm n}](t) \sigma (t) dt +  \int_{\Delta_2 \setminus \Delta_1}\,P(x) P_{\bm n} w_2(x) dx\,,
		\end{split}
	\end{equation}
	where we have used Fubini's theorem for the last identity. Now, denote by $z_i\,,\,i=1,\dots,k_3$ the points where $\mathfrak C_{w_1}[P_{\bm n}]$ changes sign in $\Delta_3$. Let us suppose that $k_3<k_1-2$ and define 
	$$
	P(x) := (x - \zeta_1)^{\epsilon_1} (x - \zeta_2)^{\epsilon_2} \prod_{i=1}^{k_2} (x-y_i) \prod_{j=1}^{k_3} (x-z_j)\,,
	$$ 
	where $\epsilon_i \in \{0,1\}\,,\,i=1,2\,,$ $\zeta_1 \in \Delta_1$ and $\zeta_2$ is located in the ``gap'' between $\Delta_2$ and $\Delta_3$ (which may consist of a single point, since  we only require  the interiors to be disjoint). We can use the parameters $\zeta_1, \zeta_2$ and $\epsilon_1, \epsilon_2$ to guarantee that
	$$P(t) \mathfrak C_{w_1}[P_{\bm n}](t) \sigma (t) \geq 0,\quad t\in \Delta_3$$ and
	$$P(x) P_{\bm n}(x) w_2(x) \leq 0,\quad  x\in \Delta_2.$$
	Since by assumption, $\deg(P) = k_2+k_3+2 < k_1+k_2 \le n$, so that  \eqref{Rakhorth} should hold for this particular choice of $P$. This is possible only if both integrands in the right hand side of \eqref{Rakhorth} were identically $0$, which is a contradiction. Hence, $k_3\geq k_1-2$, and it remains to use again the second assertion in Lemma \ref{propInterlacing1} to conclude the proof.
\end{proof}

Thus, as expected from an intermediate case between Angelesco and Nikishin settings, now a part of the zeros of the electrostatic partner $S_{\bm n,1}$ lie on $\Delta_3$ (as in the Nikishin case) while the rest are located in $\Delta_2 \setminus \Delta_1$ (Angelesco). Therefore, in this situation, part of the ``ghosts'' attractive charges interlace with part of the zeros of $P_{\bm n}$ in $\Delta_2 \setminus \Delta_1$, while other part are placed in $\Delta_3$. Of course, depending on the specific case we are dealing with, some of these sets of attractive charges may be empty.

\section{Asymptotic zero distribution}\label{sec:asymptotics}

It is instructive to observe the discrete-to-continuous transition of the electrostatic model described in the previous section, assuming that the total degree  $N=n_1+n_2\to \infty$. 

\subsection{Vector  critical measures}

  If $\mu_1$, $\mu_2$ are two  finite positive Borel measures with compact support  then their (continuous) mutual logarithmic energy is
\begin{equation}\label{defMutual}
	\langle \mu_1, \mu_2 \rangle  :=  \iint \log \frac{1}{|x-y|}\, d\mu_1(x) d\mu_2(y);
\end{equation}
 and the logarithmic energy of $\mu_i$ is
\begin{equation}\label{defEnergyContinuous}
	E(\mu_i) :=  \langle \mu_i, \mu_i \rangle, \quad i=1, 2. 
\end{equation}

Given a vector of measures $\vec \mu=(\mu_1, \dots, \mu_k)$, and a symmetric positive-definite \textit{interaction matrix}
$$
M=(m_{ij})_{i,j=1}^k,
$$
the  vector energy of $\vec \mu $ is
\begin{equation}\label{defWeightedEnergyCont}
	E_{  M}(\vec \mu  ):=
	\sum_{i,j=1}^k m_{ij} \langle \mu_i, \mu_j \rangle.
\end{equation}
In the particular case of $k=2$, when 
$$M =\begin{pmatrix}1&a\\a&1\end{pmatrix}, \quad -1<a<1, $$
we call $a$ the \textit{interaction parameter}.  

We define the critical vector measures following \cite{MR2770010} and \cite{MR3545949}. Recall that any smooth complex-valued function $h$ in the closure  $\overline \Omega $ of a domain $\Omega$ generates a local variation of $\Omega$ by $ z
\mapsto z^t=z+ t \, h(z)$, $t\in \C$. It is easy to see that $ z \mapsto z^t $ is injective for small
values of the parameter $t$. The transformation above induces a variation of sets $ e \mapsto
e^t := \{z^t:\, z \in e\}$, and  measures: $ \mu \mapsto
\mu^t$, defined by $\mu^t(e^t)=\mu(e) $; in the differential form, the pullback measure $\mu^t$
can be written as $d\mu^t (x^t)=d\mu(x)$. 	Recall also the notation, introduced in Section~\ref{sec:generalHermitePade},  of $\mathbb A_i$ for the set of zeros of $A_i$ on $\C$, $i=1,2$.
\begin{definition}[see \cite{MR3545949}]  \label{def:AcriticalCont}
	A vector measure $\vec \mu = (\mu_1, \dots, \mu_k)$ is a (continuous) \textit{critical vector  measure} if for any $h$ smooth in $\C\setminus (\mathbb A_1 \cup \mathbb A_2)$ such that $h\big|_{\mathbb A_1 \cup \mathbb A_2} \equiv 0$,
	\begin{equation}\label{derivativeEnergy}
		\frac{d}{dt}\, E_M (\vec \mu^t)\big|_{t=0} = \lim_{t\to 0} \frac{E_M (\vec \mu^t)- E_M (\vec \mu)}{t}=0.
	\end{equation}
\end{definition}
As in the discrete case, 
\begin{equation*}
		\mu_i \text{ is $F_i$-critical, with } F_i  :=\sum_{1\le j\le k,\, j\neq i} \frac{m_{ij}}{m_{ii}}\,  U^{\mu_j}  ,  \quad i\in \{ 1, \dots, k\},
\end{equation*}
which yields the following variational conditions on the components of $\vec \mu$: for $x\in \supp(\mu_i)$, with a possible exception of a subset of logarithmic capacity $0$, 
\begin{equation} \label{variationalCond}
 	U^{\mu_i}(x) + \sum_{1\le j\le k,\, j\neq i} \frac{m_{ij}}{m_{ii}}\,  U^{\mu_j}  = c_i =\const ,  \quad i\in \{ 1, \dots, k\}. 
\end{equation}

\subsection{Asymptotic electrostatic model: general case}

Under a general assumption that for each multi-index $\bm n=(n_1, n_2)\in \N^2$ the Hermite--Pad\'e polynomial $P_{\bm n}$ exists and is of degree $N=n_1+n_2$, as well as that $\deg (S_{\bm n, i})=N-n_i+1$, we consider the  zero-counting measures (see the definition in \eqref{defCountingMeaure}) for $P_{\bm n}$, $S_{\bm n,1}$ and $S_{\bm n,2}$ in the asymptotic regime
\begin{equation} \label{asymptoticRegime}
	N=n_1 + n_2 \to \infty, \quad \lim_{N\rightarrow \infty}\,\frac{n_2}{N}\,= t\in [0,1].
\end{equation}
Let us assume that (perhaps, along a subsequence of multi-indices), weak limits
\begin{equation} \label{munu}
\mu := \lim_{\bm n} \frac{1}{N}\,\nu(P_{\bm n}),\quad \nu_1 := \lim_{\bm n} \frac{1}{N}\,\nu(S_{\bm n,1}), \quad \nu_2 := \lim_{\bm n} \frac{1}{N}\,\nu(S_{\bm n,2}),
\end{equation}
exist. With our assumptions, 
\begin{equation} \label{weightconstraints1}
\|\mu\|:= \int d\mu  =1, \quad \|\nu_1\|:= \int d\nu_1 =t, \quad \|\nu_2\|:= \int d\nu_2 =1-t. 
\end{equation}
 
Since every weak-* limit of a sequence of discrete critical vector  measures is critical in the sense of Definition~\ref{def:AcriticalCont} (a proof for the scalar case can be found  in \cite[Theorem 7.1]{MR2770010}; it applies to the vector equilibrium without substantial modifications), we obtain: 
\begin{cor}\label{cor:asymptoticsGeneral}
	With the assumptions and notations above, each vector measure $(\mu,\nu_1)$ and $(\mu,\nu_2)$  is critical in the sense of Definition~\ref{def:AcriticalCont}, with the interaction parameter $ a=-1/2$ and constraints \eqref{weightconstraints1}. 
\end{cor}

It is worth also discussing a connection with a more traditional model involving 3-component vector measures. Define
\begin{equation} \label{lambdas}
\Omega=\overline{\{ x\in\R:\, \mu=\nu_1 \}}, \quad \lambda_1 :=\mu\big|_{\Omega}, \quad \lambda_2=\mu-\lambda_1, \quad \lambda_3= \nu_1-\lambda_1.
\end{equation}
Obviously,
\begin{equation} \label{weightconstraints2}
	\|\lambda_1\| +   \|\lambda_2\|= 1, \quad   \|\lambda_1\| +   \|\lambda_3\|= t.
\end{equation}
Variational conditions \eqref{variationalCond} for $(\mu, \nu_1)$ with the interaction matrix
$$ 
\begin{pmatrix}1&-1/2\\-1/2&1\end{pmatrix}
$$ 
imply that
\begin{equation*}  
	\begin{split}
		U^{\mu}(x) -\frac{1}{2}\,  U^{\nu_1}(x)  = c_1 =\const , &  \quad x\in \supp(\mu), \\ 
		U^{\nu_1}(x)  -\frac{1}{2}\, U^{\mu}(x)  = c_2 =\const ,  &  \quad x\in \supp(\nu_1),
	\end{split}
\end{equation*}
or equivalently,
\begin{equation*}  
	\begin{split}
	\frac{1}{2}\, 	U^{\lambda_1}(x) +U^{\lambda_2}(x) -\frac{1}{2}\,  U^{\lambda_3}(x)  = c_1 =\const , &  \quad x\in \supp(\mu)= \supp(\lambda_1) \cup \supp(\lambda_2), \\ 
		\frac{1}{2}\, U^{\lambda_1}(x) -\frac{1}{2}\, U^{\lambda_2}(x) + U^{\lambda_3}(x)  = c_2 =\const , &  \quad x\in  \supp(\nu_1) =  \supp(\lambda_1) \cup \supp(\lambda_3). \\ 
	\end{split}
\end{equation*}
Additionally, on $\supp(\lambda_1)$, where both identities hold, we have that
$$
	U^{\lambda_1}(x) +\frac{1}{2}\, U^{\lambda_2}(x) +\frac{1}{2}\,  U^{\lambda_3}(x)  = c_1 +c_2.
$$
\begin{cor}\label{cor:asymptotics3}
	With the assumptions and notations above,  $(\lambda_1, \lambda_2, \lambda_3)$ is a critical vector  measure satisfying the constraints  \eqref{weightconstraints2} and with  the interaction matrix
$$ 
\begin{pmatrix}
	1&1/2 & 1/2 \\1/2&1 & -1/2 \\
1/2 & -1/2 & 1
\end{pmatrix}.
$$
\end{cor}
This electrostatic model has  been used to describe the asymptotics of the zeros of Hermite--Pad\'e polynomials in several situations, see e.g. \cite{MR2475084, ALT2011,  MR3545949, MR3939592, AMFS21, MR2796829}. In particular, a spectral curve for such critical measures was derived in \cite{MR3545949}, and it was shown that $\lambda_i$'s are supported on a finite number of analytic arcs, that are trajectories of a quadratic differential globally defined on a three-sheeted Riemann surface.

\subsection{Asymptotic electrostatic model: Angelesco case} \label{sec:AsymptAngelesco}

Using the notation \eqref{munu} and Proposition~\ref{cor:zerosS12ang}, in this case 
$$
\supp(\mu) \subseteq \Delta_1 \cup \Delta_2, \quad \supp(\nu_1)\subseteq \Delta_2, \quad \supp(\nu_2)\subseteq \Delta_1, 
$$
and
$$
\nu_1 = \mu\big|_{\Delta_2}, \quad \nu_2 = \mu\big|_{\Delta_1},
$$
so that in notation  \eqref{lambdas}, 
$$
\lambda_1 = \nu_1, \quad \lambda_2 = \nu_2, \quad \lambda_3=0.
$$
Thus, in this case the electrostatic vector  model of Corollary~\ref{cor:asymptotics3} reduces to a $2\times 2$ equilibrium for $(\nu_1, \nu_2)$, with the interaction matrix
$$
M =\begin{pmatrix}
	1 & 1/2 \\
	1/2 & 1
\end{pmatrix} 
$$
and constraints 
$$
\supp (\nu_1) \subset \Delta_1, \quad \|\nu_1\|=t, \qquad \text{and} \qquad \supp (\nu_2) \subset \Delta_2, \quad  \|\nu_2\|=1-t.
$$
This is already classical, see e.g.~\cite[Ch. 5]{niksor};  actually, a stronger result is valid: the vector measure $(\nu_1,\nu_2)$ is a \textit{global minimum} for \eqref{def:EnergiaVect} and $a=1/2$. This does not follow directly from our electrostatic model.

Notice that measure $\nu=\nu_1 + \nu_2$ is the limit zero distribution of both $P_{\bm n}$ and $S_{\bm n,1}S_{\bm n,2}$. 

\begin{remark}
Although in the Angelesco case the zeros of  $P_{\bm n}$ are confined in $\Delta_1 \cup \Delta_2$, in principle up to $\sigma_i+1$ of $S_{\bm n, i}$, $i=1, 2$, are out of our control. Same happens with (a bounded number of) zeros of the polynomials $R_{\bm n}$. In order to guarantee weak convergence of the zero-counting measures, it is sufficient to impose an additional assumption: that \textit{the zeros of $S_{\bm n, 1}$, $S_{\bm n, 2}$ and $R_{\bm n}$ are uniformly bounded along the double sequence $(n_1, n_2)$}. 
\end{remark}

\subsection{Asymptotic electrostatic model: Nikishin case} \label{sec:AsymptNikishin}. 

Consider the generalized Nikishin system as described in Section~\ref{sec:nikishin}, in the asymptotic regime \eqref{asymptoticRegime} and with the additional assumption that
\begin{equation} \label{asymptoticRegimeNikishin1}
	n_2\leq n_1-m+1, 
\end{equation}
so that
\begin{equation} \label{asymptoticRegimeNikishin2}
t = \lim_{N\rightarrow \infty}\,\frac{n_2}{N} \in [0,1/2].
\end{equation}
As we have seen, all $N$ zeros of $P_{\bm n}$ live on $[a,b]$; according to Theorem~\ref{thm:vectorelectrostatics}, each one of them, endowed with a charge $+1$, interacts with zeros of $S_{\bm n, 1}$, each one with charge $-1/2$. We have seen that the majority of them (at least $\ell=n_2-1$ out of  $ n_2+\sigma_1$ possible) belongs to $[c,d]$. Imposing the same additional assumption that before, that \textit{the zeros of both  $S_{\bm n, 1}$ and $R_{\bm n}$ are uniformly bounded along the double sequence $(n_1, n_2)$}, we can use again the weak-* compactness of measures.

With the notation \eqref{munu}, 
$$
\supp(\mu) \subseteq \Delta_1 =[a,b], \quad \supp(\nu_1)\subseteq \Delta_2=[c,d],
$$
so that in notation  \eqref{lambdas}, 
$$
\lambda_1 = 0, \quad \lambda_2 = \mu, \quad \lambda_3=\nu_1.
$$
Thus, in this case the  electrostatic vector model of Corollary~\ref{cor:asymptotics3} reduces to a $2\times 2$ equilibrium for $(\mu, \nu_1)$, with the interaction matrix
$$
M =\begin{pmatrix}
	1 & -1/2 \\
	-1/2 & 1
\end{pmatrix} 
$$
and constraints 
$$
\supp (\mu) \subset \Delta_1, \quad \|\mu\|=1, \qquad \text{and} \qquad \supp (\nu_1) \subset \Delta_2, \quad  \|\nu_2\|=1-t.
$$
This model was initially put forward by Nikishin himself in \cite{Nikishin86}. Again, a stronger result is valid: the vector measure $(\mu,\nu_1)$ is a \textit{global minimum} for the vector energy, which does not follow directly from our electrostatic model.

Observe that the roles of the weights $w_1$ and $w_2$ in a Nikishin system are not symmetric, or at least the symmetry is not immediate. An argument that allows to swap $w_1$ and $w_2$, and thus break the barrier of  $n_2\leq n_1$, is based on the fact (see e.g.~\cite[Lemma 6.3.5]{StTo}) that if $u$ is a Markov function \eqref{Cauchytrans}, then 
$$\frac{1}{u(x)} = r(x) - \int\,\frac{d\tau (t)}{x-t},$$
where $r$ is a polynomial of degree $\leq 1$ and $\tau$ is a positive measure on $[c,d]$.  Standard arguments allow to extend the previous result when  $m=0$  (the classical Nikishin case), see e.g.  \cite{RHPNik}. Unfortunately, the presence of a non-trivial polynomial factor $\Pi$ in  \eqref{defW2} prevents these arguments from going through. Thus, if $m=0$, in the asymptotic regime \eqref{asymptoticRegime} we can drop the restriction \eqref{asymptoticRegimeNikishin1}, and the electrostatic model discussed above is still valid. Another interesting result that sheds light on the roles of $w_1$ and $w_2$, allowing to connect the situations of $n_1\le n_2$ and $n_1\ge n_2$, appears in \cite{zbMATH07308461}. 

\subsection{Asymptotic electrostatic model:  case of the overlapping supports}

Here we consider the asymptotic electrostatics for the intermediate case studied in \cite{MR2475084,  MR2796829} and partially analyzed in Section~\ref{sec:rakhmanov}. Recall that here we consider $n_1=n_2=n$, and thus, we are interested in what happens as $n\rightarrow \infty$.

With our notation in the current section, we have that $\supp \mu \subseteq \Delta_2$ and $\supp \nu_1 \subseteq (\Delta_2 \setminus \Delta_1) \cup \Delta_3$, in such a way that now we have,
$$\supp \lambda_1 \subseteq \Delta_2 \setminus \Delta_1\,,\, \supp \lambda_2 \subseteq \Delta_1\,,\, \supp \lambda_3 \subseteq \Delta_3\,,$$
The results in Proposition \ref{prop:Rakhzeros} guarantee that
$$\|\lambda_1\| + \|\lambda_2\| = 1\,,\; \|\lambda_2\| = \|\lambda_3\| + \,\frac{1}{2}\,. $$
Since in this intermediate case none of the measures $\lambda_i$ becomes null, the $3\times 3$ interaction matrix in Corollary 6.3 does not reduce to a $2\times 2$ one, as in the previous (extremal) cases. As for the constraints on the size of the measures, in this case we have 
$$\|\lambda_1\| = \,\frac{1}{2} - \theta \,,\; \|\lambda_2\| = \,\frac{1}{2} + \theta\,,\; \|\lambda_3\| = \theta\,,$$
where $\theta \in [0, 1/2]$ is a parameter which depends on the relative sizes and mutual positions of the three intervals $\Delta_i\,,\,i=1,2,3$; but especially on the first two ones, as asserted by the author of \cite{MR2796829}.

Moreover, as in the previous cases, for this ``critical'' value of the parameter $\theta$, a stronger result is valid. The vector measure $(\lambda_1,\lambda_2, \lambda_3)$ described above is a global minimum of the vector energy (see \cite{MR2796829}); but, again, this result does not follow directly from our electrostatic approach.

\section{Differential equation of order 3} \label{sec:ODE3}

The system of second order linear differential equations in Theorems \ref{TeoClaveHP} and \ref{thm:WronskianS/w} allowed us to derive an electrostatic interpretation of the zeros of the Hermite--Pad\'e polynomials of type II.  In  this section, we show how these equations can be combined into a single third order homogeneous differential equation, satisfied simultaneously  by the polynomial $P_{\bm n}$ and by the functions of the second kind $q_{\bm n,1}$ and $q_{\bm n,2}$. Notice that if we only cared about a third order ODE solved by $P_{\bm n}$, it would be sufficient to differentiate  one of the equations \eqref{odeHP}; thus, it is convenient to stress here that  we seek   equations whose basis of solutions is precisely $( P_{\bm n} , q_{\bm n,1}, q_{\bm n,2})$.

As it was mentioned in the introduction, third order homogeneous linear ODE whose solutions are $P_{\bm n}$  have been already described in the literature. For instance, such an equation was found  in \cite{kaliaguine:1996} for the Jacobi-Angelesco multiple orthogonal polynomials, see Section~\ref{sec:AngelescoJacobi}. In \cite{Aptekarev:97}, Aptekarev et al.~considered the case of a semiclassical weight $w$ of class $\sigma$ and multiple orthogonal polynomials  $P_{\bm n}$ of type II with respect to $w$ and $\sigma +1$ non-homotopic paths of integration, showing for instance that  in the diagonal setting $\bm n=(n,n,...,n)$, $P_{\bm n}$ satisfies a linear differential equation of order $\sigma+2$.  A kind of opposite case, when distinct classical ($\sigma=0$) weights $w_j$ are supported on the same contour $\Gamma$, was analyzed in \cite{AptBraVA}, where again  a linear ODE of order  $r+1$,  where $r$ is the number of weights, was derived.

In a clear resemblance to the definition of the polynomial $R_{\bm n}$ in \eqref{defDiffOperators}--\eqref{defRpolyn}, let 
\begin{equation}\label{defEFalt2}
	E_{\bm n}:=A_1^2A_2^2v_1 v_2 \, 	 
	\det\begin{pmatrix}P_{\bm n}&q_{\bm n,1}&q_{\bm n ,2}\\P_{\bm n}''&q_{\bm n,1}''&q_{\bm n ,2}''\\
		P_{\bm n}'''&q_{\bm n,1}'''&q_{\bm n ,2}'''\end{pmatrix}\,,\quad
	F_{\bm n}:=A_1^2A_2^2v_1 v_2\, 
	\det\begin{pmatrix}P_{\bm n}'&q_{\bm n,1}'&q_{\bm n ,2}'\\P_{\bm n}''&q_{\bm n,1}''&q_{\bm n ,2}''\\
		P_{\bm n}'''&q_{\bm n,1}'''&q_{\bm n ,2}'''\end{pmatrix}.
\end{equation}
In this section we maintain the assumption \eqref{mainconditionwronks}.

\begin{thm} \label{thmODER}
	\begin{enumerate}[a)]
	\item 	Functions $E_{\bm n}$ and $F_{\bm n}$, defined above, are polynomials, with
	\begin{align*}
	\deg(E_{\bm n}) & \leq \max\{\deg(A_1)+\sigma_2+\deg(R_{\bm n}),\deg(A_2)+\sigma_1+\deg(R_{\bm n}),\deg(B_1)+\deg(B_2)+\deg(R_{\bm n})\}, \\ 
	\deg(F_{\bm n}) & \leq \sigma_1+\sigma_2+\deg(R_{\bm n})+1,
	\end{align*}
and $\sigma_i$ defined in \eqref{def:classHP}.

\item $P_{\bm n}$, $q_{\bm n,1}$ and $q_{\bm n,2}$ are solutions of the linear differential equation with polynomial coefficients
\begin{equation}\label{ode3}
	A_1A_2R_{\bm n}y'''+\left[ (A_1(2A_2'+B_2)+A_2(2A_1'+B_1))R_{\bm n}-A_1A_2R_{\bm n}' \right] y''+E_{\bm n}y'+F_{\bm n}y=0.
\end{equation}

\item 
	In the particular case when $A_1=A_2=A$, $B_1=B_2=B$ (so that $w_1=w_2$), and $\sigma=1$, $n_1=n_2$, the differential equation \eqref{ode3} reduces to
\begin{equation} \label{ode3bis}
	A^2y'''+2A(A'+B)y''+E_{\bm n}^*\, y'+F_{\bm n}^*\, y=0,
\end{equation}
	where $E_{\bm n}^*$ and $F_{\bm n}^*$ are polynomials of degree at most $4$ and $3$, respectively.
		\end{enumerate}
\end{thm}
\begin{proof}
Recall the second order differential operators introduced in \eqref{defDiffOperators},
$$
\mathcal L_i[y]:= 	A_i   V_i  \, 	\mathfrak{Wrons}[y, P_{\bm n}, q_{\bm n,i} ], \quad i=1, 2.
$$
By \eqref{diffOp1} (see Remark~\ref{remark:operator}), 
$$
\mathcal{L}_i[y]= A_i S_{\bm n, i}\, y''+(A'_i S_{\bm n, i}-A_i S_{\bm n, i}'+B_iS_{\bm n, i})\, y'+C_{\bm n,i}y.
$$
Clearly, $\mathcal{L}_i[P_{\bm n}]=\mathcal{L}_i[q_{\bm n,i}]=0$, and by \eqref{defRpolyn},
\begin{equation} \label{expressionEll}
\mathcal{L}_1[q_{\bm n,2}]=\frac{R_{\bm n}}{A_2 v_2}, \quad \mathcal{L}_2[q_{\bm n,1}]=-\frac{R_{\bm n}}{A_1 v_1}.
\end{equation}
Consider the third order linear differential operator
\begin{equation} \label{constructionL}
\mathcal{M}[y]:=\frac{A_2^2v_2}{S_{\bm n,1}}\left( \mathcal{L}_1[q_{\bm n,2}] \left(\mathcal{L}_1[y]\right)'
-\left( \mathcal{L}_1[q_{\bm n,2}]  \right)'\mathcal{L}_1[y]\right).
\end{equation}
By construction, the differential equation $\mathcal{M}[y]=0$ is solved by $P_{\bm n}$, $q_{\bm n,1}$ and $q_{\bm n,2}$. Using \eqref{defRpolyn} and \eqref{expressionEll} we can find the explicit expressions for the coefficients of $\mathcal M$. For instance, the coefficient at $y'''$ is
$$\frac{A_2^2v_2}{S_{\bm n,1}}\frac{R_{\bm n}}{A_2 v_2}A_1S_{\bm n,1}=A_1A_2R_{\bm n}.$$
The one at $y''$ is
\begin{align*}
	&\frac{A_2^2v_2}{S_{\bm n,1}}\left(\frac{R_{\bm n}}{A_2 v_2}\Big((A_1S_{\bm n,1})'+(-A_1S_{\bm n,1}'+A_1'S_{\bm n,1}+B_1S_{\bm n,1})\Big)-\left(\frac{R_{\bm n}}{A_2 v_2}\right)'A_1S_{\bm n,1}\right)\\
	=&\frac{A_2^2v_2}{S_{\bm n,1}}\left(\frac{R_{\bm n}}{A_2 v_2}\left(2A_1'S_{\bm n,1}+B_1S_{\bm n,1}\right)-\frac{R_{\bm n}}{A_2 v_2}\left(\frac{R_{\bm n}'}{R_{\bm n}}-\frac{A_2'}{A_2}-\frac{v_2'}{v_2}\right)A_1S_{\bm n,1}\right)\\
	=&A_2R_{\bm n}\left(\left(2A_1'+B_1\right)-\left(\frac{R_{\bm n}'}{R_{\bm n}}-\frac{A_2'}{A_2}-\frac{v_2'}{v_2}\right)A_1\right)\\
	=&A_2R_{\bm n}\left(\left(2A_1'+B_1\right)-\left(\frac{R_{\bm n}'}{R_{\bm n}}-\frac{A_2'}{A_2}-\frac{A_2'+B_2}{A_2}\right)A_1\right)\\
	=&A_1(2A_2'+B_2)R_{\bm n}+A_2(2A_1'+B_1)R_{\bm n}-A_1A_2R_{\bm n}'.
\end{align*}
Similar calculations for the rest of the coefficients show that 
$$
\mathcal{M}[y]=A_1A_2R_{\bm n}y'''+(A_1(2A_2'+B_2)R_{\bm n}+A_2(2A_1'+B_1)R_{\bm n}-A_1A_2R_{\bm n}')y''+E_{\bm n}y'+F_{\bm n}y,
$$
where
\begin{align}
	E_{\bm n}&=\frac{-A_1A_2R_{\bm n}S_{\bm n,1}''+((-A_1(2A_2'+B_2)+A_2B_1)R_{\bm n}+A_1A_2R_{\bm n}')S_{\bm n,1}'+A_2R_{\bm n}C_{\bm n,1}}{S_{\bm n,1}}\nonumber\\
	&+(A_1''A_2+A_1'(2A_2'+B_2)+2A_2'B_1+A_2B_1'+B_1B_2)R_{\bm n}-A_2(A_1'+B_1)R_{\bm n}'\,,\label{def:E}
\end{align}
and
\begin{equation}\label{def:F}
	F_{\bm n}=\frac{(A_2C_{\bm n,1}'+(B_2+2A_2')C_{\bm n,1})R_{\bm n}-A_2C_{\bm n,1}R_{\bm n}'}{S_{\bm n,1}}.
\end{equation}
Construction \eqref{constructionL} can be carried out exchanging the role of the indices $i=1$ and $i=2$. It yields a third order linear differential equation with the same coefficient at $y'''$. Since $P_{\bm n}$, $q_{\bm n,1}$ and $q_{\bm n,2}$ are linearly independent (assumption \eqref{mainconditionwronks}),  we conclude that this is the same ODE. In other words,  \eqref{def:E} and \eqref{def:F} are invariant by exchange of the indices $i=1$ and $i=2$. 

Moreover, a third order linear differential equation with the same set of solutions is
$$
A_1^2A_2^2v_1v_2\, \mathfrak{Wrons}[P_{\bm n},q_{\bm n,1},q_{\bm n ,2},y]=0.
$$
Again, by \eqref{defRpolyn}, the coefficient at $y'''$ is $A_1A_2R_{\bm{n}}$, which shows that
\begin{equation} \label{identityL}
\mathcal M[y]=A_1^2A_2^2v_1v_2\, \mathfrak{Wrons}[P_{\bm n},q_{\bm n,1},q_{\bm n ,2},y].
\end{equation}
In particular, functions $E_{\bm n}$ and $F_{\bm n}$  in \eqref{def:E} and \eqref{def:F} coincide with those defined in \eqref{defEFalt2}. 

From  \eqref{def:E} and \eqref{def:F} it follows that  $E_{\bm n}$ and $F_{\bm n}$ are rational functions.
By \eqref{identityL} and the definition of $q_{\bm n,i}$, $i=1, 2$, their  poles can be located at the zeros of $A_1$ and $A_2$ only. However, expressions \eqref{defEFalt2}  and assertion c) of Proposition~\ref{prop:fromTHM2.1} imply that all their singularities are all removable, so that $E_{\bm n}$ and $F_{\bm n}$  are polynomials. Since by \eqref{odegeneral}, $\deg(C_{\bm n,1})-\deg(S_{\bm n,1})\leq \sigma_1$, identities \eqref{def:E} and \eqref{def:F} imply the claimed upper bounds for the degrees of $E_{\bm n}$ and $F_{\bm n}$ . This proves a) and b) of the statement of the theorem.

Finally, let  $A_1=A_2=A$, $B_1=B_2=B$ (so that $w_1 =w_2 $), and $\sigma=1$, $n_1=n_2$.

In this situation, $R_{\bm n}= cA^2$, $c\in \R\setminus \{0\}$ (see Remark~\ref{remark:equal}), so that by  \eqref{def:E},  
\begin{equation}\label{def:E2}
E_{\bm n}= c	 A^2 \left( A\frac{-AS_{\bm n,i}''+C_{\bm n,i}}{S_{\bm n,i}}
	+ AA''+A'B+AB'+B^2\right), \quad i=1, 2, 
\end{equation}
while  by \eqref{def:F}, 
\begin{equation}\label{def:F2}
  F_{\bm n}=c  A^2\frac{BC_{\bm n,i}+AC_{\bm n,i}'}{S_{\bm n,i}}, \quad i=1, 2.
\end{equation}
Dividing \eqref{ode3} through by $cA^2$ we get \eqref{ode3bis}.
\end{proof}

 \begin{remark} \label{remark:higherorder}
The construction given in this proof can be, apparently, extended to MOP of type II with respect to more than two weights. For instance, for a multiple orthogonal polynomial $P_{\mathbf{n}}$ with respect to semiclassical weightss $w_i$, $i=1,2,\dots, m\geq 3$,  we get the corresponding ODEs
 $$\mathcal{L}_i[y]:=A_i S_{\mathbf{n},i}y''+(A_i'S_{\mathbf{n},i}-A_iS_{\mathbf{n},i}'+B_iS_{\mathbf{n},i})y'+C_{\mathbf{n},i}y=0, \qquad i=1,2,\dots, m.$$
 Then $P_{\mathbf{n}}$  and $q_{\mathbf{n},i}$, $i=1, 2,\dots, m$, are solutions of the fourth order ODE
 \begin{equation} \label{eqHigherOrder}
 \mathcal{M}[y]:=\frac{1}{S_{\mathbf{n},1}} \left( \prod_{i=2}^m A_i^{m-1} v_i\right)
\, \mathfrak{Wrons}[\mathcal{L}_1[y], \mathcal{L}_1\left[q_{\mathbf{n},2}], \dots, \mathcal{L}_1[q_{\mathbf{n},3}], \dots, \mathcal{L}_1[q_{\mathbf{n},m}]\right]=0.
\end{equation}
As in the proof above, one could expect that the coefficients of this ODE are rational functions with only possible poles  at the roots of $S_{\mathbf{n},1}$. Analogous ODEs can be obtained by replacing in \eqref{eqHigherOrder} $S_{\mathbf{n},1}$ by $S_{\mathbf{n},i}$ and $\mathcal{L}_1$  by $\mathcal{L}_i$, with $i=2, \dots, m$. This will imply again that the poles of the coefficients of the ODE could be only at common roots of the electrostatic partners. On the other hand, these ODEs are equivalent to
 $$
\left( \prod_{i=1}^m A_i^{m} v_i\right)
\, \mathfrak{Wrons}\left[y, P_{\mathbf{n}}, q_{\mathbf{n},1}, q_{\mathbf{n},2}, \dots , q_{\mathbf{n},m}\right]=0.
 $$
Assertion c) of Proposition~\ref{prop:fromTHM2.1} yields again that all the possible poles of the coefficients should be removable. This construction leads to a linear ODE of order $m+1$ with polynomials coefficients whose degrees depend  only on the classes of the weights $w_i$,   $i=1, 2, \dots, m$.
 \end{remark}

\section{Further examples}\label{sec:examples}

In this section, we discuss   several examples of multiple orthogonal polynomials. 

\subsection{Jacobi polynomials with non-standard parameters}\label{sec:complex}
\

	Les us return to Jacobi polynomials $P_N=P_N^{(\alpha, \beta)}$ in the non-standard situation considered already in Example \ref{ex:Jacobi1}, namely, when neither $\alpha, \beta$, or  $\alpha + \beta $ are integers,  $\beta> -1$, and $-N<\alpha < -1$. As it was shown in \cite[Theorem 6.1]{ETNA05},  in this case $P_N$ is a type II multiple orthogonal polynomial. Indeed, on one hand it satisfies 
	$$
	\int_{-1}^1 x^j P_N(x) w_1(x)dx=0\,,\quad j=0,1,\dots, n_1-1,
	$$
	where $\Delta_1= [-1,1]$,  $n_1=N-[-\alpha]$,  and  $w_1(x)=(x-1)^{\alpha+[-\alpha]}(x+1)^\beta	$ (see Example \ref{ex:Jacobi1}). On the other, 
	$$
	\int_{\Delta_2} z^j P_N(z) w_2(z)dz=0\,,\quad j=0,1,\dots, n_2-1,
	$$
	where $\Delta_2$ is an arbitrary curve 	oriented clockwise, connecting $-1 - i 0$ with $-1 + i 0$ and lying
	entirely in $\C \setminus (-\infty, 1]$, except for its endpoints, $n_2=[-\alpha]$,  and $w_2(z)=(z-1)^{\alpha }(z+1)^\beta	$. 

We have established in Example \ref{ex:Jacobi1} that 
	$$
	S_{\bm n, 1}(x)=(x-1)^{[-\alpha]},
	$$
	as well as that $C_{\textbf{n},1}(x) = -\lambda_{\textbf{n},1} S_{\textbf{n},1}(x)$. As for $S_{\bm{n},2}$, we lack in this case  the second condition in  \eqref{defHP}, that is, we cannot guarantee that	
	\begin{equation}\label{nonzero}
	\int_{\Delta_2} z^{n_2} P_N(z) w_2(z)dz\neq 0,
	\end{equation} 
	(and in general, this is false), which makes the formula \eqref{leadingC} of no value. 
	
Reasoning as in Example \ref{ex:Jacobi1} and combining \eqref{ode1} and the standard differential equation for the Jacobi polynomials we arrive at the identity 
	\begin{equation}\label{id1S2}
		(x^2-1) S'_{\textbf{n},2}(x) P'_{N}(x) = (\lambda_N S_{\bm {n},2}(x) + C_{\bm {n},2}(x))\,P_{N}(x)\,.
	\end{equation} 
	Again, we have two options: either $S'_{\bm {n},2} \equiv 0$ and, thus, 
	\begin{equation}\label{S2}
		S_{\bm {n},2}(x) \equiv \;\text{const}\,,\quad C_{\bm {n},2} = - \lambda_N S_{\bm {n},2} \equiv \;\text{const}\,,
	\end{equation}
	or, otherwise, the identity
	\begin{equation}\label{id2S2}
		\frac{P'_{N}(x)}{P_{N}(x)}\,=\,\frac{\lambda_N S_{\bm {n},2}(x) + C_{\bm {n},2}(x)}{(x+1) \left((x-1) S'_{\bm {n},2}(x) - [-\alpha] S_{\bm {n},2}(x)\right) }
	\end{equation}
	holds. In this case, the facts that $P_{N}$ and $P'_{N}$ are relatively prime ($P_N$ has no multiple roots) and that the degree of $S_{\bm {n},2} \leq N-[-\alpha]\,,$ with $-N <\alpha <-1\,,$ imply that this cannot take place, and thus, \eqref{S2} is the unique possible solution.
	
	We already saw in Example  \ref{ex:Jacobi1} that the zeros of $P_{\bm {n}}$ were in equilibrium (that is, their counting measure was critical) in the external field \eqref{extFieldJacobi}.   From \eqref{expressionS_NJac}, \eqref{expreforR} and \eqref{S2}, and taking into account that in this case $A_1 = A_2 = A$, we have   that $R_{\bm n} \equiv 0$. Moreover, the zeros of $S_{\bm {n},1}$ are obviously not simple, and in such a case we cannot say anything about electrostatics for its zeros.
	
Obviously, what fails in this case is \eqref{nonzero}, which implies that the index $\bm n=(N-[-\alpha],  [-\alpha])$ is not normal.

\subsection{Multiple Hermite polynomials} \label{sec:multipleHermite}

Multiple Hermite polynomials $\frak H_{\bm n}$, $\bm n=(n_1, n_2)$,  are type II MOP of degree $\leq N=n_1+n_2$, defined by
$$
\begin{aligned}
	&\int_{-\infty}^{\infty} x^{k} \frak H_{\bm n}(x)  e ^{-x^{2}+c_{1} x} \, dx=0, \quad k=0,1, \ldots, n_1-1 ,\\
	&\int_{-\infty}^{\infty} x^{k} \frak H_{\bm n}(x)  e ^{-x^{2}+c_{2} x} \, dx=0, \quad k=0,1, \ldots, n_2-1.
\end{aligned} 
$$
If $c_{1} \neq c_{2}$ then   the weights $w_i(x)=e ^{-x^{2}+c_{i} x} $  form an AT-system, see \cite{MR2187942},  \cite[section 23.5]{Ismail05} and  \cite[section 3.4]{MR1808581}.
These MOP can be obtained using the Rodrigues formula
$$
e^{-x^{2}} \frak H_{\bm n}(x)=(-1)^{N} 2^{-N}\left(\prod_{j=1}^{2} e^{-c_{j} x} \frac{d^{n_{j}}}{d x^{n_{j}}} e^{c_{j} x}\right) e^{-x^{2}},
$$
which yields  the explicit expression
$$
\frak H_{\bm n}(x)=(-1)^{|\vec{n}|} 2^{-|\vec{n}|} \sum_{k_{1}=0}^{n_{1}}  \sum_{k_{2}=0}^{n_{2}} \binom{n_1}{ k_1}	 \binom{n_2}{ k_2}	 
c_{1}^{n_{1}-k_{1}}  c_{2}^{n_{2}-k_{2}}(-1)^{k_1+k_2} H_{k_1+k_2}(x),
$$
where $H_n(x)=2^n x^n +\dots$ is the standard Hermite polynomial \eqref{hermiteH}, see \cite[Section 23.5]{Ismail05} or \cite{MR3907776}.

In our notation,  
$$
A_1(x)=A_2(x)\equiv 1, \quad B_i(x)=-2x+c_i , \quad \sigma_i=0,  \quad i=1, 2,
$$
with $\Delta_1=\Delta_2=\mathbb R$.   
The differential equation \eqref{ode3} takes the form
$$
R_{\bm n}y'''+\left[ (B_1 + B_2)R_{\bm n}- R_{\bm n}' \right] y''+E_{\bm n}y'+F_{\bm n}y=0.
$$
Formula \eqref{degRoverAparticularcase} shows that in this case $ R_{\bm n}$ is a constant, so that the equation boils down to
 $$
y'''+ (B_1 + B_2)  y''+E_{\bm n}y'+F_{\bm n}y=0,
 $$
where by Theorem \ref{thmODER}, $E_{\bm n}$ and $F_{\bm n}$ have degree at most $2$ and $1$, respectively. These polynomials can be obtained explicitly by taking into account the behavior of the solutions of the equation at the singular points. Indeed, if we write
$$E_{\bm n}(x)=e_0+e_1x+e_2x^2\,,\qquad F_{\bm n}(x)=f_0+f_1 x\,,$$
and replace the asymptotic behavior of  $q_{\bm n,1}(x)$  as $x\to\infty$ (for instance, along the imaginary axis),  
$$q_{\bm n,1}(x)=\text{const} \times e^{x^2-c_1x}x^{-n_1-1}(1+\mathcal O(1/x)),$$
into the differential equation,  we get consecutively
$$e_2=4,\qquad e_1=-2c_1-2c_2,\qquad e_0=-2+c_1c_2-\frac{f_1}{2},\qquad f_0=2c_2n_1-2c_1n_1-\frac{c_1f_1}{2}.$$
An analogous procedure for  $q_{\bm n,2}$, together with the previous  identities, yields
$$f_1=-4n_1-4n_2\,.$$
As a consequence,  we get explicit expressions for polynomials $E_{\bm n}$ and $F_{\bm n}$, which can be expressed as follows:
$$E_{\bm n}=B_1B_2+2(n_1+n_2-1)\,,\qquad F_{\bm n}=2n_2B_1+2n_1B_2.$$
Thus, $\frak H_{\bm n}$ and the corresponding functions of the second kind  $q_{\bm n,1}$ and $q_{\bm n,2}$ are independent solutions of the equation
$$
y'''+ (B_1 + B_2)  y''+(B_1B_2+2(n_1+n_2-1))y'+(2n_2B_1+2n_1B_2)y=0,
$$
which coincides with the one  obtained previously in \cite[section 5.1]{MR3055365}.

By Theorem \ref{thm:vectorelectrostatics},  
the discrete vector  measure $\vec \nu_1:=(\nu(\frak H_{\bm n}),\nu(S_{\bm n,1}))$ is a  critical vector  measure for the energy functional $\mathcal{E}_{\vec \varphi, a}$, with $a=-1/2$ and
$$
\vec \varphi (z)=\frac{1}{2}  \left( x^2-c_1 x ,    (c_1-c_2) x  \right), \quad z=x+iy. 
$$

Direct computation shows that 
$$
m^{\pm}_k:=\int_{-\infty}^\infty x^k  e ^{-x^{2}\pm x}dx =\sqrt{\pi }   \sqrt[4]{e} \left(\frac{\pm 1}{2i}\right)^k H_k\left(\frac{i}{2}\right), \quad k=0, 1, \dots ,
$$
and that
$$
 \int_{-\infty}^\infty x^k  e ^{-x^{2}+c x}dx =e^{(c^2-1)/4} \sum_{j=0}^k \binom{k}{j} \left( \frac{c-1}{2}\right)^{k-j} m^+_j, \quad c\neq 1, \quad k=0, 1, \dots .
$$
These formulas allow us to obtain (at least, using symbolic computation) the moments of $w_i$, and in consequence,   the asymptotic expansion of  $\mathfrak C_{w_i}[\frak H_{\bm n}]$ at infinity. This yields   $S_{\bm n, i}$ by formula \eqref{defS12}. 

In the following examples we will consider the symmetric case  $n_1=n_2$ and $c_1=-c_2=c$, for which clearly $S_{\bm n, 1}(-x)$ and $S_{\bm n, 2}(x)$ coincide up to a multiplicative constant. In this case,  explicit formulas for $\frak H_{(n,n) }$ and $S_{\bm n, 1}$ are easily obtained with the help of a computer algebra system, at least for low $n$'s. 
For instance, for $\bm n=(5,5)$ and $c=1$ we have that 
$$
\frak H_{(5,5) }(x)=x^{10}-\frac{95 x^8}{4}+\frac{1405 x^6}{8}-\frac{14855 x^4}{32}+\frac{94325 x^2}{256}-\frac{39971}{1024},
$$
and up to normalization,
$$
S_{\bm n, 1}(x)=S_{\bm n, 2}(-x)=32 x^5+560 x^4+4240 x^3+17560 x^2+39970 x+39971 .
$$
All zeros of $\frak H_{(5,5) }$ are real and simple, while $S_{\bm n, 1}$ has one real zero (smaller than the zeros of $\frak H_{(5,5) }$) and two pairs of complex conjugate simple zeros.

 Asymptotics of  sequences of (rescaled) multiple Hermite polynomials  
 $$
 p_{(n,n)}(t) = \frak H_{(n,n) }\left( \sqrt{n }\,  t \right), \quad 
 $$ 
 with $c_1=-c_2=c$ proportional to $\sqrt{n}$, has been studied by Aptekarev, Bleher and Kuijlaars in a series of papers \cite{MR2172687, Bleher/Kuijlaars1, MR2276453} in the context of the random matrix theory. In particular, they found that the support of the limit of the zero-counting measures $\nu(p_{(n,n)})$ is a single interval, roughly speaking, for $0\le c \ll 2\sqrt{n}$, and is comprised of two symmetric intervals for $c \gg 2\sqrt{n}$. It is interesting to compare these conclusions with results of the  numerical experiments presented  in Figure~\ref{FigHermite35}. There, $\bm n=(35,35)$ and $c_1=-c_2=c>0$, with the phase transition happening around $c\approx 12$. Notice that for $c=1$, the zeros of   $S_{\bm n, 1}$ are visibly distributed along a curve on the complex plane. As $c$ increases, more and more zeros of $S_{\bm n, 1}$ migrate to the negative semi-axis, interlacing with the zeros of $\frak H_{(n,n) }$, until we get a two-cut situation. In this case, the configuration resembles the relative position of the zeros of  $\frak H_{(n,n) }$ and $S_{\bm n, 1}$ for the Angelesco system, described in Section \ref{sec:angelesco}, which explains why the description of the asymptotic limit of $\nu(\frak H_{\bm n})$ in this case is given in \cite{Bleher/Kuijlaars1} in terms of the Angelesco vector equilibrium problem, see Section~\ref{sec:AsymptAngelesco}.

\begin{figure}[h]
	\centering 
		\hspace{-3mm} \begin{tabular}{cc}
		\begin{overpic}[scale=0.6]{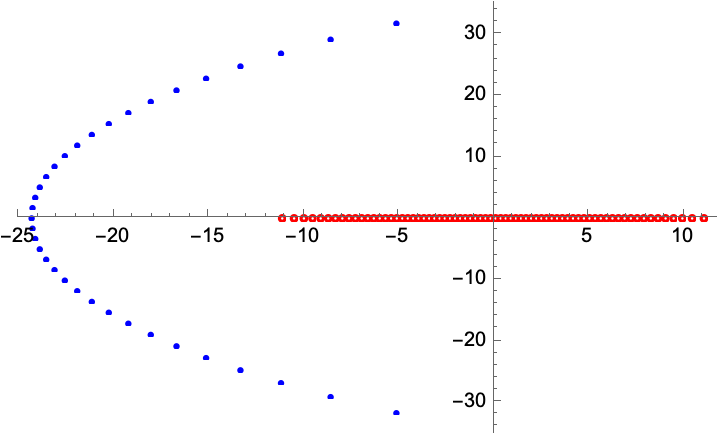}
		\end{overpic} 
			& 
				\begin{overpic}[scale=0.6]{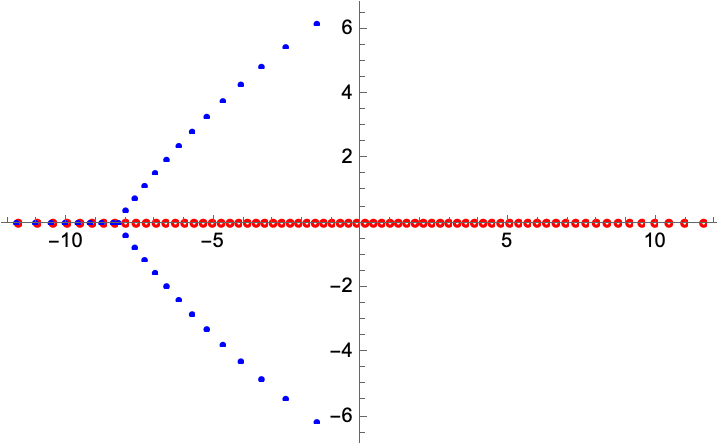}
				\end{overpic} \\
			\begin{overpic}[scale=0.6]{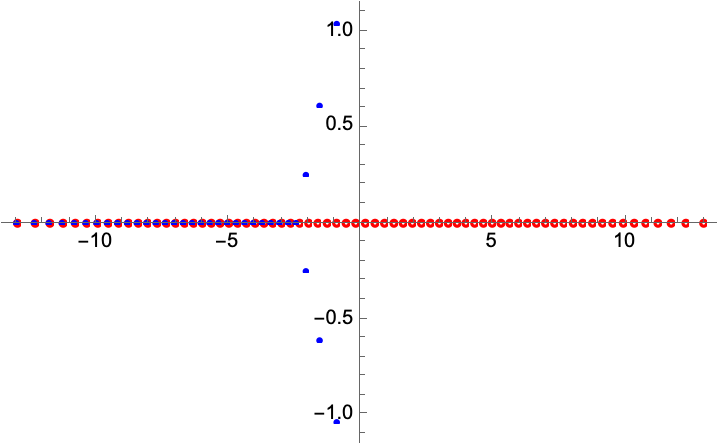}
			\end{overpic} 
			& 
			\begin{overpic}[scale=0.6]{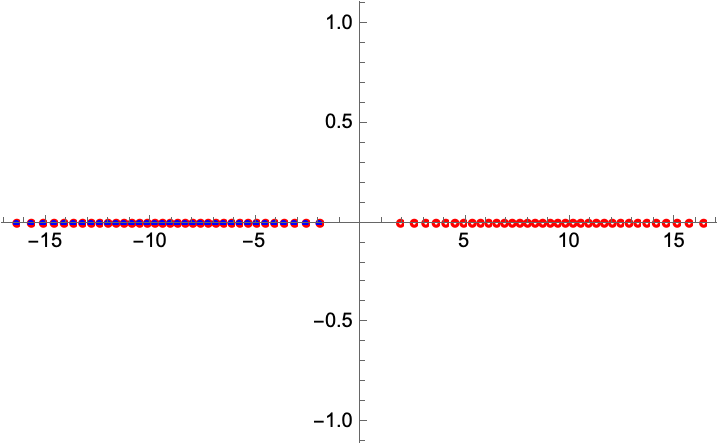}
			\end{overpic}
			\end{tabular}
	\caption{Zeros of  the multiple Hermite polynomial  $\frak H_{\bm n}$ (indicated by empty circles, all  on the real axis) and of $S_{\bm n, 1}$ (filled circles)   for  $\bm n=(35,35)$, $c_1=-c_2$, for different values of $c$: $c=1$ (top left), $c=4$ (top right), $c=8$ (bottom left) and $c=16$ (bottom right). The zeros of $S_{\bm n, 1}$ that are on the left semi-axis apparently interlace with the zeros of  $\frak H_{\bm n}$.} \label{FigHermite35}
\end{figure}
 
 A generalization of multiple Hermite polynomials to the case of polynomials $P_{\bm n}$, $\bm n = (n,n)$, satisfying the varying orthogonality conditions
$$
\int_{-\infty}^{\infty} x^{k} P_{\bm n}(x)  e ^{-n(V(x)\pm c x)} \, dx=0, \quad k=0,1, \dots, n -1,
$$
where $V$ is a polynomial of even degree and positive leading coefficient, has been carried out in \cite{MR2743878}, again associated to random matrix models with external source. The limit zero distribution of $P_{\bm n}$'s was described there in terms of a constrained 2-component vector equilibrium problem, with one of the components on the imaginary axis. The two-cut situation in the asymptotics of multiple Hermite polynomials, and thus the reduction to the Angelesco equilibrium, is in this case  equivalent to the constraint on the imaginary axis to be not achieved (not ``saturated''). The study in \cite{MR2743878} has been extended in \cite{AMFS21} (see also \cite{ALT2011} for the quartic case) to a non-symmetric situation, showing that it can be alternatively characterized in terms of a 3-component critical vector  measure with the interaction matrix from Corollary~\ref{cor:asymptotics3}. The curves outlined by the zeros of $P_{\bm n}$ and $S_{\bm n, 1}$ in Figure~\ref{FigHermite35}  are consistent with the support of the three components described in \cite[Theorem C]{AMFS21}.

\subsection{Multiple Laguerre polynomials of the first kind} \label{sec:MultLaguerreIkind}

These polynomials are defined by the orthogonality conditions  
$$
\begin{aligned}
	&\int_{0}^{\infty} x^{k} \frak L_{\bm n}(x) x^{\alpha_{1}}  e^{-x} \, d x=0, \quad k=0,1, \ldots, n_1-1 \\
	&\int_{0}^{\infty} x^{k} \frak L_{\bm n}(x) x^{\alpha_{2}} e^{-x} \mathrm{~d} x=0, \quad k=0,1, \ldots, n_2-1,
\end{aligned}\quad \deg  \frak L_{\bm n} \leq N=n_1+n_2,
$$
where $\alpha_{1}, \alpha_{2}>0$ and $\alpha_{1}-\alpha_{2} \notin \mathbb{Z}$, under which condition the weights form an AT-system; see \cite{AptBraVA}, \cite[section 23.4.1]{Ismail05},  and \cite[Section 3.2]{MR1808581}.  Not only that, since $x^\beta$, for $\beta<0$ and $x>0$, can be written as the Cauchy integral of a positive weight on $(-\infty, 0)$, coinciding up to a multiplicative constant with $|x|^\beta$, we conclude that this pair of weight forms a Nikishin system with $\Delta_1=\Delta_2=[0,+\infty)$ and $[c,d]=[-\infty, 0]$, see Section \ref{sec:nikishin}. 

Polynomials $\frak L_{\bm n}$ can be obtained by using the Rodrigues formula,
$$
(-1)^{N} e^{-x} \frak L_{\bm n}(x)=\prod_{j=1}^{2}\left(x^{-\alpha_{j}} \frac{d^{n_{j}}}{d x^{n_{j}}} x^{n_{j}+\alpha_{j}}\right) e^{-x}, \quad N=n_1+n_2, 
$$
from which one can find the explicit expression
\begin{align*}
(-1)^{N} e^{-x} \frak L_{\bm n}(x) & = 	 \left(\alpha_{1}+1\right)_{n_{1}}  \left(\alpha_{2}+1\right)_{n_{2}} {}_{2} F_{2}
\left(\begin{array}{c}
		\alpha_{1}+n_{1}+1,   \alpha_{2}+n_{2}+1 \\
		\alpha_{1}+1,  \alpha_{2}+1
	\end{array} \bigg| -x\right).
\end{align*}
 
In our notation,  
$$
A_1(x)=A_2(x)=x, \quad B_i(x)= \alpha_i - x, \quad \sigma_i=0,  \quad i=1, 2,
$$
with $\Delta_1=\Delta_2=[0,+\infty)$.  These polynomials  satisfy the differential equation  \eqref{ode3}, which  takes the form
\begin{equation} \label{laguerremult}
	x^2 R_{\bm n}(x) y'''(x)+\left[ (B_1(x)+B_2(x)+4) xR_{\bm n}(x)- x^2 R_{\bm n}' (x)\right] y''(x)+E_{\bm n}(x)y'(x)+F_{\bm n}(x)y(x)=0.
\end{equation}
It follows from formula  \eqref{degRoverAparticularcase} that  $ R_{\bm n}(x)/x$ is a constant.  Thus, the equation reduces to 
$$
x^3 y'''+ x^2 \left[ B_1(x)+B_2(x)+3 \right] y''+E_{\bm n}y'+F_{\bm n}y=0, 
$$
where $E_{\bm n}$ and $F_{\bm n}$ are of degrees at most $3$ and $2$, respectively. As in Section~\ref{sec:multipleHermite},  the asymptotics of the corresponding functions of second kind $q_{\bm n,1}$ and $q_{\bm n,2}$ at $\infty$ yields some constraints on the coefficients of $E_{\bm n}$ and $F_{\bm n}$, which unfortunately are not sufficient to determine the polynomials in this case. We need to make use also of the predicted   behavior of the solutions at the origin. 

Notice that $0$ is a regular singular point (a Fuchsian singularity) of \eqref{laguerremult}. The fact that the weights constitute an AT-system on $[0,+\infty)$ implies also that $\frak L_{n_1, n_2}(0)\neq 0$. In consequence,   $E_{\bm n}(0)=0$ and  $F_{\bm n}(0)=0$. Expanding the solutions at the origin we conclude that the indicial polynomial  must vanish at $0$, $-\alpha_1$ and $-\alpha_2$. All this additional information allows us to determine $E_{\bm n}$ and $F_{\bm n}$:
\begin{align*}
	E_{\bm n}&=x(x^2+(n_1+n_2-\alpha_1-\alpha_2-3)x+(1+\alpha_1)(1+\alpha_2)\,,\\
	F_{\bm n}&=x(-(n_1+n_2)x+n_1+n_2+n_1n_2+\alpha_1n_2+\alpha_2n_1)\,.
\end{align*}
Canceling the common factor $x$ in the  four coefficients of \eqref{laguerremult} yields  the  equation that appeared already in \cite[Section 4.3]{AptBraVA},
$$
\begin{aligned}
	x^{2} y''' (x) &+\left(-2 x^{2}+\left(\alpha_{1}+\alpha_{2}-3\right) x\right) y''(x)+\left(x^{2}-x\left(\alpha_{1}+\alpha_{2}-n_1-n_2+3\right)\right.\\
	&\left.+\left(\alpha_{1}+1\right)\left(\alpha_{2}+1\right)\right) y'(x)-\left(x(n_1+n_2)-\left(n_1+n_2+n_1 n_2+\alpha_{1} n_2 +\alpha_{2} n_1\right)\right) y(x)=0.
\end{aligned}
$$

From the conclusions of Section~\ref{sec:nikishin} it follows that the zeros of $\frak L_{\bm n}$ are all positive, and those of  $S_{\bm n,1}$ are negative. By Theorem \ref{thm:vectorelectrostatics},  
the  discrete vector measure $\vec \nu_1:=(\nu(\frak L_{\bm n}),\nu(S_{\bm n,1}))$ is a critical vector  measure for the energy functional $\mathcal{E}_{\vec \varphi, a}$, with $a=-1/2$ and
$$
\vec \varphi (x)=\frac{1}{2}  \left( x-(\alpha_1+1) \log(x)  ,    (\alpha_1-\alpha_2-1)\log|x| \right) .   
$$

Direct computations show that 
$$
\int_{0}^\infty x^{k+\alpha_j}  e ^{-  x}dx =\begin{cases}
	 \sqrt{\pi} \,  2^{-k-1} (2k+1)!!  , & j=1, \\
	 (k+1)!, & j=2,
\end{cases}  \quad k=0, 1, \dots ,
$$
which allows to find the moments of $w_i$,  the asymptotic expansion of  $\mathfrak C_{w_i}[\frak L_{\bm n}]$ at infinity, and in consequence, $S_{\bm n, i}$ (using formula \eqref{defS12}). 

Let us consider the particular case of $n_1=n_2=n$, with $\alpha_1=1/2$, $\alpha_2=1$. 
Then, for $\bm n=(5,5)$,
\begin{align*}
\frak L_{(5,5) }(x)  = & x^{10}-\frac{165 x^9}{2}+\frac{5445 x^8}{2}-\frac{186615 x^7}{4} +\frac{7224525 x^6}{16}-\frac{80613225 x^5}{32}+\frac{127182825
	x^4}{16} \\
& -\frac{107120475 x^3}{8}+10758825 x^2-\frac{13253625 x}{4}+\frac{467775}{2} ,
\end{align*}
and up to normalization,
\begin{align*}
	S_{\bm n, 1 }(x) & =   8 x^5+2720 x^4+107500 x^3+1945020 x^2+46682295 x+1425581520, \\
	S_{\bm n, 2}(x) & =  	32 x^5+2960 x^4+67424 x^3+1313480 x^2+37066290 x+1173966885.
\end{align*}
All zeros of $\frak L_{(5,5) }$ are positive and simple, while each $S_{\bm n, j}$ has one  negative  and two pairs of complex conjugate simple zeros, all of them simple. 

Asymptotics of  sequences of (rescaled) multiple Laguerre polynomials  of the first kind 
$$
p_{(n,n)}(t) = \frak L_{(n,n) }\left( 2 n t  \right)
$$ 
was obtained in \cite{zbMATH05356762} and \cite{MR3471160}. It was shown that the support of the weak-* limit of the zero-counting measures $\nu(p_{(n,n)})$ is the interval $\left[ 0, 27/8 \right]$, with the density presenting the usual square root vanishing at the rightmost endpoint of the support. The expression of the density was derived from the recurrence relation satisfied by polynomials $\frak L_{(n_1,n_2) }$ and no equilibrium problem associated to that distribution was given.

\begin{figure}[h]
	\centering 
		\begin{overpic}[scale=0.7]{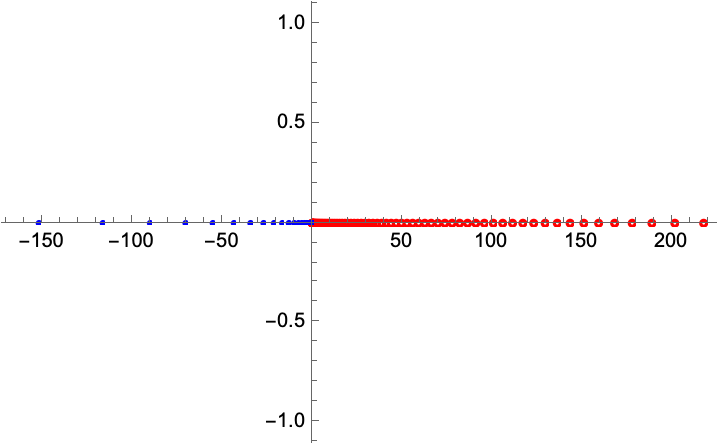}
		\end{overpic} 
	\caption{Zeros of  the multiple Laguerre polynomial of the first kind $\frak L_{\bm n}$ (indicated by empty circles, all  on  the positive semiaxis) and of $S_{\bm n, 1}$ (filled circles, all on the negative semiaxis)   for  $\bm n=(35,35)$ and $\alpha_1=1/2$, $\alpha_2=1$. Nine real zeros of $S_{\bm n, 1}$ (ranging from $-74000$ to   $-201.53$, are not represented. } \label{FigLaguerreI35}
\end{figure}

Again, it is interesting to compare these conclusions with results of the  numerical experiments presented  in Figure~\ref{FigLaguerreI35}, where we take $\bm n=(35,35)$, with $\alpha_1=1/2$, and $\alpha_2=1$. 
The largest zero of  $\frak L_{(35,35) }$ is $217.597$, which is consistent with the expected value of
$$
\frac{27}{8} \times 70 = 236.25.
$$

\subsection{Multiple Laguerre polynomials of the second kind} \label{sec:multipleLag2kind}

These polynomials are defined by the orthogonality conditions
$$
\begin{aligned}
	&\int_{0}^{\infty} x^{k} \mathcal L_{\bm n}(x) x^{\alpha} e^{-c_{1} x} \, dx=0, \quad k=0,1, \ldots, n_1-1,  \\
	&\int_{0}^{\infty} x^{k} \mathcal L_{\bm n}(x) x^{\alpha} e^{-c_{2} x} \, dx=0, \quad k=0,1, \ldots, n_2-1, 
\end{aligned}
$$
where we assume that $\alpha>0$ and $c_{1}, c_{2}>0$ with $c_{1} \neq c_{2}$, under which condition the weights form an AT-system; see e.g.~\cite{MR2187942}, \cite{zbMATH05343896},  or   \cite[section 3.3]{MR1808581}.

An explicit expression can be found in \cite[Section 3]{AptBraVA} or \cite[Section 23.4]{Ismail05}:
$$
	\mathcal L_{\bm n}(x)=  \sum_{k_{1}=0}^{n_{1}}   \sum_{k_{2}=0}^{n_{2}} (-1)^{k_1+k_2} \frac{(k_1+k_2) !}{c_{1}^{k_{1}}  c_{2}^{k_{2}}} \binom{n_{1} }{k_{1} } \binom{n_{2} }{k_{2} } \binom{N+\alpha}{	k_1+k_2}   x^{N-k_1-k_2}.
$$ 
In our notation,  
$$
A_1(x)=A_2(x)=x, \quad B_i(x)= \alpha -c_i x, \quad \sigma_i=0,  \quad i=1, 2,
$$
with $\Delta_1=\Delta_2=[0,+\infty)$.  These polynomials also satisfy the differential equation \eqref{laguerremult}, 
that is,
$$
x^2 R_{\bm n}(x) y'''(x)+\left[ (-(c_1+c_2)x^2 + 2 (\alpha +2)x)R_{\bm n}(x)- x^2 R_{\bm n}' (x)\right] y''+E_{\bm n}(x)y'(x)+F_{\bm n}(x)y(x)=0.
$$
Formula \eqref{degRoverAparticularcase} shows that now $ R_{\bm n}(x)/x$ is a polynomial of degree $1$. 
Furthermore, by \eqref{IdR1} and since $B_1-B_2$ vanishes at $0$, $R_{\bm n}\mathfrak{L}_{\bm n}$ has a double root at $0$. Again, the fact that the weights constitute an AT-system on $[0,+\infty)$ implies also that  $\mathfrak{L}_{\bm n}(0)\neq 0$. Hence, $R_{\bm n}(x)/x^2$ is a constant, and the third order differential equation is
$$x^4y'''-x^3((c_1+c_2)x-\alpha-2)y''+E_{\bm n}y'+F_{\bm n}y=0.
$$
Arguments as described in Sections~\ref{sec:multipleHermite} and \ref{sec:MultLaguerreIkind}, using the asymptotics of the functions of second kind both at $0$ and at $\infty$ and the fact that the indicial polynomial vanishes at $0$ and $-\alpha$, yield the expressions for polynomials $E_{\bm n}$ and $F_{\bm n}$:
\begin{align*}
	E_{\bm n}&=c_1c_2x^4-[(c_1+c_2)(\alpha+1)-n_1c_1-n_2c_2]x^3+\alpha(\alpha+1)x^2,\\
	F_{\bm n}&=-c_1c_2(n_1+n_2)x^3+\alpha(c_1n_1+c_2n_2) x^2.
\end{align*}
Canceling the common factor $x^2$ in the differential equation yields 
$$
\begin{gathered}
	x^{2} y'''(x)-\left(x^{2}\left(c_{1}+c_{2}\right)-2 x(\alpha+1)\right) y''(x)+\left(x^{2} c_{1} c_{2}-x\left[\left(c_{1}+c_{2}\right)(\alpha+1)-n_1 c_{1}-n_2 c_{2}\right]\right. \\
	+\alpha(\alpha+1)) y'(x)-\left(x c_{1} c_{2}(n_1+n_2)-\alpha\left(n_1 c_{1}+n_2 c_{2}\right)\right) y(x)=0,
\end{gathered}
$$
which matches the equation  found in \cite[section 4.3]{AptBraVA} and \cite[section 5.2]{MR3055365}.

By Theorem \ref{thm:vectorelectrostatics},  
the  discrete vector measure $\vec \nu_1:=\left(\nu(\mathcal L_{\bm n}),\nu(S_{\bm n,1})\right)$ is a critical  vector  measure for the energy functional $\mathcal{E}_{\vec \varphi, a}$, with $a=-1/2$ and
$$
\vec \varphi (x)=\frac{1}{2}  \left( c_1 x - (\alpha +1) \log x ,    (c_2-c_1) x  \right), \quad z=x+iy.
$$

Since for $c>0$ and $\alpha>-1$,
$$
\int_{0}^\infty x^{\alpha}  e ^{- c x}dx = c^{-\alpha-1} \Gamma(\alpha+1),  
$$
we can easily calculate  the moments of $w_i$,  the asymptotic expansion of  $\mathfrak C_{w_i}[\mathcal L_{\bm n}]$ at infinity, and in consequence, $S_{\bm n, i}$ (using formula \eqref{defS12}).

Let us consider the particular case of $n_1=n_2=n$, with $\alpha=1$, $c_1=1$, and $c_2=2$.  
Then, for $\bm n=(5,5)$,
\begin{align*}
	\mathcal L_{(5,5) }(x)  = & x^{10}-\frac{165 x^9}{2}+2750 x^8-\frac{96525 x^7}{2}+487575 x^6-\frac{5831595 x^5}{2}+10239075 x^4\\ & -20270250 x^3+20790000 x^2-9355500 x+1247400 ,
\end{align*}
and up to normalization,
\begin{align*}
	S_{\bm n, 1 }(x) & =   2 x^5+25 x^4+605 x^3+16580 x^2+506065 x+16197810, \\
	S_{\bm n, 2}(x) & =  4 x^5-320 x^4+9975 x^3-151645 x^2+1115560 x-2967600.
\end{align*}
All zeros of $\mathcal L_{(5,5) }$ are positive and simple,  $S_{\bm n, 1}$ has one  negative  and two pairs of complex conjugate simple zeros, while $S_{\bm n, 2}$, has one positive and two pairs of complex conjugate roots, all of them simple.

Asymptotics of  sequences of (rescaled) multiple Laguerre polynomials  of the second kind,
$$
p_{(n,n)}(t) = \mathcal L_{(n,n) }\left(  n t \right),  
$$ 
with varying $0<c_1<c_2$ proportional to $n$, has been studied by Lysov and Wielonsky in \cite{zbMATH05343896} using the Riemann-Hilbert technique and the analysis of the Riemann surface derived from the differential equation. 
In particular, they found that there is a critical value $\kappa \approx 12.11\dots$ such that for $0<c_2/c_1<\kappa$, the support of the limit of the zero-counting measures $\nu(p_{(n,n)})$ is a single interval of the form $[0,d]$, $d>0$, while for $c_2/c_1>\kappa$, it is comprised of two real intervals $[0,a]\cup [b,d]$, with $0<a<b<d$.  The expression of the density was also derived, but  no equilibrium problem associated to that distribution was given. 
It is interesting to compare these conclusions with results of the  numerical experiments presented  in Figure~\ref{FigLaguerreII25}, where we take  $\bm n=(35,35)$. We observe that for small values of  $c_2/c_1$ the zeros of $S_{\bm n, 1 }$ sit on a curve on the complex plane. However,  for large ratios $c_2/c_1$, zeros of  $\mathcal L_{(n,n) }$ split into two groups, and the zeros of $S_{\bm n, 1 }$, all real, approximately interlace with the zeros of $\mathcal L_{(n,n) }$ on the leftmost subinterval. This allows us to conjecture that in this case,  the asymptotic zero distribution can be described again in terms of the Angelesco-type vector equilibrium, see Section~\ref{sec:AsymptAngelesco}.

 \begin{figure}[h]
 	\centering 
 		\hspace{-3mm} \begin{tabular}{cc}
 		\begin{overpic}[scale=0.55]{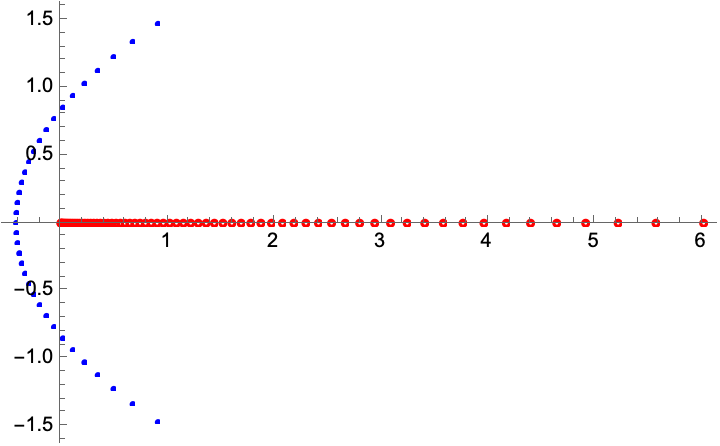}
 		\end{overpic} 
 		& 
 		\begin{overpic}[scale=0.55]{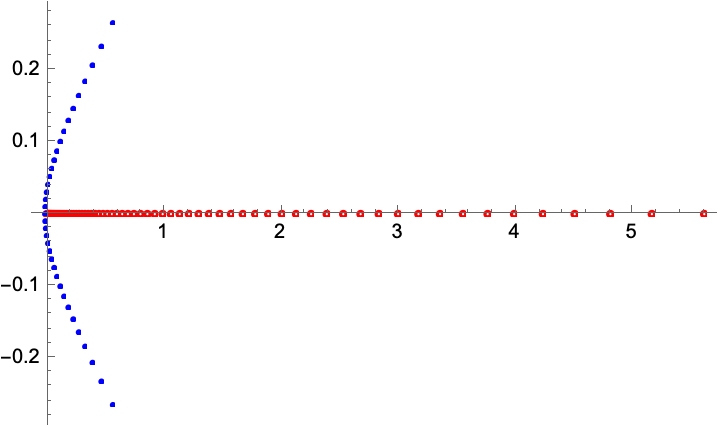}
 		\end{overpic} 
 	\\
 		\begin{overpic}[scale=0.55]{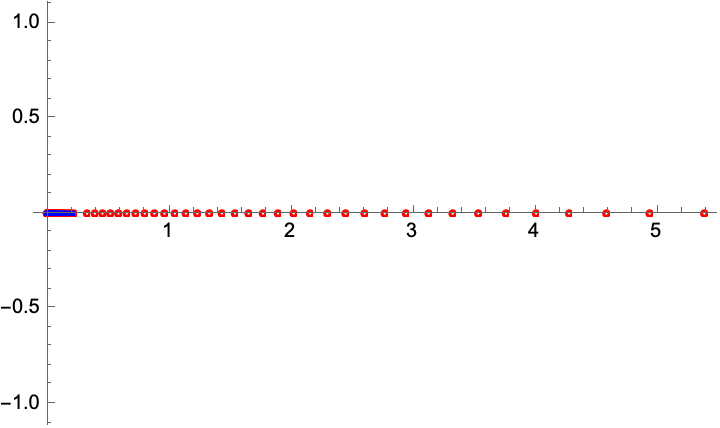}
 		\end{overpic} 
 		& 
 		\begin{overpic}[scale=0.55]{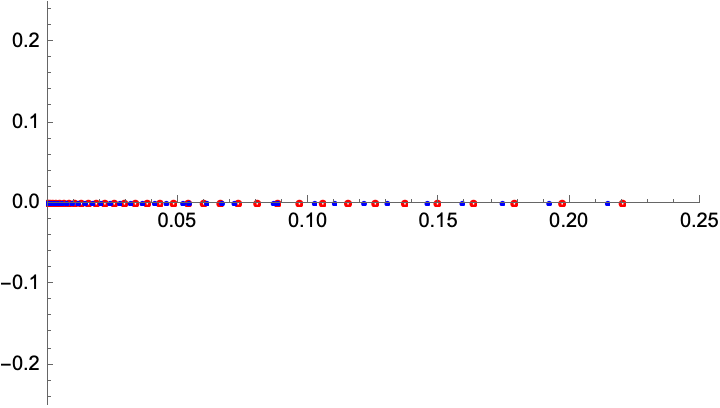}
 		\end{overpic}
 	\end{tabular}
 	\caption{Zeros of  the multiple Laguerre polynomial of the second kind $\mathcal L_{\bm n}$ (indicated by empty circles, all  on  the positive semiaxis) and of $S_{\bm n, 1}$ (filled circles)   for  $\bm n=(35,35)$, and $(c_1,c_2)=(35,70)$ (top left),  $(c_1,c_2)=(35,140)$ (top right),  and $(c_1,c_2)=(35,525)$ (bottom left).  Bottom right: zoom of the interval $(0,0.25)$ for $(c_1,c_2)=(35,525)$.} \label{FigLaguerreII25}
 \end{figure}

\subsection{Jacobi-Pi\~neiro polynomials} \label{sec:JacobiPinero}

The Jacobi-Piñeiro polynomials are multiple orthogonal polynomials associated with an AT system consisting of Jacobi weights on $[0,1]$ with different powers at 0 and the same behavior at $1 $. They are defined by the orthogonality conditions
$$
\begin{aligned}
	&\int_{0}^{1} x^{k} P_{\bm n}(x) x^{\beta_1} (1-x)^\alpha\, dx=0, \quad k=0,1, \ldots, n_1-1,  \\
	&\int_{0}^{1} x^{k} P_{\bm n}(x) x^{\beta_2} (1-x)^\alpha\, dx=0, \quad k=0,1, \ldots, n_2-1.
\end{aligned}
$$
In our notation, $\Delta_1=\Delta_2=[0,1]$ and $w_j(x)= x^{\beta_j} (1-x)^\alpha$,  $j=1, 2$, with  $\alpha, \beta_1, \beta_2>-1$ and $\beta_1-\beta_2\notin \mathbb Z$. We have 
$$
A_1(x)=A_2(x)=x(x-1), \quad B_i(x)= (\beta_i +\alpha) x-\beta_i, \quad \sigma_i=0,  \quad i=1, 2.
$$
For the same reason mentioned at the beginning of Section~\ref{sec:MultLaguerreIkind},   this pair of weight forms a Nikishin system with $\Delta_1=\Delta_2=[0,1]$ and $[c,d]=[-\infty, 0]$, see Section \ref{sec:nikishin}. 

These polynomials were first studied by Piñeiro \cite{MR884516} when $\alpha=0$. The general case appears in \cite[p. 162]{niksor}. 
There is a Rodrigues formula for Jacobi-Pi\~neiro polynomials  $P_{\bm n}$, $\bm n=(n_1, n_2)$,  see \cite[Section 23.3.2]{Ismail05}: with $N=n_1+n_2$,  and up to a constant factor,
\begin{align*}
	P_{\bm n}(x)  &=(1-x)^{-\alpha}  \prod_{j=1}^{2}\left(x^{-\beta_{j}} \frac{d^{n_{j}}}{d x^{n_{j}}} x^{n_{j}+\beta_{j}}\right)(1-x)^{\alpha+N} .
\end{align*}
There is even an explicit expression \cite[Section 3.1]{MR1808581}: again, up to a multiplicative constant,
$$
   P_{\bm n}(x)  =  \sum_{k=0}^{n_1}\binom{	\beta_{1}+n_1 }{
   	k} \binom{	
   	\alpha +N }{
   	n_1-k} \sum_{j=0}^{n_2} \binom{	\beta_{2}+N-k }{		j} \binom{	
		\alpha +k+n_2 }{
		n_2-j}   x^{N-k-j}(x-1)^{k+j}.
$$

Proposition \ref{prop:WronskianS/w} assures that polynomial $R_{\bm n}$ has degree at most $3$. By \eqref{IdR1}, $R_{\bm n}P_{\bm n}$ has a double root at $1$ (observe that $B_1-B_2$ also vanish at $1$) and a simple one at $0$. Since the weights form an AT-system, $P_{\bm n}(0)\neq 0$ and $P_{\bm n}(1)\neq 0$, so that,  up to a multiplicative constant,  $R_{\bm n}(z) =z(z-1)^2$ .

By Theorem \ref{thm:vectorelectrostatics},  and taking into account the expression for $R_{\bm n}$, we conclude that 
the  discrete vector measure $\vec \nu_1:=(\nu(\frak L_{\bm n}),\nu(S_{\bm n,1}))$ is a critical vector  measure for the energy functional $\mathcal{E}_{\vec \varphi, a}$, with $a=-1/2$ and
$$
\vec \varphi = \left(  \frac{\beta_1+1}{2}\log \frac{1}{\left| x\right|  } +  \frac{\alpha+1}{2}\log \frac{1}{\left| x-1\right|  } ,    \left(\beta_1-\beta_2+\frac{1}{2} \right) \log\frac{ 1}{  \left|z \right|}   \right), \quad z=x+iy.   
$$

The third order differential equation can be also obtained following the arguments used in Sections \ref{sec:multipleHermite}--\ref{sec:multipleLag2kind}, and making use of the asymptotics at $\infty$, $0$, and $1$, and of the known roots of the indicial polynomials at both finite points. For instance, in the case $\alpha=0$ (studied by Piñeiro in \cite{MR884516}), polynomials $E_{\bm n}$ and $F_{\bm n}$ (coefficients of $y'$ and $y$, respectively) are
\begin{align*}
	E_{\bm n}&=-x(x-1)^3 ((1 + \beta_1) (1 + \beta_2) +
	x (n_1^2 + n_2^2 + n_1n_2+ n_1(1  + \beta_1) +
	n_2 (1 + \beta_2) - (2 + \beta_1) (2 + \beta_2))),\\
	F_{\bm n}&=-x(x-1)^3 (n_1 + n_2) (1 + n_1 + \beta_1) (1 + n_2 + \beta_2).
\end{align*}
Canceling the common factor $x(x-1)^3$ we obtain  the differential equation 
\begin{align*}
	&x^2(x-1)y'''+x ( x (5 + \beta_1 + \beta_2)-3 - \beta_1 - \beta_2 )y''\\
	&-(x (n_1^2 + n_2^2 + n_1n_2+ n_1(1 + \beta_1) +
	n_2 (1 + \beta_2) - (2 + \beta_1) (2 + \beta_2))-(1 + \beta_1) (1 + \beta_2) 
	)y'\\
	&-(n_1 + n_2) (1 + n_1 + \beta_1) (1 + n_2 + \beta_2)y=0.
\end{align*}
This is the particular case (after canceling the common factor $x-1$)  of  the equation derived in \cite[Section 4.3]{AptBraVA}.

As for the electrostatic partners, let us consider the example when $\alpha=\beta_1=0$ and $\beta_2=-1/2$. Direct computation shows that the moments of $w_i$ are
$$
\int_{0}^1 x^{k }  w_j(x) dx = \begin{cases}
(k+1)^{-1}, & j=1,  \\
2/(2j+1), & j=2,
\end{cases} \quad k=0, 1, \dots ,
$$
which allows to find the asymptotic expansion of  $\mathfrak C_{w_i}[P_{\bm n}]$ at infinity, and in consequence, $S_{\bm n, i}$ (using formula \eqref{defS12}) by means of symbolic computation. For instance, in the case $\bm n=(5,5)$ we obtain that 
\begin{align*}
	\mathcal L_{(5,5) }(x)    = & x^{10}-\frac{380 x^9}{87}+\frac{1615 x^8}{203}-\frac{20672 x^7}{2639}+\frac{9044 x^6}{2001}-\frac{5168 x^5}{3335}+\frac{204
		x^4}{667} \\ & -\frac{64 x^3}{2001}+\frac{x^2}{667}-\frac{4 x}{182091}+\frac{1}{30045015},
\end{align*}
and up to normalization,
\begin{align*}
	S_{\bm n, 1 }(x)   = &\  882230895 x^5+4709406975 x^4+5720142090 x^3+8795888965 x^2 \\
	& +11696347475 x+11645469674, \\
	S_{\bm n, 2}(x)   = & \ 5192762585 x^5+313459871725 x^4+662076961780 x^3+782465377400 x^2 \\
	& +1267133219685 x+1386883054197.
\end{align*}
All zeros of $P_{(5,5) }$ are positive and simple,  both $S_{\bm n, j}$ have one  negative  and two pairs of complex conjugate simple zeros,  all of them simple. 

Zero asymptotics for  sequences of Jacobi-Pi\~neiro polynomials  $P_{(n,n)}$ as $n\to \infty$ and $\alpha$, $\beta_j$'s fixed, was obtained in \cite{zbMATH05356762} and \cite{MR3471160}. Again, the expression of the density, this time on $[0,1]$, was derived from the recurrence relation satisfied by polynomials $\frak L_{(n_1,n_2) }$ and no equilibrium problem associated to that distribution was given. For comparison, results of the  numerical experiments are presented  in Figure~\ref{FigJP25}, where we take $\bm n=(75,75)$, with $\beta_1=0$, $\beta_2=-1/2$, and $\alpha=0$. According to our discussion in Section~\ref{sec:nikishin},  the zeros of  $S_{\bm n, 1}$ are real and negative, and the asymptotic zero distribution can be described  in terms of the Nikishin-type vector equilibrium, see Section~\ref{sec:AsymptNikishin}.

 \begin{figure}[h]
	\centering 
		\hspace{-3mm} \begin{tabular}{cc}
		\begin{overpic}[scale=0.6]{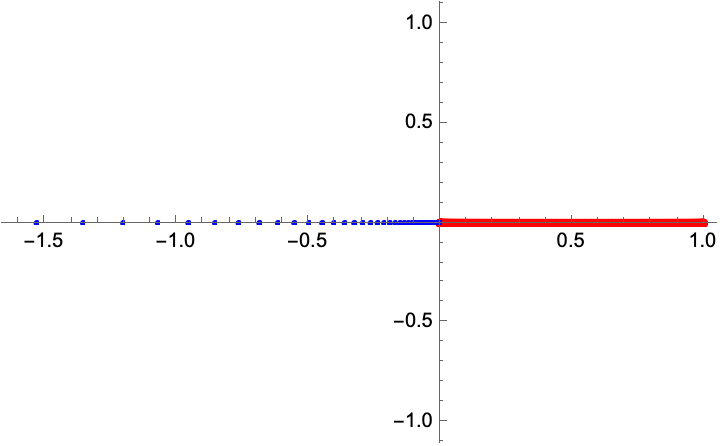}
		\end{overpic} 
			& 
				\begin{overpic}[scale=0.6]{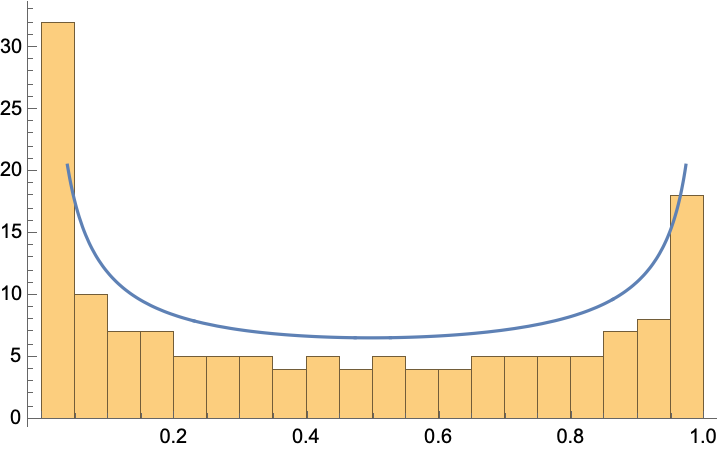}
					\end{overpic}
			\end{tabular}
	\caption{Left: zeros of  the Jacobi-Pi\~neiro polynomial   $P_{\bm n}$  (all  on  $[0,1]$) and of $S_{\bm n, 1}$ (filled circles, all negative)   for  $\bm n=(75,75)$, with $\beta_1=0$, $\beta_2=-1/2$, and $\alpha=0$;   approximately 19 real zeros of $S_{\bm n, 1}$ (ranging from  $-771$ to $-1.74$), are not represented. Right: the histogram of the zeros of $P_{\bm n}$ and the plot of the asymptotic density, predicted in \cite{zbMATH05356762}.} \label{FigJP25}
\end{figure}

\subsection{Angelesco-Jacobi polynomials} \label{sec:AngelescoJacobi}

These polynomials, known also as Jacobi-Jacobi  polynomials (see \cite{Aptekarev:97}) are Hermite--Pad\'e polynomials $P_{\bm n}$, $\bm n=(n_1, n_2)$, satisfying orthogonality relations
\begin{equation*}
\begin{split}
	&\int_{a}^{0} x^{k} P_{\bm n}(x)(x-a)^{\alpha}|x|^{\beta}(1-x)^{\gamma}  \, d x=0, \quad k=0,1,2, \dots, n_1-1, \\
&\int_{0}^{1} x^{k} P_{\bm n}(x)(x-a)^{\alpha} x^{\beta}(1-x)^{\gamma} x^{k} \, d x=0, \quad k=0,1,2, \dots, n_2-1,
\end{split}
\end{equation*}
with $a<0$. In our notation,
	$$
	A(x)=x(x-a)(x-1),  \quad B(x)= \alpha x (x-1)+ \beta (x-a)(x-1)+ \gamma x(x-a),
	$$
	and
	\begin{align*}
		w_1(x) & =w_2(x)=w(x)=(x-a)^{\alpha}|x|^{\beta}(1-x)^{\gamma}, \\ 
		v_1(x) & =v_2(x)=v(x)=(x-a)^{\alpha+1}|x|^{\beta+1}(1-x)^{\gamma+1},
	\end{align*}
	with $\alpha, \beta, \gamma>-1$ and 
	$$
	\Delta_1=[a,0], \quad \Delta_2=[0,1]
	$$
	(cf.~Example~\ref{exampleAngelescoJJ}). 
	
	This is an Angelesco system of semiclassical weights of class $\sigma=1$, see \eqref{def:class}. As it follows from  \cite[Theorem 2.1]{Aptekarev:97}, for $n_1=n_2=n$, $\bm n =(n,n)$, $n\in \N$, $P_{\bm n}$ can be expressed using a Rodrigues formula,
	$$
	P_{\bm n}(x)=\frac{1}{w(x)} \left( \frac{d}{dx}\right)^n \left[ A^n(x)w (x) \right],
	$$
	or  explicitly, see   \cite[Section 3.5]{MR1808581}: up to a constant factor, 
	$$
P_{\bm n}(x)=  \sum_{k=0}^{n} \sum_{j=0}^{n-k} \frac{(-n)_{k+j} (-\alpha-n)_{j} (-\gamma-n)_{k} }{(\beta+1)_{k+j} k ! j !}(x-a)^{n-k}(x-1)^{k+j} x^{n-j}.
$$
	
	Kaliaguin \cite{kaliaguine:1981} studied the case of $a=-1$, with the particular sub-case of   $B\equiv 0$ or $w(x)\equiv 1$ going back   to the work of Appell \cite{Appell1901}. The case of   $-1<a<0$ was addressed in \cite{kaliaguine:1996}, but see \cite{Aptekarev:97} for further historical details. 
	
	If $B\equiv 0$ (that is, $\alpha=\beta=\gamma=0$), definition \eqref{defTransform2} reduces to 
	$$
	S_{\bm n, 1} =  A \times \mathfrak{Wrons}\left[P_{\bm n} ,\mathfrak C_{w_1}[P_{\bm n}] \right] , 
	 \quad \mathfrak C_{w_1}[P_{\bm n}] (x)=\int_{a}^0 \frac{P_{\bm n} (t) }{t-x}dt.
	$$
	Notice that in particular, $x=1$ is one of the $n+1$ zeros of $S_{\bm n, 1} $, $n-1$ of which interlace with the zeros of $P_{\bm n}$ on $[0, 1)$.
	
	\begin{figure}[h]
		\centering 
		\hspace{-3mm} \begin{tabular}{cc}
			\begin{overpic}[scale=0.55]{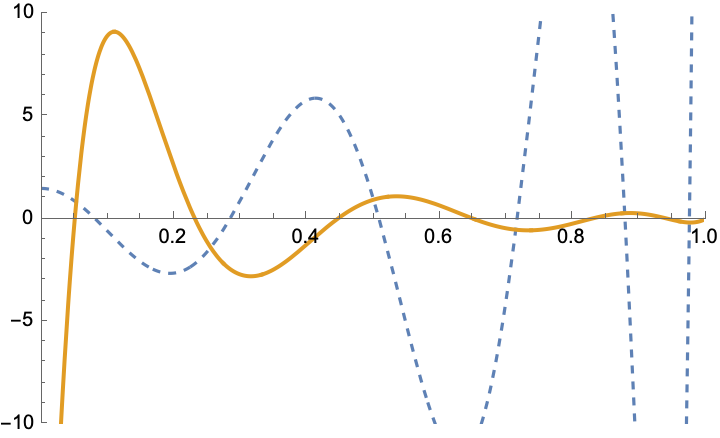}
			\end{overpic} 
			& 
			\begin{overpic}[scale=0.55]{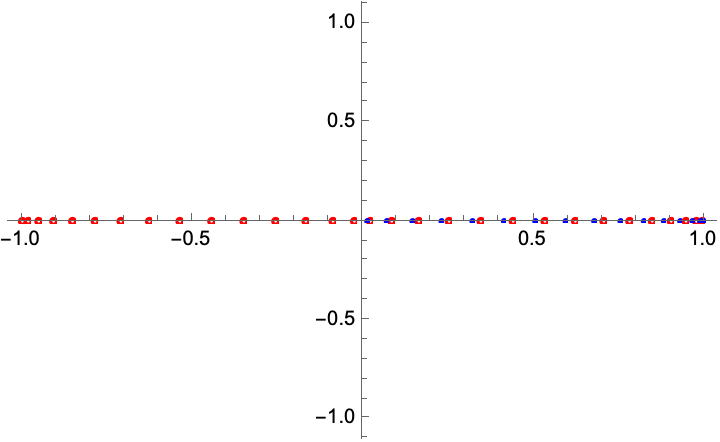}
			\end{overpic}
		\end{tabular}
		\caption{Appell's polynomials ($\alpha=\beta=\gamma=0$). Left: graph of  $P_{\bm n}$ (dashed line) and $S_{\bm n, 1}$ (thick line)  on $[0,1]$ for  $\bm n=(6,6)$ in the   case. Right: zeros of    $P_{\bm n}$  (empty circles, all  on  $[-1,1]$) and of $S_{\bm n, 1}$ (filled circles, all on $[0,1]$)   for  $\bm n=(15,15)$. }\label{Fig1Appel}
	\end{figure}

	In Appell's case, $w\equiv 1$ and $a=-1$,  
	$$
	P_{\bm n}(x)= \left( \frac{d}{dx}\right)^n \left[ x(x^2-1)   \right]^n  ,
	$$
	and
	$$
	S_{\bm n, 1} (x)=(-1)^{n+1} 	S_{\bm n, 2} (-x), \quad \deg S_{\bm n, 1}=n+1.
	$$
	These explicit formulas allow to use symbolic computation (e.g.~Mathematica) to find explicit expressions for small values of $n$. For instance, for $\bm n=(6,6)$,  and up to normalization,
	\begin{align*}
	P_{\bm n}(x) & = x^{12}-\frac{44 x^{10}}{17}+\frac{165 x^8}{68}-\frac{220 x^6}{221}+\frac{75 x^4}{442}-\frac{2 x^2}{221}+\frac{1}{18564}, \\
	 S_{\bm n, 1}(x) &= (x-1) \left(12288 x^6-38763 x^5+47253 x^4-27822 x^3+8018 x^2-991 x+33\right),
	\end{align*}
	and their graphs are plotted in Figure~\ref{Fig1Appel}. We can clearly observe the interlacing predicted by Proposition~\ref{cor:zerosS12ang}, which in the limit $n\to\infty$ gives the description in therms of the Angelesco equilibrium problem, as described in Section~\ref{sec:AsymptAngelesco}.

However, according to Remark~\ref{rem:linearcomb}, the critical configuration for the zeros of $P_{\bm n}$ in this case is not unique: we can use for the second component the zeros of any linear combination of $S_{\bm n, 1}$ and $S_{\bm n, 2}$. In Figure~\ref{Fig2Appell} we illustrate the behavior of zeros of $S_{\bm n, 1} + t S_{\bm n, 2}$, for different values of $0<t\leq 1$ (notice that the representation of the zeros of $S_{\bm n, 1}$ in Figure~\ref{Fig1Appel}, right, corresponds to $t=0$). 

\begin{figure}[h]
	\centering 
	\hspace{-3mm} \begin{tabular}{cc}
		\begin{overpic}[scale=0.55]{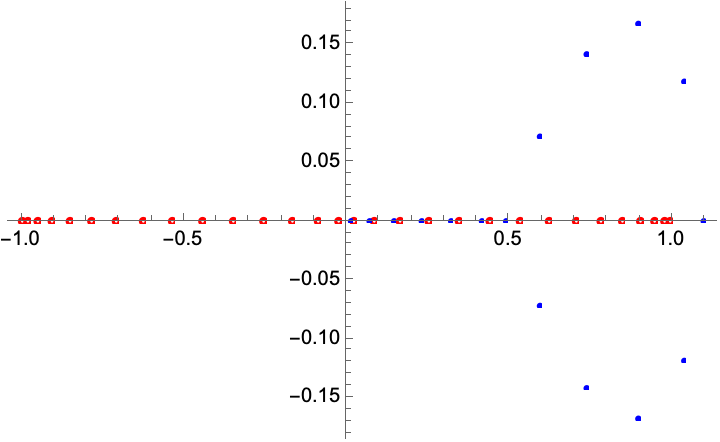}
		\end{overpic} 
		& 
		\begin{overpic}[scale=0.55]{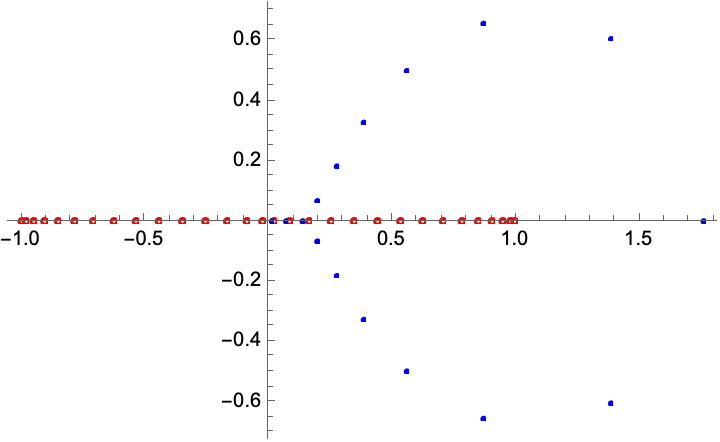}
		\end{overpic} \\
	\begin{overpic}[scale=0.55]{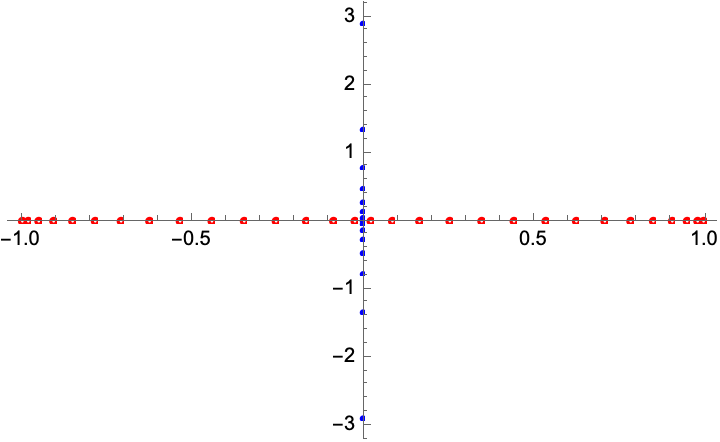}
	\end{overpic} 
	& 
	\begin{overpic}[scale=0.55]{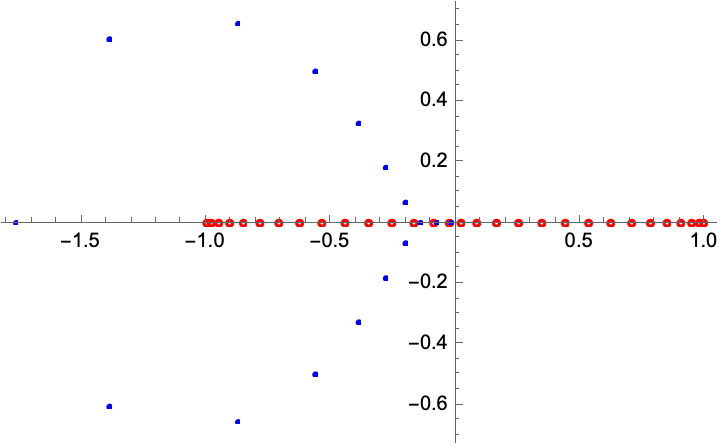}
	\end{overpic}
	\end{tabular}
	\caption{Appell's polynomials ($\alpha=\beta=\gamma=0$) and $\bm n=(35,35)$: zeros of $P_{\bm n}$ (indicated by empty circles, all  on $(-1,1)$) and of $S_{\bm n, 1} + t S_{\bm n, 2}$ (filled circles)  for  $t=10^{-10}$ (top left), $t=10^{-5}$ (top right), $t=1$ (bottom left), and $t=10^{5}$ (bottom right). }\label{Fig2Appell}
\end{figure}

\subsection{Multiple orthogonal polynomials for the cubic weight}

Although we have not discussed the purely complex weights, we finish our presentation with the illustrative example of polynomials  $P_{\bm n}$, $\bm n=(n_1, n_2)$, satisfying orthogonality relations
\begin{equation*}
	\begin{split}
		&\int_{\Delta_1} z^{k} P_{\bm n}(z)e^{-z^3} d z=0, \quad k=0,1,2, \dots, n_1-1, \\
		&\int_{\Delta_2} z^{k} P_{\bm n}(z)e^{-z^3} d z=0, \quad k=0,1,2, \dots, n_2-1,
	\end{split}
\end{equation*}
where  $\Delta_1$ and $\Delta_2$ are contours on the complex plane, extending to $\infty$ on their two ends along the directions determined by the angles $-2\pi/3$ and $0$, and $-2\pi/3$ and $2\pi/3$, respectively. They were introduced in \cite{MR3304586}  and studied in full generality in \cite{MR3939592}.

In our notation,
$$
A_1(x)=A_2(x)=A(x)=1,  \quad B_1(x)=B_2(x)=B(x)= -3x^2.
$$
Detailed explanation of the algorithm of computation of the zeros of $P_{\bm n}$ was given in \cite[Section 9]{MR3939592}. Since the moments of the weight were given explicitly, we use the expansion of $\mathfrak D_{w }[P_{\bm n}]$ at infinity to calculate the expressions for $S_{\bm n, 1} $ and $ S_{\bm n, 2}$, see Figure~\ref{FigCubic}.

It is interesting to compare the location of the zeros of $S_{\bm n, 1} $ with those of  \textit{type I} multiple orthogonal polynomials  $\mathfrak A_{\bm n}$ and $\mathfrak B_{\bm n}$,  defined by the following conditions:
$$
\deg \mathfrak A_{\bm n} \leq n-1, \quad \deg \mathfrak B_{\bm n}\leq  m-1,
$$
and
\begin{equation}\label{mops_conditionsTypeI}
	\begin{aligned}
		\int_{\Delta_1}z^k \mathfrak A_{\bm n}(z) e^{- z^3}dz + 	\int_{\Delta_2}z^k \mathfrak B_{\bm n}(z) e^{- z^3}dz =0, & \quad k=0,\dots, N-2, \\
		\int_{\Delta_1}z^k \mathfrak A_{\bm n}(z) e^{- z^3}dz + 	\int_{\Delta_2}z^k \mathfrak B_{\bm n}(z) e^{- z^3}dz =1, & \quad k=N-1,
	\end{aligned}
\end{equation}
where $N=n+m$. According to Figure~\ref{FigCubic}, the zeros of $\mathfrak B_{\bm n}$ ``interlace'' with the zeros of $S_{\bm n, 1} $, which brings up a natural question of a possible connection of these two polynomials.

\begin{figure}[h]
	\centering 
		\begin{overpic}[scale=0.7]{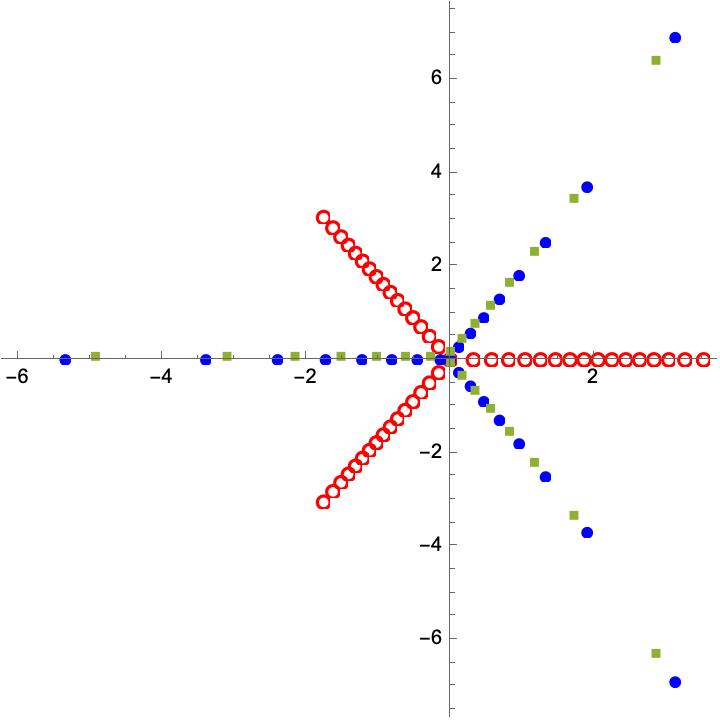}
		\end{overpic} 
	\caption{Zeros of  the multiple orthogonal polynomial  $P_{\bm n}$ with respect to the cubic weight  (indicated by empty circles, part of them on the positive semiaxis, forming a symmetric star) and of $S_{\bm n, 1}$ (filled circles)   for  $\bm n=(25,25)$.  For comparison, zeros of type I MOP $\mathfrak B_{\bm n}$ (filled squares) are also represented. }\label{FigCubic}
\end{figure}

\appendix

\section{Properties of the electrostatic partner} \label{appendixA}

\begin{prop} \label{prop:fromTHM2.1}
	If $P$ is a polynomial of degree $N\in \N$, $q$ the function of the second kind \eqref{defqN}, and $S$ is the electrostatic partner of $P$ defined in \eqref{defCompanion} then:
	\begin{enumerate}[a)]
		\item if $z_0\in \C$ is a zero of $P$ of multiplicity $k\geq 1$ then $S$ also has a root at $z_0$ with multiplicity at least $k-1$. If in addition $A(z_0)=0$ then the multiplicity of $z_0$ in $S$ is at least $k$.
		
		\item if $z_0\in \C\setminus \Delta $ is a zero of $ \mathfrak C_w[P]$ of multiplicity $k\geq 1$ then $S$ also has a root at $z_0$ with multiplicity at least $k-1$. If in addition $A(z_0)=0$ then the multiplicity of $z_0$ in $S$ is at least $k$.
		
		\item Let $\Omega$ be a simply-connected domain such that $A(z)\neq 0$ for $z\in \Omega$, $\mathfrak C_w[P]$ holomorphic in $\Omega$ and let $w$ be a holomorphic branch of this function in $\Omega$. If $z_0$ is a boundary point of $\Omega$, then
		\begin{equation} \label{limitbounded}
			\lim_{z\to z_0} (z-z_0) A^k (z) w(z)\frac{d^k}{dz^k}q(z)=0,\qquad k=0,1,\dots,
		\end{equation}
		where we take non-tangential limit with $z\in \Omega$.
	\end{enumerate}
\end{prop}
\begin{proof}
	Assume $z_0\notin \Delta$. If $z_0\in \C$ is a zero of $p$ of multiplicity $k\geq 1$ then the assertion in a) follows directly from the definition \eqref{defTransform2}. Same argument works to prove a) if we assume that $z_0\in \C$ is a zero of $\mathfrak C_w[P]$ of multiplicity $k\geq 1$. 
	
	On the other hand, from the assumption \eqref{endpoints} it follows that for $z_0\in \Delta$,
	\begin{equation} \label{local1}
		\lim_{z\to z_0} (z-z_0) \,  \mathfrak C_w[P](z)=0,  
	\end{equation}
	which proves b) also in this case. 
	
	We turn now to c). 	If $A(z_0)\neq 0$ then both $w^{-1}$ and $q$ are bounded at $z_0$ and the assertion is obvious. Hence, let $A(z_0)=0$.

	Denote
	$$
	h_k(z):=(z-z_0) A^k(z) w(z)\frac{d^k}{dz^k}q(z), \quad k=0, 1, 2, \dots
	$$
	Obviously, $h_0(z)= (z-z_0) \,  \widehat p (z)$, and \eqref{limitbounded} for $k=0$ is consequence of \eqref{local1}. 
	
	Using the induction in $k$, assume that  \eqref{limitbounded} is establised for a certain $k\ge 0$, i.e.
	$$
	\lim_{z\to z_0}  h_k(z)=0,
	$$ 
	where we always  take non-tangential limit from $ \Omega$. 
	By Lemma~\ref{lemmaBD} below, 
	\begin{equation} \label{local2}
		\lim_{z\to z_0} (z-z_0) h'_k(z)=0.
	\end{equation} 
	Thus, using \eqref{ratW}, 
	\begin{align*}
		0 & = \lim_{z\to z_0} (z-z_0) h'_k(z) = \lim_{z\to z_0} A(z) h'_k(z) \\
		& = \lim_{z\to z_0} \left[ h_k(z) \left(\frac{A(z)}{z-z_0} + kA'(z) + A(z) \frac{w'(z)}{w(z)} \right) + h_{k+1}(z) \right] \\
		& =  \lim_{z\to z_0} \left[  h_k(z) \left(\frac{A(z)}{z-z_0} + kA'(z) + B(z)  \right) + h_{k+1}(z) \right] =  \lim_{z\to z_0}h_{k+1}(z),
	\end{align*}
	which proves \eqref{limitbounded}. 
\end{proof}

\begin{lem}\label{lemmaBD}
	Let $\Omega$ be a domain, $z_0\in \partial \Omega$ a boundary point satisfying the following property: there exists a sector  $\Gamma:=\{ z\in \C:\, |\arg(z-z_0) - \theta_0|<\delta \}$, with certain $\theta_0\in [0,2\pi)$, $\delta>0$, such that for a sufficiently small $r>0$, $\Gamma \cap \{ z\in \C:\,  |z-z_0|<r\}\subset \Omega$.
	If function $f$ is holomorphic in $\Omega$ and 
	$$
	\lim_{z\to z_0, \, z\in \Gamma} \frac{f(z)}{z-z_0}=0,
	$$
	then
	$$
	\lim_{z\to z_0, \, z\in \Gamma} f'(z)=0.
	$$
\end{lem}
\begin{proof}
	Without loss of generality we can assume that $z_0=0$ and $\theta_0=0$.
	
	Each $0<q<1$ defines a sub-sector  $\Gamma_q$ of $\Gamma$ given by
	$$
	\Gamma_q:=\{ z\in \Gamma:\, \left|\tan \left(\delta - \arg(z) \right)\right|  \ge q \}.
	$$
	Notice that as $q\to 0$, $\Gamma_q$ exhausts $\Gamma$.
	
	Fix a  $0<q<1$ and an arbitrarily small $\varepsilon>0$. By assumptions, there exists $0<r'=r'(\varepsilon)<r$ such that 
	$$
	t\in \Gamma, \, |t|<r' \quad \Rightarrow \quad |f(t)|\le \varepsilon |t|.
	$$
	Let $z\in \Gamma_q$; by construction, the circle $C_z:= \{t\in \C:\, |t-z|= q |z| \} \subset \Gamma$. We assume $|z|$ small enough so that $|t|<r'$ for all  $t\in C_z$. Then
	$$
	f'(z)=\frac{1}{2\pi i} \oint_{C_z} \frac{f(t)}{(t-z)^2}dt,
	$$
	and
	$$
	|f'(z)|\le \frac{1}{q |z|} \max_{t\in C_z} |f(t)| \le  \frac{\varepsilon}{q |z|} \max_{t\in C_z} |t|=\frac{\varepsilon}{q  } \frac{|z| + q|z|}{|z|}= \frac{\varepsilon (1+q)}{q  },
	$$
	which proves the assertion.
\end{proof}

\section{Electrostatic partner in the real case} \label{sec:realcase}

The construction in Section \ref{sec:quasiorth} is carried out in a very general setup. In this Appendix, we prove a technical result valid in the real case, when $A$ and $B$ have real coefficients, 
\begin{equation} \label{realcase}
	\Delta\subset \mathbb R \quad \text{ and } \quad w(x)\geq 0 \text{ on } \Delta.
\end{equation}
We assume that $P \not \equiv 0$ is a  polynomial with real coefficients, and as before, denote   $S=\mathfrak D_{w}[P] $.
\begin{lem} \label{propInterlacing1}
	Let $\zeta_1<\zeta_2$ be two consecutive real zeros of $P$, such that $(\zeta_1,\zeta_2)  \cap \Delta=\emptyset$, $S(\zeta_j)\ne 0$ for $j=1, 2$, and $A$ preserves sign on $(\zeta_1,\zeta_2)$. Then $S\,  \mathfrak C_w[P]$ does change sign in $(\zeta_1,\zeta_2)$.
	
	Analogously, if $y_1<y_2$ are two consecutive real zeros of $\mathfrak C_w[P]$, such that $(y_1,y_2)  \cap \Delta=\emptyset$, $(AS)(y_j)\ne 0$ for $j=1, 2$, and $A$ preserves sign on $(y_1,y_2)$, then $S P $  changes sign in $(y_1,y_2)$.
\end{lem}
\begin{proof}
	Since $S(\zeta_j)\ne 0$, $j=1, 2$, it follows from b) in  Proposition~\ref{prop:fromTHM2.1} that these are simple zeros of $P$, $A(\zeta_j)\ne 0$, $j=1, 2$, and thus, $P'(\zeta_1)  P'(\zeta_2)<0$. 	
	Evaluating in the definition \eqref{defTransform2}, we get
	$$
	S(\zeta_j) =\mathfrak D_{w}[P](\zeta_j)  = - A(\zeta_j) \mathfrak C_w[P] (\zeta_j) P' (\zeta_j), \quad j=1, 2,
	$$
	so that
	$$
	\left( 	S\,  \mathfrak C_w[P] \right)(\zeta_j)   = - \left( A  (\mathfrak C_w[P])^2    P' \right) (\zeta_j), \quad j=1, 2.
	$$
	Since $A$ preserves sign on $[\zeta_1,\zeta_2]$, we get that   
	$$
	\left( 	S \mathfrak C_w[P]  \right) (\zeta_1) \left( 	S \mathfrak C_w[P]	  \right) (\zeta_2) <0.
	$$
	This proves the first assertion. 
	
	Similarly from \eqref{defTransform2}, 
	$$
	S(y_j)= \left( A  \left( \mathfrak C_w[P]\right)'   P \right)  (y_j), \quad j=1, 2,
	$$
	so that
	$$
	(SP)(y_j)= \left( A  \left( \mathfrak C_w[P]\right)'   P^2\right)  (y_j), \quad j=1, 2.
	$$
	If $y_1<y_2$ do not coincide with the zeros of $AS $, then as \eqref{odegeneral} shows, these are simple zeros of $\widehat p$, so that
	$$
	\left( \mathfrak C_w[P] \right)' (y_1) \left( \mathfrak C_w[P] \right)' (y_2) <0,
	$$
	and we conclude that
	$$
	\left( S  P \right) (y_1)\left( S P \right) (y_2)<0.
	$$
\end{proof}

\section*{Acknowledgments}

The first author was partially supported by Simons Foundation Collaboration Grants for Mathematicians (grant 710499).
He also acknowledges the support of the Spanish Government and the European Regional Development Fund (ERDF) through grant PID2021-124472NB-I00, Junta de Andaluc\'{\i}a (research group FQM-229 and Instituto Interuniversitario Carlos I de F\'{\i}sica Te\'orica y Computacional), and by the University of Almer\'{\i}a (Campus de Excelencia Internacional del Mar  CEIMAR) in the early stages of this project. 

The second and third authors were partially supported by Spanish Ministerio de Ciencia, Innovación y Universidades, under grant MTM2015-71352-P.

The third author was additionally supported by Junta de Andalucía (research group FQM-384), the University of Granada  (Research Project ERDF-UGR A-FQM-246-UGR20), and by the IMAG–Maria de Maeztu grant CEX2020-001105-M/AEI/10.13039/501100011033.

The authors also grateful to Alexandre Eremenko, who suggested the idea of the proof of Lemma~\ref{lemmaBD}, and to the anonymous referees, whose outstanding work helped to considerably improve this manuscript.


%

\begin{thebibliography}{100}
	
	\bibitem{MR501106}
	M.~Adler and J.~Moser.
	\newblock On a class of polynomials connected with the {K}orteweg-de {V}ries
	equation.
	\newblock {\em Comm. Math. Phys.}, 61(1):1--30, 1978.
	
	\bibitem{angelesco}
	M.~A. Angelesco.
	\newblock Sur deux extensions des fractions continues alg\'ebriques.
	\newblock {\em C.R. Acad. Sci. Paris}, (18):262--263, 1919.
	
	\bibitem{Appell1901}
	P.~Appell.
	\newblock Sur une suite des polyn\^omes ayant toutes leurs racines r{\'e}elles.
	\newblock {\em Arch. Math. Phys.}, 1(3):69--71, 1901.
	
	\bibitem{MR1702555}
	A.~I. Aptekarev.
	\newblock Strong asymptotics of polynomials of simultaneous orthogonality for
	{N}ikishin systems.
	\newblock {\em Mat. Sb.}, 190(5):3--44, 1999.
	
	\bibitem{MR2475084}
	A.~I. Aptekarev.
	\newblock Asymptotics of {H}ermite-{P}ad\'e approximants for a pair of
	functions with branch points.
	\newblock {\em Dokl. Akad. Nauk}, 422(4):443--445, 2008.
	\newblock English transl. in {Dokl. Math.} 78:2 (2008), 717-719.
	
	\bibitem{MR2172687}
	A.~I. Aptekarev, P.~M. Bleher, and A.~B.~J. Kuijlaars.
	\newblock Large {$n$} limit of {G}aussian random matrices with external source.
	{II}.
	\newblock {\em Comm. Math. Phys.}, 259(2):367--389, 2005.
	
	\bibitem{AptBraVA}
	A.~I. Aptekarev, A.~Branquinho, and W.~Van~Assche.
	\newblock Multiple orthogonal polynomials for classical weights.
	\newblock {\em Transactions AMS}, 355:3887--3914, 2003.
	
	\bibitem{MR870267}
	A.~I. Aptekarev and V.~A. Kalyagin.
	\newblock Analytic properties of functions representable by
	{$P^{(m)}$}-fractions with periodic coefficients.
	\newblock {\em Akad. Nauk SSSR Inst. Prikl. Mat. Preprint}, (57):12, 1986.
	
	\bibitem{MR2963452}
	A.~I. Aptekarev and A.~B.~J. Kuijlaars.
	\newblock Hermite-{P}ad\'e approximations and multiple orthogonal polynomial ensembles.
	\newblock \textit{Russ. Math. Surv}. 66, No. 6, 1133--1199 (2011); translation from Usp. Mat. Nauk 66, No. 6, 123--190 (2011).
	
	\bibitem{MR2656323}
	A.~I. Aptekarev and V.~G. Lysov.
	\newblock Systems of {M}arkov functions generated by graphs and the asymptotics
	of their {H}ermite-{P}ad\'{e} approximants.
	\newblock {\em Mat. Sb.}, 201(2):29--78, 2010.
	\newblock English trans. in {Sb. Math.} 201:2 (2010), 183--234.
	
	\bibitem{ALT2011}
	A.~I. Aptekarev, V.~G. Lysov, and D.~N. Tulyakov.
	\newblock Random matrices with an external source and the asymptotics of
	multiple orthogonal polynomials.
	\newblock {\em Mat. Sb.}, 202(2):3--56, 2011.
	
	\bibitem{Aptekarev:97}
	A.~I. Aptekarev, F.~Marcell\'an, and I.~A. Rocha.
	\newblock Semiclassical multiple orthogonal polynomials and the properties of
	{J}acobi-{B}essel polynomials.
	\newblock {\em J.\ Approx.\ Theory}, 90:117--146, 1997.
	
	\bibitem{MR1240781}
	A.~I. Aptekarev and H.~Stahl.
	\newblock Asymptotics of {H}ermite-{P}ad\'e polynomials.
	\newblock In {\em Progress in approximation theory ({T}ampa, {FL}, 1990)},
	volume~19 of {\em Springer Ser. Comput. Math.}, pages 127--167. Springer, New
	York, 1992.
	
	\bibitem{MR3602528}
	A.~I. Aptekarev, W.~Van~Assche, and M.~L. Yattselev.
	\newblock Hermite-{P}ad\'{e} approximants for a pair of {C}auchy transforms
	with overlapping symmetric supports.
	\newblock {\em Comm. Pure Appl. Math.}, 70(3):444--510, 2017.
	
	\bibitem{MR2305271}
	H.~Aref.
	\newblock Vortices and polynomials.
	\newblock {\em Fluid Dynam. Res.}, 39(1-3):5--23, 2007.
	
	\bibitem{MR3335711}
	A.~M. Barry, F.~Hajir, and P.~G. Kevrekidis.
	\newblock Generating functions, polynomials and vortices with alternating signs
	in {B}ose-{E}instein condensates.
	\newblock {\em J. Phys. A}, 48(15):155205, 18, 2015.
	
	\bibitem{MR824780}
	A.~B. Bartman.
	\newblock A new interpretation of the {A}dler-{M}oser {K}d{V} polynomials:
	interaction of vortices.
	\newblock In {\em Nonlinear and turbulent processes in physics, {V}ol. 3
		({K}iev, 1983)}, pages 1175--1181. Harwood Academic Publ., Chur, 1984.
	
	\bibitem{BertolaGrava2022}
	M.~Bertola, E.~Chavez-Heredia, and T.~Grava.
	\newblock The {S}tieltjes--{F}ekete problem and degenerate orthogonal
	polynomials.
	\newblock Preprint {arXiv:}2206.06861, 2022.
	
	\bibitem{MR2743878}
	P.~Bleher, S.~Delvaux, and A.~B.~J. Kuijlaars.
	\newblock Random matrix model with external source and a constrained vector
	equilibrium problem.
	\newblock {\em Comm. Pure Appl. Math.}, 64(1):116--160, 2011.
	
	\bibitem{Bleher/Kuijlaars1}
	P.~Bleher and A.~B.~J. Kuijlaars.
	\newblock Large {$n$} limit of {G}aussian random matrices with external source.
	{I}.
	\newblock {\em Comm. Math. Phys.}, 252(1-3):43--76, 2004.
	
	\bibitem{MR2187942}
	P.~M. Bleher and A.~B.~J. Kuijlaars.
	\newblock Integral representations for multiple {H}ermite and multiple
	{L}aguerre polynomials.
	\newblock {\em Ann. Inst. Fourier (Grenoble)}, 55(6):2001--2014, 2005.
	
	\bibitem{MR2276453}
	P.~M. Bleher and A.~B.~J. Kuijlaars.
	\newblock Large {$n$} limit of {G}aussian random matrices with external source.
	{III}. {D}ouble scaling limit.
	\newblock {\em Comm. Math. Phys.}, 270(2):481--517, 2007.
	
	\bibitem{Bocher97}
	M.~B{\^o}cher.
	\newblock The roots of polynomials that satisfy certain differential equations
	of the second order.
	\newblock {\em Bull.\ Amer.\ Math.\ Soc.}, 4:256--258, 1897.
	
	\bibitem{MR412503}
	J.~W. Brown.
	\newblock On {A}ngelesco-type polynomials.
	\newblock {\em Ricerca (Napoli) (3)}, 24(maggio-agosto):3--7, 1973.
	
	\bibitem{MR1576413}
	J.~L. Burchnall and T.~W. Chaundy.
	\newblock A {S}et of {D}ifferential {E}quations which can be {S}olved by
	{P}olynomials.
	\newblock {\em Proc. London Math. Soc. (2)}, 30(6):401--414, 1930.
	
	\bibitem{Bulo}
	J.~Bustamante and G.~L\'{o}pez~Lagomasino.
	\newblock Hermite--Pad\'{e} approximation for Nikishin systems of analytic
	functions.
	\newblock {\em Russian Acad. Sci. Sb. Math.}, 77:367--384, 1994.
	
	\bibitem{zbMATH07222111}
	K.~{Castillo}, M.~N. {de Jesus}, and J.~{Petronilho}.
	\newblock {An electrostatic interpretation of the zeros of sieved
		ultraspherical polynomials}.
	\newblock {\em {J. Math. Phys.}}, 61(5):053501, 19, 2020.
	
	\bibitem{MR86898}
	T.~S. Chihara.
	\newblock On quasi-orthogonal polynomials.
	\newblock {\em Proc. Amer. Math. Soc.}, 8:765--767, 1957.
	
	\bibitem{chihara:1978}
	T.~S. Chihara.
	\newblock {\em An Introduction to Orthogonal Polynomials}.
	\newblock Gordon and Breach, New York, 1978.
	
	\bibitem{MR2538285}
	P.~A. Clarkson.
	\newblock Vortices and polynomials.
	\newblock {\em Stud. Appl. Math.}, 123(1):37--62, 2009.
	
	\bibitem{zbMATH05356762}
	E.~{Coussement}, J.~{Coussement}, and W.~{Van Assche}.
	\newblock {Asymptotic zero distribution for a class of multiple orthogonal
		polynomials}.
	\newblock {\em {Trans. Am. Math. Soc.}}, 360(10):5571--5588, 2008.
	
	\bibitem{MR2214212}
	J.~Coussement and W.~Van~Assche.
	\newblock Differential equations for multiple orthogonal polynomials with
	respect to classical weights: raising and lowering operators.
	\newblock {\em J. Phys. A}, 39(13):3311--3318, 2006.
	
	\bibitem{MR2000g:47048}
	P.~A. Deift.
	\newblock {\em Orthogonal polynomials and random matrices: a
		{R}iemann-{H}ilbert approach}.
	\newblock New York University Courant Institute of Mathematical Sciences, New
	York, 1999.
	
	\bibitem{MR2402236}
	M.~V. Demina and N.~A. Kudryashov.
	\newblock Special polynomials and rational solutions of the hierarchy of the
	second {P}ainlev\'{e} equation.
	\newblock {\em Teoret. Mat. Fiz.}, 153(1):58--67, 2007.
	
	\bibitem{MR2989511}
	M.~V. Demina and N.~A. Kudryashov.
	\newblock Point vortices and classical orthogonal polynomials.
	\newblock {\em Regul. Chaotic Dyn.}, 17(5):371--384, 2012.
	
	\bibitem{MR2924501}
	M.~V. Demina and N.~A. Kudryashov.
	\newblock Vortices and polynomials: non-uniqueness of the {A}dler-{M}oser
	polynomials for the {T}kachenko equation.
	\newblock {\em J. Phys. A}, 45(19):195205, 12, 2012.
	
	\bibitem{MR3508236}
	M.~V. Demina and N.~A. Kudryashov.
	\newblock Multi-particle dynamical systems and polynomials.
	\newblock {\em Regul. Chaotic Dyn.}, 21(3):351--366, 2016.
	
	\bibitem{MR2530570}
	M.~V. D\"{e}mina, N.~A. Kudryashov, and D.~I. Sinel\cprime shchikov.
	\newblock The polygonal method for constructing exact solutions of some
	nonlinear differential equations to describe waves on water.
	\newblock {\em Zh. Vychisl. Mat. Mat. Fiz.}, 48(12):2151--2162, 2008.
	
	\bibitem{MR4091604}
	D.~K. Dimitrov and B.~Shapiro.
	\newblock Electrostatic problems with a rational constraint and degenerate
	{L}am\'{e} equations.
	\newblock {\em Potential Anal.}, 52(4):645--659, 2020.
	
	\bibitem{Dimitrov:00}
	D.~K. Dimitrov and W.~Van~Assche.
	\newblock Lam\'e differential equations and electrostatics.
	\newblock {\em Proc. Amer. Math. Soc.}, 128(12):3621--3628, 2000.
	\newblock Erratum in Proc.\ Amer.\ Math.\ Soc.\ 131 (2003), no.\ 7, 2303.
	
	\bibitem{DrSt}
	K.~Driver and H.~Stahl.
	\newblock Normality in Nikishin systems.
	\newblock {\em Indag. Math.}, 5:161{\~n}--187, 1994.
	
	\bibitem{FiLo}
	U.~Fidalgo~Prieto and G.~L\'{o}pez~Lagomasino.
	\newblock Nikishin systems are perfect.
	\newblock {\em Constr. Approx.}, 34:297--356, 2011.
	
	\bibitem{MR3055365}
	G.~Filipuk, W.~Van~Assche, and L.~Zhang.
	\newblock Ladder operators and differential equations for multiple orthogonal
	polynomials.
	\newblock {\em J. Phys. A}, 46(20):205204, 24, 2013.
	
	\bibitem{MR826706}
	P.~J. Forrester and J.~B. Rogers.
	\newblock Electrostatics and the zeros of the classical polynomials.
	\newblock {\em SIAM J. Math. Anal.}, 17(2):461--468, 1986.
	
	\bibitem{Gakhov}
	F.~D. Gakhov.
	\newblock {\em Boundary value problems}.
	\newblock Dover Publications Inc., New York, 1990.
	\newblock Translated from the Russian, Reprint of the 1966 translation.
	
	\bibitem{GRS97}
	A.~A. Gonchar, E.~A. Rakhmanov, and V.~N. Sorokin.
	\newblock Hermite-{P}ad\'e approximants for systems of {M}arkov-type functions.
	\newblock {\em Sb.\ Math.}, 188:671--696, 1997.
	
	\bibitem{Grunbaum98}
	F.~A. Gr{\"u}nbaum.
	\newblock Variations on a theme {H}eine and {S}tieltjes: An electrostatic
	interpretation of the zeros of certain polynomials.
	\newblock {\em 99}, pages 189--194, 1998.
	
	\bibitem{Heine1878}
	E.~Heine.
	\newblock {\em Handbuch der Kugelfunctionen}, volume~II.
	\newblock G.\ Reimer, Berlin, $2$nd. edition, 1878.
	
	\bibitem{Ismail2000}
	M.~E.~H. Ismail.
	\newblock An electrostatic model for zeros of general orthogonal polynomials.
	\newblock {\em Pacific J. Math.}, 193:355--369, 2000.
	
	\bibitem{Ismail2000c}
	M.~E.~H. Ismail.
	\newblock More on electrostatic models for zeros of orthogonal polynomials.
	\newblock In {\em Proceedings of the International Conference on Fourier
		Analysis and Applications (Kuwait, 1998)}, volume~21, pages 191--204, 2000.
	
	\bibitem{Ismail2001}
	M.~E.~H. Ismail.
	\newblock Functional equations and electrostatic models for orthogonal
	polynomials.
	\newblock In {\em Random matrix models and their applications}, volume~40 of
	{\em Math. Sci. Res. Inst. Publ.}, pages 225--244. Cambridge Univ. Press,
	Cambridge, 2001.
	
	\bibitem{Ismail05}
	M.~E.~H. Ismail.
	\newblock {\em Classical and quantum orthogonal polynomials in one variable},
	volume~98 of {\em Encyclopedia of Mathematics and its Applications}.
	\newblock Cambridge University Press, Cambridge, 2005.
	\newblock With two chapters by Walter Van Assche, With a foreword by Richard A.
	Askey.
	
	\bibitem{MR3926161}
	M.~E.~H. Ismail and X.-S. Wang.
	\newblock On quasi-orthogonal polynomials: their differential equations,
	discriminants and electrostatics.
	\newblock {\em J. Math. Anal. Appl.}, 474(2):1178--1197, 2019.
	
	\bibitem{kaliaguine:1981}
	V.~A. Kaliaguine.
	\newblock On a class of polynomials defined by two orthogonality relations.
	\newblock {\em Math. USSR Sbornik}, 38(4):563--580, 1981.
	
	\bibitem{MR1391763}
	V.~A. Kaliaguine.
	\newblock On operators associated with {A}ngelesco systems.
	\newblock {\em East J. Approx.}, 1(2):157--170, 1995.
	
	\bibitem{kaliaguine:1996}
	V.~A. Kaliaguine and A.~Ronveaux.
	\newblock On a system of ``classical'' polynomials of simultaneous
	orthogonality.
	\newblock {\em J.\ Comput.\ Appl.\ Math.}, 67:207--217, 1996.
	
	\bibitem{MR2336164}
	N.~A. Kudryashov and M.~V. Demina.
	\newblock Relations between zeros of special polynomials associated with the
	{P}ainlev\'{e} equations.
	\newblock {\em Phys. Lett. A}, 368(3-4):227--234, 2007.
	
	\bibitem{ETNA05}
	A.~B.~J. Kuijlaars, A.~Mart{\'{\i}}nez-Finkelshtein, and R.~Orive.
	\newblock Orthogonality of {J}acobi polynomials with general parameters.
	\newblock {\em Electron. Trans. Numer. Anal.}, 19:1--17 (electronic), 2005.
	
	\bibitem{MR2470930}
	A.~B.~J. Kuijlaars, A.~Mart{\'{\i}}nez-Finkelshtein, and F.~Wielonsky.
	\newblock Non-intersecting squared {B}essel paths and multiple orthogonal
	polynomials for modified {B}essel weights.
	\newblock {\em Comm. Math. Phys.}, 286(1):217--275, 2009.
	
	\bibitem{MR4238535}
	G.~L\'{o}pez-Lagomasino.
	\newblock An introduction to multiple orthogonal polynomials and
	{H}ermite-{P}ad\'{e} approximation.
	\newblock In {\em Orthogonal polynomials: current trends and applications},
	volume~22 of {\em SEMA SIMAI Springer Ser.}, pages 237--271. Springer, Cham,
	2021.
	
	\bibitem{RHPNik}
	G.~L\'{o}pez~Lagomasino and W.~Van~Assche.
	\newblock The {R}iemann--{H}ilbert analysis for a {N}ikishin system.
	\newblock {\em Sbornik: Mathematics}, 209(7):1019--1050, 2018.
	
	\bibitem{Loutsenko:2004}
	I.~Loutsenko.
	\newblock Equilibrium of charges and differential equations solved by
	polynomials.
	\newblock {\em J. Phys. A}, 37(4):1309--1321, 2004.
	
	\bibitem{zbMATH05343896}
	V.~{Lysov} and F.~{Wielonsky}.
	\newblock {Strong asymptotics for multiple Laguerre polynomials}.
	\newblock {\em {Constr. Approx.}}, 28(1):61--111, 2008.
	
	\bibitem{zbMATH07308461}
	V.~G. {Lysov}.
	\newblock {Mixed type Hermite-Pad\'e approximants for a Nikishin system}.
	\newblock {\em {Proc. Steklov Inst. Math.}}, 311:199--213, 2020.
	
	\bibitem{Magnus}
	A.~P. Magnus.
	\newblock Painlev\'{e}--type differential equations for the recurrence
	coefficients of semi--classical orthogonal polynomials.
	\newblock {\em J. Comp. Appl. Math.}, 57:215--237, 1995.
	
	\bibitem{Mahler}
	K.~Mahler.
	\newblock Perfect systems.
	\newblock {\em Compositio Math.}, 19:95--166, 1968.
	
	\bibitem{MR2345246}
	F.~Marcell{{\'a}}n, A.~Mart{\'{\i}}nez-Finkelshtein, and
	P.~Mart{\'{\i}}nez-Gonz{{\'a}}lez.
	\newblock Electrostatic models for zeros of polynomials: old, new, and some
	open problems.
	\newblock {\em J. Comput. Appl. Math.}, 207(2):258--272, 2007.
	
	\bibitem{MR1340939}
	F.~Marcell{{\'a}}n and I.~A. Rocha.
	\newblock On semiclassical linear functionals: integral representations.
	\newblock In {\em Proceedings of the {F}ourth {I}nternational {S}ymposium on
		{O}rthogonal {P}olynomials and their {A}pplications ({E}vian-{L}es-{B}ains,
		1992)}, volume~57, pages 239--249, 1995.
	
	\bibitem{MR1637827}
	F.~Marcell{{\'a}}n and I.~A. Rocha.
	\newblock Complex path integral representation for semiclassical linear
	functionals.
	\newblock {\em J. Approx. Theory}, 94(1):107--127, 1998.
	
	\bibitem{MR2647571}
	A.~Mart{\'{\i}}nez-Finkelshtein and E.~A. Rakhmanov.
	\newblock On asymptotic behavior of {H}eine-{S}tieltjes and {V}an {V}leck
	polynomials.
	\newblock In {\em Recent trends in orthogonal polynomials and approximation
		theory}, volume 507 of {\em Contemp. Math.}, pages 209--232. Amer. Math.
	Soc., Providence, RI, 2010.
	
	\bibitem{MR2770010}
	A.~Mart{\'{\i}}nez-Finkelshtein and E.~A. Rakhmanov.
	\newblock Critical measures, quadratic differentials, and weak limits of zeros
	of {S}tieltjes polynomials.
	\newblock {\em Comm. Math. Phys.}, 302(1):53--111, 2011.
	
	\bibitem{MR2003j:33031}
	A.~Mart{\'{\i}}nez-Finkelshtein and E.~B. Saff.
	\newblock Asymptotic properties of {H}eine-{S}tieltjes and {V}an {V}leck
	polynomials.
	\newblock {\em J. Approx. Theory}, 118(1):131--151, 2002.
	
	\bibitem{MR3545949}
	A.~Mart\'{\i}nez-Finkelshtein and G.~L.~F. Silva.
	\newblock Critical measures for vector energy: global structure of trajectories
	of quadratic differentials.
	\newblock {\em Adv. Math.}, 302:1137--1232, 2016.
	
	\bibitem{MR3939592}
	A.~Mart\'{\i}nez-Finkelshtein and G.~L.~F. Silva.
	\newblock Critical measures for vector energy: asymptotics of non-diagonal
	multiple orthogonal polynomials for a cubic weight.
	\newblock {\em Adv. Math.}, 349:246--315, 2019.
	
	\bibitem{AMFS21}
	A.~Mart\'{\i}nez-Finkelshtein and G.~L.~F. Silva.
	\newblock Spectral curves, variational problems and the hermitian matrix model
	with external source.
	\newblock {\em Comm. Math. Phys.}, 383:2163--2242, 2021.
	
	\bibitem{MR2327035}
	E.~Mukhin and A.~Varchenko.
	\newblock Multiple orthogonal polynomials and a counterexample to the {G}audin
	{B}ethe ansatz conjecture.
	\newblock {\em Trans. Amer. Math. Soc.}, 359(11):5383--5418, 2007.
	
	\bibitem{MR3471160}
	T.~Neuschel and W.~Van~Assche.
	\newblock Asymptotic zero distribution of {J}acobi-{P}i\~{n}eiro and multiple
	{L}aguerre polynomials.
	\newblock {\em J. Approx. Theory}, 205:114--132, 2016.
	
	\bibitem{MR1831715}
	P.~K. Newton.
	\newblock {\em The {$N$}-vortex problem}, volume 145 of {\em Applied
		Mathematical Sciences}.
	\newblock Springer-Verlag, New York, 2001.
	\newblock Analytical techniques.
	
	\bibitem{Nik82}
	E.~M. Nikishin.
	\newblock On simultaneous {P}ad\'{e} approximants.
	\newblock {\em Math. USSR. Sb.}, 41:409--425, 1982.
	
	\bibitem{Nikishin86}
	E.~M. Nikishin.
	\newblock Asymptotic behavior of linear forms for simultaneous {P}ad\'e
	approximants.
	\newblock {\em Izv. Vyssh. Uchebn. Zaved. Mat.}, pages 33--41, 84, 1986.
	
	\bibitem{niksor}
	E.~M. Nikishin and V.~N. Sorokin.
	\newblock {\em Rational approximations and orthogonality}.
	\newblock American Mathematical Society, Providence, RI, 1991.
	
	\bibitem{MR2723248}
	F.~W.~J. Olver, D.~W. Lozier, R.~F. Boisvert, and C.~W. Clark, editors.
	\newblock {\em N{IST} handbook of mathematical functions}.
	\newblock U.S. Department of Commerce, National Institute of Standards and
	Technology, Washington, DC; Cambridge University Press, Cambridge, 2010.
	\newblock With 1 CD-ROM (Windows, Macintosh and UNIX).
	
	\bibitem{MR884516}
	L.~R. Pineiro Diaz.
	\newblock On simultaneous approximations for some collection of {M}arkov
	functions.
	\newblock \textit{Mosc. Univ. Math. Bull.} 42, No. 2, 52--55 (1987); translation from Vestn. Mosk. Univ., Ser. I 1987, No. 2, 67--70 (1987).
	
	\bibitem{MR2796829}
	E.~A. Rakhmanov.
	\newblock On the asymptotics of {H}ermite-{P}ad\'e polynomials for two {M}arkov
	functions.
	\newblock {\em Mat. Sb.}, 202(1):133--140, 2011.
	
	\bibitem{MR3058747}
	E.~A. Rakhmanov and S.~P. Suetin.
	\newblock Asymptotic behavior of {H}ermite-{P}ad\'e polynomials of the first
	kind for a pair of functions forming a {N}ikishin system.
	\newblock {\em Uspekhi Mat. Nauk}, 67(5(407)):177--178, 2012.
	
	\bibitem{MR3137137}
	E.~A. Rakhmanov and S.~P. Suetin.
	\newblock Distribution of zeros of {H}ermite-{P}ad\'e polynomials for a pair of
	functions forming a {N}ikishin system.
	\newblock {\em Mat. Sb.}, 204(9):115--160, 2013.
	
	\bibitem{Ronveaux95}
	A.~Ronveaux, editor.
	\newblock {\em Heun's differential equations}.
	\newblock The Clarendon Press Oxford University Press, New York, 1995.
	\newblock With contributions by F. M. Arscott, S. Yu.\ Slavyanov, D. Schmidt,
	G. Wolf, P. Maroni and A. Duval.
	
	\bibitem{MR2873079}
	P.~Rutka and R.~Smarzewski.
	\newblock Complete solution of the electrostatic equilibrium problem for
	classical weights.
	\newblock {\em Appl. Math. Comput.}, 218(10):6027--6037, 2012.
	
	\bibitem{MR2763943}
	B.~Shapiro.
	\newblock Algebro-geometric aspects of {H}eine-{S}tieltjes theory.
	\newblock {\em J. Lond. Math. Soc. (2)}, 83(1):36--56, 2011.
	
	\bibitem{Shohat}
	J.~A. Shohat.
	\newblock A differential equation for orthogonal polynomials.
	\newblock {\em Duke Math. J.}, 5:401--417, 1939.
	
	\bibitem{zbMATH06596279}
	B.~{Simanek}.
	\newblock {An electrostatic interpretation of the zeros of paraorthogonal
		polynomials on the unit circle}.
	\newblock {\em {SIAM J. Math. Anal.}}, 48(3):2250--2268, 2016.
	
	\bibitem{StTo}
	H.~Stahl and V.~Totik.
	\newblock {\em General orthogonal polynomials}, volume~43 of {\em Encyclopedia
		of Mathematics and its Applications}.
	\newblock Cambridge University Press, Cambridge, 1992.
	
	\bibitem{zbMATH06963441}
	S.~{Steinerberger}.
	\newblock {Electrostatic interpretation of zeros of orthogonal polynomials}.
	\newblock {\em {Proc. Am. Math. Soc.}}, 146(12):5323--5331, 2018.
	
	\bibitem{Stieltjes1885}
	T.~J. Stieltjes.
	\newblock Sur certains polyn\^omes que v\'erifient une \'equation
	diff\'erentielle lin\'eaire du second ordre et sur la teorie des fonctions de
	{L}am\'e.
	\newblock {\em Acta Math.}, 6:321--326, 1885.
	
	\bibitem{Stieltjes1885b}
	T.~J. Stieltjes.
	\newblock Sur les polyn\^omes de {Jacobi}.
	\newblock {\em {C. R. Acad. Sci., Paris}}, 100:439--440, 1885.
	
	\bibitem{MR1554668}
	T.~J. Stieltjes.
	\newblock Un th\'{e}or\`eme d'alg\`ebre.
	\newblock {\em Acta Math.}, 6(1):319--320, 1885.
	\newblock Extrait d'une lettre adress\'{e}e \`a M. Hermite.
	
	\bibitem{Szego75}
	G.~Szeg\H{o}.
	\newblock {\em Orthogonal polynomials}.
	\newblock Amer. Math. Soc., Colloq. Publ., 1975.
	
	\bibitem{MR1379147}
	G.~Valent and W.~Van~Assche.
	\newblock The impact of {S}tieltjes' work on continued fractions and orthogonal
	polynomials: additional material.
	\newblock In {\em Proceedings of the {I}nternational {C}onference on
		{O}rthogonality, {M}oment {P}roblems and {C}ontinued {F}ractions ({D}elft,
		1994)}, volume~65, pages 419--447, 1995.
	
	\bibitem{VanAssche06}
	W.~Van~Assche.
	\newblock Pad\'e and {H}ermite-{P}ad\'e approximation and orthogonality.
	\newblock {\em Surv. Approx. Theory}, 2:61--91, 2006.
	
	\bibitem{MR1808581}
	W.~Van~Assche and E.~Coussement.
	\newblock Some classical multiple orthogonal polynomials.
	\newblock volume 127, pages 317--347. 2001.
	\newblock Numerical analysis 2000, Vol. V, Quadrature and orthogonal
	polynomials.
	
	\bibitem{MR3304586}
	W.~Van~Assche, G.~Filipuk, and L.~Zhang.
	\newblock Multiple orthogonal polynomials associated with an exponential cubic
	weight.
	\newblock {\em J. Approx. Theory}, 190:1--25, 2015.
	
	\bibitem{MR3907776}
	W.~Van~Assche and A.~Vuerinckx.
	\newblock Multiple {H}ermite polynomials and simultaneous {G}aussian
	quadrature.
	\newblock {\em Electron. Trans. Numer. Anal.}, 50:182--198, 2018.
	
	\bibitem{zbMATH07122712}
	J.~F. {van Diejen}.
	\newblock {Gradient system for the roots of the Askey-Wilson polynomial}.
	\newblock {\em {Proc. Am. Math. Soc.}}, 147(12):5239--5249, 2019.
	
	\bibitem{zbMATH07326239}
	J.~F. {van Diejen}.
	\newblock {Stable equilibria for the roots of the symmetric continuous Hahn and
		Wilson polynomials}.
	\newblock In {\em Orthogonal polynomials: current trends and applications.
		Proceedings of the 7th EIBPOA conference, Universidad Carlos III de Madrid,
		Legan\'es, Spain, July 3--6, 2018}, pages 171--192. Cham: Springer, 2021.
	
	\bibitem{Vleck1898}
	E.~B. Van~Vleck.
	\newblock On the polynomials of {S}tieltjes.
	\newblock {\em Bull.\ Amer.\ Math.\ Soc.}, 4:426--438, 1898.
	
	\bibitem{Helmholtz1858}
	H.~von Helmholtz.
	\newblock {\"U}ber {Integrale} der hydrodynamischen {Gleichungen}, welche den
	{Wirbelbewegungen} entsprechen.
	\newblock {\em J. Reine Angew. Math.}, 55:25--55, 1858.
	\newblock {English translation by Tait, P.G., 1867. On integrals of the
		hydrodynamical equations, which express vortex-motion. Philos. Mag. 33(4),
		485--512.}
	
\end{thebibliography}

\def\cprime{$'$}

\end{document}